\input ./style/arxiv-general.cfg
\documentclass[MSNbibl,number,citesort,dvips]{arxbj}
\makeatletter
   \@ifpackageloaded{graphicx}{}{\usepackage{graphicx}}
\makeatother
\usepackage{mathbh,dcolumn}
\usepackage{algcompatible}
\usepackage{algorithm}
\usepackage{longtable}
\usepackage{accents}%
\usepackage{mathrsfs}
\algnewcommand\RETURN{\STATE\textbf{Return: }}

\volume{22}
\issue{3}
\pubyear{2016}
\firstpage{1617}
\lastpage{1670}
\doi{10.3150/15-BEJ706}
\docsubty{FLA}

\makeatletter
\newcolumntype{d}[1]{D{.}{.}{#1}}
\newcommand{\rrvert}{\vert}
\newcommand{\rrVert}{\Vert}
\newcommand{\llvert}{\vert}
\newcommand{\llVert}{\Vert}
\renewcommand{\mathring}[1]{\accentset{\circ}{#1}}
\newtheorem{thmm}{Theorem}[section]
\newtheorem{prop}{Proposition}[section]
\newtheorem{lemme}{Lemma}[section]
\newtheorem{lemmee}{Lemma}[section]
\newproclaim{definition}{Definition}[section]

\newproclaim{hyp}{Assumption}[section]

\newtheorem{Claim}{Claim}[section]
\newtheorem{Claimm}{Claim}[section]
\newremark{exemple}{Example}
\newremark{remark}{Remark}

\newremark{ExempleSimuDimmm}{Example}
\newremark{ExempleSimuDimd}{Example}

\renewcommand{\d}{\mathrm{d}}
\newcommand{\R}{\mathbb{R}}
\renewcommand{\P}{\mathbb{P}}
\newcommand{\N}{\mathbb{N}}
\newcommand{\B}{\mathcal{B}}
\newcommand{\E}{\mathbb{E}}
\newcommand{\A}{\mathcal{A}}
\newcommand{\F}{\mathscr{F}}
\renewcommand{\L}{\mathbb{L}}
\newcommand{\M}{\mathcal{M}}
\newcommand{\XX}{\mathbb{X}}
\newcommand{\tttheta}{\hat{\bolds{\theta}}}
\newcommand{\1}{\mathbh{1}}

\newcommand{\ttheta}{\bolds{\theta}}
\newcommand{\Fd}{\mathscr{F}_{\mathrm{dis}}}
\makeatother

\begin{document}
\begin{frontmatter}

\title{Robust estimation on a parametric model via~testing}
\runtitle{Robust estimation on a parametric model via testing}

\begin{aug}
\author[A]{\inits{M.}\fnms{Mathieu}~\snm{Sart}\corref{}\ead[label=e1]{mathieu.sart@univ-st-etienne.fr}}
\address[A]{Universit\'{e} de Lyon,
Universit\'{e} Jean Monnet,
CNRS UMR 5208 and
Institut Camille Jordan,
Maison de l'Universit\'{e}, 10 rue Tr\'{e}filerie,
CS 82301,
42023 Saint-Etienne Cedex 2,
France.\\ \printead{e1}}
\end{aug}

%
\received{\smonth{9} \syear{2013}}
%
\revised{\smonth{1} \syear{2015}}


\begin{abstract}
We are interested in the problem of robust parametric estimation of a
density from $n$ i.i.d. observations. By using a practice-oriented
procedure based on robust tests, we build an estimator for which we
establish non-asymptotic risk bounds with respect to the Hellinger
distance under mild assumptions on the parametric model. We show that
the estimator is robust even for models for which the maximum
likelihood method is bound to fail. A numerical simulation illustrates
its robustness properties. When the model is true and regular enough,
we prove that the estimator is very close to the maximum likelihood
one, at least when the number of observations $n$ is large. In
particular, it inherits its efficiency. Simulations show that these two
estimators are almost equal with large probability, even for small
values of~$n$ when the model is regular enough and contains the true density.
\end{abstract}

%
\end{frontmatter}

\section{Introduction}\label{sec1}
We consider $n$ independent and identically distributed random
variables $X_1,\ldots,X_n$ defined on an abstract probability space
($\Omega, \mathcal{E},\P)$ with values in the measure space $(\XX
,\mathcal{F},\mu)$. We suppose that the distribution of $X_i$ admits a
density $s$ with respect to $\mu$ and aim at estimating $s$ by using a
parametric approach.

\subsection{About the maximum likelihood estimator}
The maximum likelihood method is one of the most widespread estimation
methods to deal with this statistical setting.
Indeed, it is well known that it provides estimators with nice
statistical properties when the parametric model is true and regular enough.

Nevertheless, it is also recognized that it 
breaks down for many parametric models $\F$ of interest. A simple one
is the translation model $\F= \{f (\cdot- \theta),   \theta\in
\Theta\}$ where $\lim_{x \rightarrow0} f(x) = +\infty$, in which the
maximum likelihood estimator (m.l.e. for short) does not exist. Other
counterexamples may be found
in Pitman \cite{Pitman1979}, Ferguson \cite{Ferguson1982}, Le Cam
\cite
{Lecam1990mle}, Birg{\'e} \cite{BirgeTEstimateurs} among other references.

Another known defect of the m.l.e. is its lack of robustness. This
means that if the assumption that $s$ belongs to the parametric model
$\F$ is only slightly violated, the m.l.e. may perform poorly. As an
example, consider the model $\F= \{\theta^{-1} \1_{[0,\theta]},
\theta> 0\}$, in which the maximum likelihood estimator is $\hat
{\theta
}_{\mathrm{mle}}^{-1} \1_{[0,\hat{\theta}_{\mathrm{mle}}]}$ with $\hat
{\theta}_{\mathrm{mle}} = \max_{1 \leq i \leq n} X_i$. Suppose that the
true density $s$ does not belong to $\F$ but lies in a very small
neighbourhood of it. For instance, assume that $s = (1-p) \1_{[0,1]} +
p 2^{-1} \1_{[0,2]}$ for some $p \in(0,1)$. If $p$ is very small, the
true underlying density $s$ is very close to $\1_{[0,1]} \in\F$ and a
good estimator $\hat{f}$ of $s$ should therefore be also close to $\1
_{[0,1]}$, at least when $n$ is large enough and $p$ is small enough.
Nonetheless, whatever $p > 0$, the estimator $\hat{\theta}_{\mathrm
{mle}}^{-1} \1_{[0, \hat{\theta}_{\mathrm{mle}}]}$ converges almost
surely to $2^{-1} \1_{[0,2]}$ when $n$ goes to infinity.\vspace*{1pt} It is thus a
very poor estimate of $s$ when $p$ is small.

\subsection{Alternative estimators}
Several attempts have been made in the literature to overcome the
difficulties of the maximum likelihood approach.
When the model is regular enough, the classical notion of efficiency
can be used to measure the quality of an estimator (when the model is
not regular enough, the optimal rate of convergence may not be the
usual root-$n$ rate). For these models, the $L$-estimators commonly
accomplish a good trade-off between robustness and efficiency. Some
estimators have the nice feature to be simultaneously robust and
asymptotically efficient. This is the case, for example, of the minimum
Hellinger distance
estimators introduced by Beran \cite{Beran1977} and studied in Donoho
and Liu \cite{Donoho1988}, Lindsay \cite{Lindsay1994} among other
references. We refer to Basu et al. \cite{basu2011} for an introduction
to these estimators.

Things become more complicated when the model is less regular and even
more when the maximum likelihood estimators do not even exist. We do
not know if the aforementioned estimation strategies can be adapted to
cope with these models in a satisfactory way.
Building a robust and optimal estimator is not straightforward in some
models (where ``optimal'' means that it achieves the optimal rate of
convergence when the model holds true). Think, for instance, about the
translation model $\F=  \{f (\cdot- \theta),   \theta\in[-1,
1] \} $ where
%
\begin{equation}
\label{eqModelIrregulier}
f (x) = \cases{ \displaystyle\frac{1}{ 4 \sqrt{|x|}} \1_{[-1,1]} (x),&
\quad$\mbox{for all $x \in\R \setminus\{0\}$,}$ \vspace*{2pt}
\cr
0, &\quad $\mbox{for $x =0$.}$}
\end{equation}
The median is a natural robust estimator, but it converges slowly to
the right parameter since it only reaches the rate $n^{-1}$ whereas the
optimal one is $n^{-2}$.

\subsection{Estimation via testing}
There is in the literature a more or less universal strategy of
estimation that leads to robust and optimal estimators. It even manages
to deal with models for which the maximum likelihood method is bound to
fail. Its basic principle is to use tests to derive estimators.
Historically, this idea of using tests for building estimators dates
back to the 1970s with the works of Lucien Le Cam.
More recently, Birg{\'e} \cite{BirgeTEstimateurs} significantly
extended the scope of these procedures by relating them to the problem
of model selection, providing at the same time new perspectives on
estimation theory.
It gave birth to a series of papers; see Birg{\'e} \cite
{BirgePoisson,Birge2012,BirgeDens},
Baraud and Birg\'{e} \cite{BaraudBirgeHistogramme},
Baraud \cite{BaraudMesure,Baraud2012},
Sart \cite{SartMarkov,Sart2012},
Baraud et al. \cite{RhoEstimation}.
The main feature of these procedures is that they allow to obtain
general theoretical results in various statistical settings (such as
general model selection theorems) which are usually unattainable by the
traditional procedures (such as those based on the minimization of a
penalized contrast).

In density estimation, these papers show that under very mild
assumptions on the parametric model $\F= \{f_{\theta}, \theta\in
\Theta\}$, one can design an estimator $\hat{s} = f_{\hat{\theta}}$
such that
%
\begin{equation}
\label{RelIntro} \P \biggl[C h^2 (s, f_{\hat{\theta}} ) \geq\inf
_{\theta\in\Theta} h^2 (s, f_{\theta}) +
\frac{D_{\F}}{n} + \xi \biggr] \leq \mathrm{e}^{-n \xi} \qquad\mbox{for all $\xi> 0$,}
\end{equation}
where $C$ is a numerical positive constant, $h$ the Hellinger distance,
and $D_{\F}$ measures, in some sense, the ``massiveness'' of $\F$. We
recall that the Hellinger distance is defined on the cone $\L^1_+ (\XX,
\mu)$ of non-negative integrable functions on $\XX$ with respect to
$\mu
$ by
\[
h^2(f,g) = \frac{1}{2} \int_{\XX} \bigl(
\sqrt{f (x)} - \sqrt{g(x)} \bigr)^2 \,\d\mu(x)\qquad \mbox{for all $f,g \in
\L^1_+ (\XX, \mu)$.}
\]

When $s$ does belong to the model $\F$, that is, when there exists
$\theta
_0 \in\Theta$ such that $s = f_{\theta_0}$, the estimator $\hat{s}$
achieves a quadratic risk of order $n^{-1}$ with respect to the
Hellinger distance. Besides, if we can relate the Hellinger distance
$h(f_{\theta_0}, f_{\theta})$ to a distance between the parameters
$\theta_0$, $\theta$, the convergence rate of $\hat{\theta}$
to $\theta
_0$ may be deduced from {(\ref{RelIntro})}. For instance, when $\Theta
\subset\R$, and when there exists $\alpha> 0$ such that $h^2
(f_{\theta_0}, f_{\theta}) \sim|\theta_0 - \theta|^{\alpha}$, the
estimator $\hat{\theta}$ reaches the rate $n^{-1/\alpha}$. When the
model is regular enough, $h^2 (f_{\theta_0}, f_{\theta}) \sim|\theta_0
- \theta|^{2}$, and the estimator $\hat{\theta}$ attains the usual
root-$n$ rate.

It is worth mentioning that one does not have to assume that the
unknown density $s$ belongs to the model, which is important since one
cannot usually ensure that this is the case in practice. We rather use
the model $\F$ as an approximating class (sieve) for $s$. Inequality
(\ref{RelIntro}) shows that the estimator $\hat{s} = f_{\hat{\theta}}$
cannot be strongly influenced by any type of small departures from the
model (measured through the Hellinger distance). As a matter of fact,
if $\inf_{\theta\in\Theta} h^2 (s, f_{\theta}) \leq a n^{-1}$ with $a
> 0$, which means that the model is slightly misspecified, the
quadratic risk of the estimator $\hat{s} = f_{\hat{\theta}}$ remains of
order $n^{-1}$. This can be interpreted as a robustness property (that
is, not shared by the m.l.e.).

\subsection{The purposes of this paper}
One of the most annoying drawbacks of the estimators based on tests is
that their practical construction is numerically very difficult. Two
steps are required to build these estimators. In the first step, we
discretize the model $\F$, that is, we build a thin net $\Fd$ in $\F$
that must be finite or countable. In the second step, we use the tests
to pairwise compare the elements of $\Fd$. Therefore, the number of
tests we need to compute is of the order of the square of the
cardinality of $\Fd$. Unfortunately, this cardinality is often very
large, making the construction of the estimators difficult in practice.
In this paper, we present a new estimation procedure based on the test
designed by Baraud \cite{BaraudMesure} and on an iterative construction
of confidence sets. This procedure does not involve the pairwise
comparison of all the elements of $\Fd$ but only of a small (random and
suitably chosen) part of them, which results in a significant reduction
of the numerical complexity. In particular, this makes it possible to
evaluate the quality of the estimator by means of numerical simulations
in situations where the procedure
of Baraud \cite{BaraudMesure} would have required the computation of an
intractable number of tests.

This estimation procedure outperforms the maximum likelihood one in
many aspects.
Similarly to the procedure of Baraud \cite{BaraudMesure}, the estimator
$\hat{s} = f_{\hat{\theta}}$ exists in parametric models where the
m.l.e. does not. We establish a risk bound akin to {(\ref{RelIntro})}.
In particular, when the model $\F$ is true, that is, when there exists
$\theta_0 \in\Theta$ such that $s = f_{\theta_0} \in\F$, the
estimator $\hat{\theta}$ converges to the true parameter $\theta_0$ at
the right rate of convergence. When the model is only approximately
true, which means that the Hellinger distance between $s$ and the model
$\F$ is small, the estimator $\hat{s}$ of $s$ still performs well.

An additional significant property of this estimator is that it
essentially coincides with the m.l.e. (with large probability), when
the model is true and regular enough, even when the number of
observations $n$ is small. It seems to be, in this case, as good as the
m.l.e. This property was brought to light by numerical simulations in
the first draft of this paper. During the revision process, an
asymptotic theoretical connection between an estimator based on tests
and the m.l.e. was established, for the first time, in Theorem~4
of Baraud et al. \cite{RhoEstimation}.
The techniques developed in this paper helped us to prove theoretically
that our estimator was asymptotically very close to the m.l.e. and that
it inherited in particular its nice asymptotic properties such as
efficiency (at least under suitable regularity assumptions on the model
$\F$). These regularity assumptions are however different from theirs.
They may therefore hold true in some parametric models where
those of Baraud et al. \cite{RhoEstimation} do not.

\subsection{Organization of the paper and notations}
For the sake of clarity, we start by considering models parametrized by
a one-dimensional parameter. In Section~\ref{SectionEstimationDim1}, we
present our procedure and the associated theoretical results. We
evaluate its performance in practice by carrying out numerical
simulations in the next section. We study the multi-dimensional case in
Section~\ref{SectionEstimationDimQuel}. 
We postpone the main proofs to Section~\ref{SectionProofs} except the
one of Theorem~\ref{54515215415152151514} which is quite
technical and deferred to the \hyperref[SectionAnnexe1]{Appendix}.

We now introduce some notation that will be used all along the paper.
The number $x \vee y$ stands for $\max(x,y)$ and $x_+$ stands for $x
\vee0$. We set $\N^{\star} = \N\setminus\{0\}$. The vector
$(\theta
_1,\ldots,\theta_d)$ of $\R^d$ is denoted by the bold letter $\ttheta$.
We write indifferently $h(f_{\ttheta}, f_{\ttheta'})$ or $h (\ttheta,
\ttheta')$. The cardinality of a finite set $A$ is denoted by $|A|$.
For $(E,d)$ a metric space, $x \in E$ and $A \subset E$, the distance
between $x$ and $A$ is denoted by $d(x,A)= \inf_{a \in A} d(x,a)$. The
indicator function of a subset $A$ is denoted by $\1_A$.
The notation $C$, $C'$, $C''$ stand for quantities independent of $n$.
When they depend on other parameters, this dependency will be specified
in the text. The values of $C$, $C'$, $C'', \ldots$ may change from
line to line.

\section{Models parametrized by a one-dimensional parameter} \label
{SectionEstimationDim1}

\subsection{Assumption on the model}
We start by considering sets of densities $\F=  \{f_{\theta},
\theta\in\Theta \}$ indexed by a finite interval $\Theta=
[m,M]$ of $\R$. Such a set will be called a one-dimensional model.
Throughout this section, the models are assumed to satisfy the
following property.

\begin{hyp} \label{HypSurLeModeleQuelquonqueDebutDim1}
There exist positive numbers $\alpha$, $\underline{R}$, $\overline{R}$
such that for all $\theta, \theta' \in[m,M]$,
\[
\underline{R} \bigl|\theta- \theta'\bigr|^{\alpha} \leq
h^2 \bigl({\theta}, {\theta'} \bigr) \leq\overline{R} \bigl|
\theta- \theta'\bigr|^{\alpha},
\]
where $h (\theta, \theta'  )$ stands for the Hellinger
distance $h  (f_{\theta}, f_{\theta'}  )$ between the two
densities $f_{\theta}$ and $f_{\theta'}$.
\end{hyp}

This assumption allows to connect a (quasi) distance between the
parameters to the Hellinger one between the corresponding densities.
A similar assumption may be found in Theorem~5.8 of Chapter~1 of
Ibragimov and Has'minskii \cite{Ibragimov1981} to prove results on the
maximum likelihood estimator. They require, however, the application
$\theta\mapsto f_{\theta} (x)$ to be continuous for $\mu$-almost
all $x$ to ensure the existence and the consistency of the m.l.e.
Without this additional assumption, the m.l.e. may not exist as shown
by the translation model $\F=  \{f (\cdot- \theta),   \theta
\in
[-1, 1] \} $ where $f$ is defined in the \hyperref[sec1]{Introduction} by (\ref
{eqModelIrregulier}) (note that Assumption~\ref
{HypSurLeModeleQuelquonqueDebutDim1} holds for this model with $\alpha
= 1/2$).

Under suitable regularity conditions on the model, Theorem~7.6 of
Chapter~1 of Ibragimov and Has'minskii \cite{Ibragimov1981} shows that
this assumption is fulfilled with $\alpha= 2$. Other kinds of
sufficient conditions implying Assumption~\ref
{HypSurLeModeleQuelquonqueDebutDim1} may be found in this book (see the
beginning of Chapter~5 and Theorem~1.1 of Chapter~6). Other examples
and counterexamples are given in Chapter~7
of Dacunha-Castelle \cite{DacunhaCastelle}.
Several models of interest satisfying this assumption will appear later
in the paper.

\subsection{Basic ideas} \label{SectionHeuristique} We now present the
heuristic on which our estimation procedure is based.
We assume in this section that $s$ belongs to the model $\F$, that is,
there exists $\theta_0 \in\Theta= [m, M]$ such that $s = f_{\theta_0}$.
The starting point is the existence for all $\theta, \theta' \in
\Theta
$ of a measurable function $T(\theta, \theta')$ of the observations
$X_1,\ldots,X_n$ such that:
\begin{longlist} [1.]
\item[1.] For all $\theta, \theta' \in\Theta$, $T(\theta, \theta') =
- T
(\theta', \theta) $.
\item[2.]
There exists $\kappa> 0$ such that if
$\E [T (\theta,\theta')  ]$ is non-negative, then
$ h^2 (\theta_0, {\theta}) > {\kappa} h^2 ({\theta},{\theta'})$.
\item[3.]
For all $\theta,\theta' \in\Theta$,
$T(\theta, \theta') $ and $\E [T (\theta, \theta')  ]$ are
close (in a suitable sense).
\end{longlist}
For all $\theta\in\Theta$, $r > 0$, let $\B(\theta, r)$ be the
Hellinger ball centered at $\theta$ with radius $r$, that is,
%
\begin{equation}
\label{eqDefinitionBouleHel} \B(\theta, r) = \bigl\{\theta' \in\Theta, h \bigl({
\theta}, {\theta'}\bigr) \leq r \bigr\}.
\end{equation}
For all $\theta, \theta' \in\Theta$, we deduce from the first point
that either $T ({\theta}, {\theta'})$ is non-negative, or $T ({\theta
'}, {\theta})$ is non-negative. It is likely that it follows from 2 and 3
that in the first case
\[
\theta_0 \in\Theta\setminus\B \bigl(\theta, \kappa^{1/2}
h \bigl({\theta}, {\theta'}\bigr) \bigr)
\]
while in the second case
\[
\theta_0 \in\Theta\setminus\B \bigl(\theta',
\kappa^{1/2} h \bigl({\theta }, {\theta'}\bigr) \bigr).
\]
These sets may be interpreted as confidence sets for $\theta_0$.

The main idea is to build a decreasing sequence (in the sense of
inclusion) of intervals $(\Theta_i)_i$. Set $\theta^{(1)} = m$,
$\theta
'^{(1)} = M$, and $\Theta_1 = [\theta^{(1)}, \theta'^{(1)}]$ (which is
merely $\Theta$). If $T ({\theta^{(1)}}, {\theta'^{(1)}} )$ is
non-negative, we consider a set $\Theta_2$ such that
\[
\Theta_1 \setminus\B \bigl(\theta^{(1)},
\kappa^{1/2} h \bigl({\theta ^{(1)}}, {\theta'^{(1)}}
\bigr) \bigr) \subset\Theta_2 \subset\Theta_1
\]
while if $T ({\theta^{(1)}}, {\theta'^{(1)}} )$ is non-positive, we
consider a set $\Theta_2$ such that
\[
\Theta_1 \setminus\B \bigl(\theta'^{(1)},
\kappa^{1/2} h \bigl({\theta ^{(1)}}, {\theta'^{(1)}}
\bigr) \bigr) \subset\Theta_2 \subset\Theta_1.
\]
The set $\Theta_2$ may thus also be interpreted as a confidence set for
$\theta_0$. Thanks to Assumption~\ref
{HypSurLeModeleQuelquonqueDebutDim1}, we can define $\Theta_2$ as an
interval $\Theta_2 = [\theta^{(2)}, \theta'^{(2)}]$.

We then repeat the construction to build an interval $\Theta_3 =
[\theta
^{(3)}, \theta'^{(3)}]$ included in $\Theta_2$ such that either
\[
\Theta_3 \supset\Theta_2 \setminus\B \bigl(
\theta^{(2)}, \kappa^{1/2} h \bigl({\theta^{(2)}}, {
\theta'^{(2)}}\bigr) \bigr)\quad \mbox{or}\quad \Theta _3
\supset\Theta_2 \setminus\B \bigl(\theta'^{(2)},
\kappa^{1/2} h \bigl({\theta^{(2)}}, {\theta'^{(2)}}
\bigr) \bigr)
\]
according to the sign of $T ({\theta^{(2)}}, {\theta'^{(2)}} )$.

By induction, we build a decreasing sequence of such intervals $(\Theta
_i)_i$. We now consider an integer $N$ large enough so that the length
of $\Theta_N$ is small enough. We then define the estimator $\hat
{\theta
}$ as the center of the set $\Theta_N$ and estimate $s$ by $f_{{\hat
{\theta}}}$.

\subsection{Definition of the test} \label{SectionDefTest}
The test $T(\theta,\theta')$ we use in our estimation strategy is the
one of Baraud \cite{BaraudMesure} applied to two suitable densities of
the model. More precisely, let $\overline{T}$ be the functional defined
for all $g,g' \in\L^1_+ (\XX,\mu)$ by
%
\begin{eqnarray}
\label{eqFonctionnalBaraud}
\hspace*{-10pt} \overline{T}\bigl(g, g'\bigr) = \frac{1}{n}
\sum_{i=1}^n \frac{\sqrt{g'
(X_i)} - \sqrt{g (X_i)}}{\sqrt{g (X_i) + g'(X_i)}} +
\frac{1}{2 } \int_{\XX} \sqrt{g(x) + g'(x)}
\bigl(\sqrt{g' (x)} - \sqrt{g (x)} \bigr) \,\d\mu(x),
\end{eqnarray}
where the convention $0 / 0 = 0$ is in use.

We consider $t \in(0,1]$ and $\varepsilon= t (\overline{R}
n)^{-1/\alpha}$. We then define the finite sets
\begin{eqnarray*}
\Theta_{\mathrm{dis}} = \bigl\{m+ k \varepsilon, k \in\N, k \leq (M-m)
\varepsilon^{-1} \bigr\},\qquad \Fd= \{f_{\theta}, \theta \in
\Theta_{\mathrm{dis}} \}
\end{eqnarray*}
and the map $\pi$ on $[m, M]$ by
\[
\pi({x}) = m + \bigl\lfloor(x - m) / \varepsilon\bigr\rfloor\varepsilon
\qquad\mbox{for all $x \in[m, M]$},
\]
where $\lfloor\cdot\rfloor$ denotes the integer part. The test
$T(\theta,\theta')$ is finally defined by
\[
{T} \bigl({\theta},{\theta'}\bigr) = \overline{T}(f_{\pi(\theta)},f_{\pi
(\theta
')})\qquad
\mbox{for all $\theta,\theta' \in[m, M]$}.
\]
The aim of the parameter $t$ is to tune the thinness of the net $\Fd$.
The smaller $t$, the thinner $\Fd$.

\subsection{Estimation procedure} \label
{SectionEstimationProcedureDim1} We shall build a decreasing sequence
$(\Theta_i)_{i \geq1}$ of intervals of $\Theta= [m,M]$ as explained
in Section~\ref{SectionHeuristique}.
Let $\kappa> 0$, and for all $\theta, \theta' \in[m,M]$ such that
$\theta' > \theta$, let $\overline{r} (\theta, \theta')$,
$\underline
{r} (\theta, \theta')$ be two positive numbers satisfying
%
\begin{eqnarray}
[m,M] \cap \bigl[ \theta, \theta+\overline{r} \bigl(\theta,\theta'
\bigr) \bigr] &\subset& \B \bigl(\theta, \kappa^{1/2} h \bigl({\theta}, {
\theta'}\bigr) \bigr) ,\label{EquationSurRi1}
\\
{} [m,M] \cap \bigl[ \theta'-\underline{r} \bigl(\theta,
\theta'\bigr), \theta ' \bigr] &\subset& \B \bigl(
\theta', \kappa^{1/2} h \bigl({\theta}, {\theta
'}\bigr) \bigr), \label{EquationSurRi2}
\end{eqnarray}
where we recall that $ \B(\theta, {\kappa}^{1/2} h ({\theta},
{\theta
'}) )$ and $ \B(\theta', {\kappa}^{1/2} h ({\theta}, {\theta'}) )$ are
the Hellinger balls defined by~{(\ref{eqDefinitionBouleHel})}.

We set $\theta^{(1)} = m$, $\theta'^{(1)} = M$ and $\Theta_1 =
[\theta
^{(1)},\theta'^{(1)}]$. We define the sequence $(\Theta_i)_{i \geq1}$
by induction. When $\Theta_i = [\theta^{(i)}, \theta'^{(i)}]$, we set
\begin{eqnarray*}
\theta^{(i+1)} &=& \cases{\displaystyle \theta^{(i)} + \min \biggl\{
\overline{r} \bigl(\theta^{(i)},\theta'^{(i)}
\bigr), \frac{\theta'^{(i)} - \theta^{(i)}}{2} \biggr\}, & \quad
$\mbox{if $ T \bigl({\theta
^{(i)}},{\theta'^{(i)}}\bigr) \geq0$},$
\vspace*{2pt}
\cr
\theta^{(i)}, &\quad $\mbox{otherwise}$}
\\
\theta'^{(i+1)} &=& \cases{\displaystyle \theta'^{(i)}
- \min \biggl\{\underline{r} \bigl(\theta^{(i)},\theta
'^{(i)}\bigr), \frac{\theta'^{(i)} - \theta^{(i)}}{2} \biggr\}, &\quad
 $\mbox {if $ T
\bigl({\theta^{(i)}},{\theta'^{(i)}}\bigr)
\leq0$},$ \vspace*{2pt}
\cr
\theta'^{(i)}, &\quad $\mbox{otherwise.}$}
\end{eqnarray*}
We then define $\Theta_{i+1} = [\theta^{(i+1)}, \theta'^{(i+1)}]$.

The role of conditions (\ref{EquationSurRi1}) and (\ref
{EquationSurRi2}) is to ensure that $\Theta_{i+1}$ is big enough to
contain one of the two confidence sets
\[
\Theta_i \setminus\B \bigl(\theta^{(i)}, {
\kappa}^{1/2} h \bigl({\theta ^{(i)}}, {\theta'^{(i)}}
\bigr) \bigr) \quad\mbox{and}\quad \Theta_i \setminus\B \bigl(
\theta'^{(i)}, {\kappa}^{1/2} h \bigl({
\theta^{(i)}}, {\theta'^{(i)}}\bigr) \bigr).
\]
The parameter $\kappa$ allows to tune the level of these confidence
sets. There is a minimum in the definitions of $\theta^{(i+1)}$ and
$\theta'^{(i+1)}$ in order to guarantee the inclusion of $\Theta_{i+1}$
in $\Theta_i$.

We now consider a positive number $\eta$ and build these intervals
until their lengths become smaller than $\eta$.
The estimator is then defined as the center of the last interval. This
parameter $\eta$ stands for a measure of the accuracy of the estimation
and must be small enough to get a suitable risk bound for the estimator.
The algorithm is therefore the following.
\begin{algorithm}[H]
\caption{}
\label{AlgorithmDim1}
\begin{algorithmic} [1]
\STATE$\theta\leftarrow m$, $\theta' \leftarrow M$
\WHILE{$\theta' - \theta> \eta$}
\STATE Compute $r = \min \{\overline{r} (\theta,\theta'),
(\theta
' - \theta)/2 \}$
\STATE Compute $r' = \min \{\underline{r} (\theta,\theta'),
(\theta
' - \theta)/2  \}$
\STATE Compute $\mathrm{Test} = T(\theta,\theta')$
\IF{$\mathrm{Test} \geq0$}
\STATE$\theta\leftarrow\theta+ r$
\ENDIF
\IF{$\mathrm{Test} \leq0$}
\STATE$\theta' \leftarrow\theta' - r'$
\ENDIF
\ENDWHILE
\RETURN$\hat{\theta} = (\theta+ \theta')/2$
\end{algorithmic}
\end{algorithm}

The convergence of the algorithm is guaranteed under very mild
conditions on $\overline{r} (\theta, \theta') $ and $\underline{r}
(\theta, \theta') $. For instance, a sufficient condition is that the
functions $\overline{r} (\cdot,\cdot) $, $\underline{r} (\cdot
,\cdot) $
are positive and continuous on the set $\{(\theta, \theta'),   m
\leq
\theta< \theta' \leq M \}$.
Moreover, its numerical complexity can be bounded as soon as $\overline
{r} (\theta, \theta') $ and $\underline{r} (\theta, \theta') $ are
large enough as we shall see in Section~\ref{SectionDefinitionRminBarreDim1}.

\subsection{A non-asymptotic risk bound} \label{SectionPropEstimateurDim1} The following theorem specifies the values
of the parameters $t$, $\kappa$, $\eta$ that allow to control the risk
of the estimator $\hat{s} = f_{\hat{\theta}}$.

\begin{thmm} \label{ThmPrincipalDim1}
Suppose that Assumption~\ref{HypSurLeModeleQuelquonqueDebutDim1}
holds. Set
%
\begin{equation}
\label{eqEsperanceTest} \bar{\kappa} = 3/2 - \sqrt{2}.
\end{equation}
Assume that $t \in(0,1]$, $\kappa\in(0,\bar{\kappa})$, $\eta\in(0,
(\overline{R} n)^{-1/\alpha} ]$
and that $\overline{r} (\theta,\theta')$, $\underline{r} (\theta
,\theta
')$ are such that (\ref{EquationSurRi1}) and (\ref{EquationSurRi2})
hold and that the algorithm converges.

Then, for all $\xi> 0$, the estimator $\hat{\theta}$ derived from
Algorithm \ref{AlgorithmDim1}
satisfies
\[
\P \biggl[ C h^2(s,f_{\hat{\theta}}) \geq h^2(s, \F) +
\frac{D_{\F}}{n} + \xi \biggr] \leq \mathrm{e}^{- n \xi},
\]
where $D_{\F} = 1 \vee\log (1 + t^{-1} ( (1/ {{\alpha}}) (c
\overline{R} / \underline{R}) )^{1/\alpha} )$
with $c$ depending only on $\kappa$, and where $C > 0$ depends only on
$\kappa$ and $\overline{R}/\underline{R}$. Besides, if
\[
h^2\bigl({\theta_2}, {\theta_2'}
\bigr) \leq h^2\bigl({\theta_1}, {\theta_1'}
\bigr)\qquad \mbox{for all $m \leq\theta_1 \leq\theta_2 <
\theta_2' \leq \theta_1' \leq
M$}
\]
then $C$ depends only on $\kappa$.
\end{thmm}

We deduce from this risk bound that if $s = f_{\theta_0}$ belongs to
the model $\F$, the estimator ${\hat{\theta}}$ converges almost surely
to $\theta_0$. Besides, we may then derive from Assumption~\ref
{HypSurLeModeleQuelquonqueDebutDim1} that there exist positive numbers
$a$, $b$ such that
\[
\P \bigl[n^{1/\alpha} |\hat{\theta} - \theta_{0} | \geq \xi
\bigr] \leq a \mathrm{e}^{-b \xi^{\alpha}} \qquad\mbox{for all $\xi> 0$.}
\]
We emphasize here that this exponential inequality on $\hat{\theta}$ is
non-asymptotic but that the numbers $a$ and $b$ are, unfortunately, far
from optimal (since their values depend on several parameters involved
in the algorithm such as $t$ or $\kappa$).
As explained in the \hyperref[sec1]{Introduction}, this theorem also shows that the
estimator $\hat{s}$ possesses robustness properties with respect to the
Hellinger distance.

\subsection{Connection with the maximum likelihood estimator}
When the model is true and regular enough, the above theorem states
that $\sqrt{n} (\hat{\theta} - \theta_0)$ is sub-Gaussian (since in
this case Assumption~\ref{HypSurLeModeleQuelquonqueDebutDim1} holds
with $\alpha= 2$). Actually, in favourable situations, $\hat{\theta}$
shares the nice asymptotic properties of the m.l.e., and in particular
its efficiency.

\begin{thmm} \label{thmLienMLEDim1}
Suppose that the model $\F$ satisfies the following conditions:
\begin{longlist}[(viii)]
\item[(i)] There exists $\theta_0 \in(m,M)$ such that $s = f_{\theta_0}
\in\F$.
\item[(ii)] The model is identifiable, that is, for all $\theta\neq\theta
'$, $f_{\theta} \neq f_{\theta'}$.
\item[(iii)] For $\mu$-almost all $x \in\XX$, the mapping $\theta\mapsto
f_{\theta} (x) $ is continuous and positive on $[m,M]$ and two times
differentiable on $(m, M)$. Its first and second derivatives are
denoted, respectively, by $\dot{f}_{\theta} (x)$ and $\ddot{f}_{\theta
}(x)$. For $\mu$-almost all $x \in\XX$, the function $\theta\mapsto
\dot{f}_{\theta} (x)$ can be extended by continuity to $[m,M]$.
%
\item[(iv)] For all $\theta\in[m,M]$, the Fisher information
\[
I(\theta) = \int_{\XX} \bigl( \dot{l}_{\theta} (x)
\bigr)^2 f_{\theta} (x) \,\d\mu(x) \qquad\mbox{with }
\dot{l}_{\theta}(x) = \frac
{\partial\log f_{\theta} (x)}{\partial\theta}
\]
is non-zero and satisfies $\sup_{\theta\in[m,M]} I (\theta) <
\infty
$. Moreover, $\theta\mapsto I(\theta)$ is continuous at $\theta_0$.
\item[(v)] The integrals $\int_{\XX} \dot{f}_{\theta_0} (x) \,\d\mu(x)$,
$\int_{\XX} \ddot{f}_{\theta_0} (x) \,\d\mu(x)$ exist and are zero.
\item[(vi)] There exist two positive functions $\varphi_1$, $\varphi_2$ and
two numbers $\gamma_1 > 2/3$, $\gamma_2 > 0$ such that for all
$\theta,
\theta' \in(m,M)$ and $\mu$-almost all $x \in\XX$,
\begin{eqnarray*}
\bigl\llvert \log f_{\theta'} (x) - \log f_{\theta} (x) \bigr
\rrvert &\leq& \varphi_1(x) \bigl\llvert \theta'- \theta
\bigr\rrvert ^{\gamma_1},
\\
\bigl\llvert \ddot{l}_{\theta'} (x) - \ddot{l}_{\theta} (x) \bigr
\rrvert &\leq& \varphi_2(x) \bigl\llvert \theta' -
\theta\bigr\rrvert ^{\gamma_2},
\end{eqnarray*}
where $\ddot{l}_{\theta} (x)$ stands for the second derivative of
$\theta\mapsto\log f_{\theta}(x)$. Moreover, $\E [\varphi_1^3
(X_1)  ] $ and $\E [\varphi_2 (X_1)  ] $ are finite.

Furthermore, assume the following conditions on the algorithm:
%
%
\item[(vii)] The parameter $t$ depends on $n$ (one then writes $t^{(n)}$ in
place of $t$) and $t^{(n)}$ tends to $0$ in such a way that $|\log
t^{(n)}| = \mathrm{o} (n)$ when $n$ goes to infinity. The positive
parameter $\eta$ depends on $n$ and is smaller than $t^{(n)}
(\overline
{R} n)^{-1/2}$.
\item[(viii)] The parameter $\kappa\in(0,\bar{\kappa})$ is chosen
independently of $n$, the parameters $\overline{r} (\theta,\theta')$,
$\underline{r} (\theta,\theta')$ are chosen in such a way that (\ref
{EquationSurRi1}) and (\ref{EquationSurRi2}) hold and that the
algorithm converges.
\end{longlist}
Then Assumption~\ref{HypSurLeModeleQuelquonqueDebutDim1} holds with $\alpha= 2$ and there
exist $C > 0$ (that may depend on $\kappa$ and $\underline{R}$ but not
on $n$) and a sequence $(\zeta_n)_{n \geq1}$ in $[0,1]$ converging to
$0$ such that
\[
\P \Biggl[\exists\tilde{\theta} \in(m,M), \sum_{i=1}^n
\dot {l}_{\tilde{\theta}} (X_i) = 0 \mbox{ and } |\hat{\theta } -
\tilde{\theta} | \leq C \frac{t^{(n)}}{\sqrt{n}} \Biggr] \geq 1 - \zeta_n.
\]
In particular, $\hat{\theta}$ is asymptotically efficient, that is,
$\sqrt{n} (\hat{\theta} - \theta_0)$ converges in distribution to a
normal distribution with mean zero and variance $1/I(\theta_0)$.
Moreover, if there exists $\lambda> 0$ such that $\E[\exp(\lambda
\varphi_2 (X_1)) ]$, $\E[ \exp(\lambda|\dot{l}_{\theta_0} (X_1)|) ]$
and $\E[\exp(\lambda|\ddot{l}_{\theta_0}(X_1)|) ]$ are finite, then
there exists $b > 0$ such that the sequence $( \zeta_n \exp(b
n))_{n\geq1}$ is bounded above.
\end{thmm}

The main interest of $\hat{\theta}$ as compared to the m.l.e. when the
model is regular enough lies in the fact that one usually does not know
whether $s$ belongs to the model or not. If the model is true, $\hat
{\theta}$ inherits the nice asymptotic statistical properties of the
m.l.e. However, it possesses robustness properties with respect to the
Hellinger distance, which is definitively not the case for the m.l.e.

\begin{remark*}
When the model is regular enough but does not
contain the unknown density $s$, the theoretical properties of the
estimator $\hat{\theta}$ are only guaranteed by Theorem~\ref
{ThmPrincipalDim1}. When $t = t^{(n)}$ depends on $n$ and satisfies the
assumptions of Theorem~\ref{thmLienMLEDim1}, the term $D_{\F}/n$
appearing in Theorem~\ref{ThmPrincipalDim1} converges to $0$, but at a
rate slower than $1/n$. 
It is, for instance, of the order of $\log n/n$ when $t^{(n)} = a/n^k$
with $a > 0$, $k > 0$. This deteriorates the risk bound and this could
get worse since we may make this rate of convergence arbitrarily slow
by playing with $t^{(n)}$. We conjecture that this phenomenon is due to
technical difficulties and that the estimator remains good even when $t
= t^{(n)}$ is arbitrarily small or even zero (that is, with $\Fd= \F$)
as suggested by the numerical simulations (in Section~\ref{SectionSimuDim1}).
\end{remark*}

\subsection{Numerical complexity} \label{SectionDefinitionRminBarreDim1}
The numerical complexity of the estimation procedure depends on several
parameters ($\eta$, $\kappa$, $\overline{r} (\theta,\theta')$,
$\underline{r} (\theta,\theta')$) that must be chosen by the
statistician (since they are involved in the algorithm).

The role of the parameter $\eta$ is to stop the algorithm when the
confidence sets are small enough. Consequently, the smaller $\eta$, the
longer it takes to compute the estimator. Nevertheless, we shall see at
the end of this section that the time of construction of the estimator
grows slowly when $\eta$ decreases.

The parameter $\kappa$ tunes the level of the confidence sets, and thus
also the speed of the procedure: the larger $\kappa$, the faster the
procedure. Note, however, that the preceding theorems require that
$\kappa$ be smaller than $\bar{\kappa}$. There is no theoretical
guarantee when $\kappa$ is larger than $\bar{\kappa}$.

The values of the parameters $\overline{r} (\theta,\theta')$,
$\underline{r} (\theta,\theta')$ do not change the theoretical
statistical properties of the estimator given by Theorems \ref
{ThmPrincipalDim1} and \ref{thmLienMLEDim1} (provided that (\ref
{EquationSurRi1}) and (\ref{EquationSurRi2}) hold) but strongly
influence its construction time. The larger they are, the faster the
procedure is.
There are three different situations:

\textit{First case}: The Hellinger distance $ h ({\theta},
{\theta'}) $ can be made explicit. We have thus an interest in defining
them as the largest numbers for which (\ref{EquationSurRi1}) and (\ref
{EquationSurRi2}) hold, that is,
%
\begin{eqnarray}
\overline{r} \bigl(\theta,\theta'\bigr) &=& \sup \bigl\{r > 0,
[m,M] \cap [ \theta, \theta+r ] \subset\B \bigl(\theta, \kappa ^{1/2} h
\bigl({\theta}, {\theta'}\bigr) \bigr) \bigr\}, \label
{EqDefinitionRDim1Optimal1}
\\
\underline{r} \bigl(\theta,\theta'\bigr) &=& \sup \bigl\{r > 0,
[m,M] \cap \bigl[ \theta'-r, \theta' \bigr] \subset\B
\bigl(\theta', \kappa^{1/2} h \bigl({\theta}, {
\theta'}\bigr) \bigr) \bigr\} \label{EqDefinitionRDim1Optimal2}.
\end{eqnarray}

\textit{Second case}: The Hellinger distance $ h ({\theta},
{\theta'}) $ can be quickly evaluated numerically but the computation
of (\ref{EqDefinitionRDim1Optimal1}) and (\ref
{EqDefinitionRDim1Optimal2}) is difficult. We may then define them by
%
\begin{equation}
\label{EqDefintionRDim2} \underline{r} \bigl(\theta,\theta'\bigr) =
\overline{r} \bigl(\theta,\theta'\bigr) = \bigl( (\kappa/ {
\overline{R}}) h^2 \bigl({\theta}, {\theta'}\bigr)
\bigr)^{1/\alpha}.
\end{equation}
One can verify that (\ref{EquationSurRi1}) and (\ref{EquationSurRi2})
hold. When the model is regular enough and $\alpha= 2$, the value
of $\overline{R}$ can be calculated by using Fisher information [see,
for instance, Theorem~7.6 of Chapter~1 of
Ibragimov and Has'minskii \cite{Ibragimov1981}].

\textit{Third case}: The computation of the Hellinger
distance $ h ({\theta}, {\theta'}) $ involves the numerical computation
of an integral and this computation is slow. An alternative definition
is then
%
\begin{equation}
\label{EqDefintionRDim1} \underline{r} \bigl(\theta,\theta'\bigr) =
\overline{r} \bigl(\theta,\theta'\bigr) = (\kappa\underline{R}/
\overline{R})^{1/\alpha} \bigl(\theta' - \theta \bigr).
\end{equation}
As in the second case, one can check that (\ref{EquationSurRi1}) and
(\ref{EquationSurRi2}) hold.
Note, however, that the computation of the test also involves in most
cases the numerical computation of an integral (see (\ref
{eqFonctionnalBaraud})). This third case is thus mainly devoted to
models for which this numerical integration can be avoided, as for the
translation models $\F=  \{f (\cdot- \theta),   \theta\in
[m,M] \}$ with $f$ even, $\XX= \R$ and $\mu$ the Lebesgue measure
(the second term of (\ref{eqFonctionnalBaraud}) is $0$ for these
models).

We can upper bound the numerical complexity of the algorithm when
$\overline{r} (\theta,\theta')$ and $\underline{r} (\theta,\theta')$
are large enough. More precisely, we have the following.

\begin{prop} \label{PropCalculComplexiteDimen1}
Suppose that the assumptions of Theorem~\ref{ThmPrincipalDim1} hold and
that $\underline{r} (\theta,\theta')$, $\overline{r} (\theta
,\theta')$
are larger than
%
\begin{equation}
\label{eqSurretR} (\kappa\underline{R}/ \overline{R})^{1/\alpha} \bigl(
\theta' - \theta \bigr).
\end{equation}
Then the algorithm converges in less than
\[
1 + \max \bigl\{ \bigl( {\overline{R}}/{ (\kappa\underline{R})}
\bigr)^{1/\alpha}, {1}/{\log2} \bigr\} \log \biggl( \frac{M-m}{\eta} \biggr)
\]
iterations.\vadjust{\goodbreak}
\end{prop}

This is an improvement with respect to the procedure of Baraud \cite
{BaraudMesure} where the number of tests computed is roughly of the
order of $|\Fd|^2$, which is much larger than the above bound when
$\varepsilon$ is small enough (and $\eta= \varepsilon$).

\section{Simulations for one-dimensional models} \label
{SectionSimuDim1} In what follows, we carry out a simulation study in
order to investigate more precisely the performance of our estimator.
We simulate samples $(X_1,\ldots,X_n)$ with density $s$ and use our
procedure to estimate $s$.

\subsection{Models} Our simulation study is based on the following models.

\begin{ExempleSimuDimmm}\label{ex1}
$\F=  \{f_{\theta},   \theta\in[0.01, 100] \} $ where
$f_{\theta} (x) = \theta \mathrm{e}^{- \theta x} \1_{[0,+\infty)} (x)$ for all
$x \in\R$.
\end{ExempleSimuDimmm}

\begin{ExempleSimuDimmm}\label{ex2}
$\F=  \{f (\cdot-\theta),   \theta\in[-100, 100] \} $
where $f$ is the density of a standard Gaussian distribution.
\end{ExempleSimuDimmm}

\begin{ExempleSimuDimmm}\label{ex3}
$\F=  \{f (\cdot-\theta),   \theta\in[-10, 10] \} $ where
$f$ is the density of a standard Cauchy distribution.
\end{ExempleSimuDimmm}

\begin{ExempleSimuDimmm}\label{ex4}
$\F=  \{f_{\theta},   \theta\in[0.01, 10] \} $ where
$f_{\theta} = \theta^{-1} \1_{[0,\theta]}$.
\end{ExempleSimuDimmm}

\begin{ExempleSimuDimmm}\label{ex5}
$\F=  \{f_{\theta},   \theta\in[-10, 10] \} $ where
$f_{\theta} (x) = \frac{1}{(x-\theta+1)^2 }\1_{[\theta,+\infty)} (x)$
for all $x \in\R$.
\end{ExempleSimuDimmm}

\begin{ExempleSimuDimmm}\label{ex6}
$\F=  \{\1_{[\theta- 1/2, \theta+ 1/2]},   \theta\in[-10,
10] \}$.
\end{ExempleSimuDimmm}

\begin{ExempleSimuDimmm}\label{ex7}
$\F=  \{f (\cdot- \theta),   \theta\in[-1, 1] \} $ where
$f$ is defined by (\ref{eqModelIrregulier}).
\end{ExempleSimuDimmm}

In these examples, we shall mainly compare our estimator with the
maximum likelihood one. In Examples \ref{ex1}, \ref{ex2}, \ref{ex4} and \ref{ex5},
the m.l.e. $\tilde
{\theta}_{\mathrm{mle}}$ can be made explicit and is thus easy to
compute. Finding the m.l.e. is more delicate for the problem of
estimating the location parameter of a Cauchy distribution, since the
likelihood function may be multimodal. We refer
to Barnett \cite{Barnett1966} for a discussion of numerical methods
devoted to the maximization of the likelihood. In this simulation
study, we avoid the issues of the numerical algorithms by computing the
likelihood at $10^6$ equally spaced points between $\max(-10, \hat
{\theta}-1)$ and $\min(10, \hat{\theta}+1)$ (where $\hat{\theta}$ is
our estimator) and at $10^{6}$ equally spaced points between $\max(-10,
\tilde{\theta}_{\mathrm{median}}-1)$ and $\min(10, \tilde{\theta
}_{\mathrm
{median}}+1)$ where $\tilde{\theta}_{\mathrm{median}}$ is the median. We
then select among these points the one for which the likelihood is
maximal. In Example~\ref{ex4}, we shall also compare our estimator to the
estimator of the family $\{a \max_{1 \leq i \leq n} X_i, a > 0\}$ that
minimizes the Hellinger quadratic risk, that is,
\[
\tilde{\theta}_{\mathrm{best}} = \biggl(\frac{4 n}{2 n + 1} \biggr)^{2/(2n-1)}
\max_{1 \leq i \leq n} X_i.
\]
In Example~\ref{ex6}, we shall compare our estimator to
\[
\tilde{\theta}' = \frac{1}{2} \Bigl(\max
_{1 \leq i \leq n} X_i + \min_{1 \leq i \leq n}
X_i \Bigr).
\]
In the case of Example~\ref{ex7}, the likelihood is infinite at each
observation and the maximum likelihood method fails. We shall then
compare our estimator to the median and the empirical mean but also to
the maximum spacing product estimator $\tilde{\theta}_{\mathrm{mspe}}$
(m.s.p.e. for short).
This estimator was introduced by Cheng and Amin \cite{Cheng1983} and
Ranneby \cite{Ranneby1984} to deal with parametric models for which the
likelihood is unbounded. It is known to possess nice theoretical
properties when $s$ does belong to $\F$. We refer, for instance, to the
two aforementioned papers and
to Ekstr{\"o}m \cite{Ekstrom1998}, Shao and Hahn \cite{Shao1999}, Ghosh
and Jammalamadaka \cite{Ghost2001}, Anatolyev and Kosenok \cite
{Anatolyev2005}. This estimator is, however, not robust. In Example~\ref{ex4},
it is, for instance, defined by $(1+1/n)\max_{1 \leq i \leq n} X_i$
when all the observations $X_i$ are positive, and is therefore highly
sensitive to outliers. In Example~\ref{ex7}, any estimator with values in
$[-1,1]$ is a m.s.p.e. when $\max_{1 \leq i \leq n} X_i - \min_{1
\leq
i \leq n} X_i > 2$. The practical construction of the m.s.p.e. in
Example~\ref{ex7} involves the problem of finding a global maximum of the
maximum product function on $\Theta= [-1,1]$ which may be multimodal.
We compute it by considering $2 \times10^5$ equally spaced points
between $-1$ and $1$ and by calculating, for each of these points, the
function to maximize. We then select the point for which the function
is maximal. Using more points would give more accurate results,
especially when $n$ is large, but we are limited by the capacity of the
computer.

\subsection{Implementation of the procedure}
In this simulation study, we take arbitrarily $\kappa= \bar{\kappa
}/2$. We choose $\eta$ very small but not too much to avoid undesirable
numerical issues. More precisely, $\eta= (M-m) / 10^{8}$ (it is small
enough in view of the values of $\alpha$ and $n$).

The choice of $\underline{r} (\theta,\theta')$ and $\overline{r}
(\theta
,\theta')$ varies according to the examples.
In Examples \ref{ex1}, \ref{ex2}, \ref{ex4} and~\ref{ex6}, we define them by (\ref
{EqDefinitionRDim1Optimal1}) and (\ref{EqDefinitionRDim1Optimal2}). In
Examples \ref{ex3} and \ref{ex5}, we define them by (\ref{EqDefintionRDim2}). In the
first case, $\alpha= 2$ and $\overline{R} = 1/16$, while in the second
case, $\alpha= 1$ and $\overline{R} = 1/2$. In the case of Example~\ref{ex7},
we use (\ref{EqDefintionRDim1}) with $\alpha= 1/2$, $\underline{R} =
0.17$ and $\overline{R} = 1/\sqrt{2}$.

It remains to choose $t$ which tunes the thinness of the net $\Fd$.
When the model is regular enough and contains $s$, a good choice of $t$
seems to be $t = 0$ (that is, $\Theta_{\mathrm{dis}} = \Theta$, $\Fd=
\F$
and $T ({\theta}, {\theta'}) = \overline{T}(f_{\theta}, f_{\theta'})$),
since then the simulations suggest that our estimator is almost equal
to the m.l.e. (with large probability). In all the simulations, we take
$t = 0$ (although this is not theoretically justified).

\subsection{Risks when \texorpdfstring{$s\in\F$}{$s in\mathscr{F}$}} \label{SectionDim1VraiS} We begin to
simulate $N$ samples $(X_1,\ldots,X_n)$ when the true density $s$
belongs to the model $\F$. They are generated according to the
density $s = f_1$ in Examples \ref{ex1}, \ref{ex4} and according to $s = f_0$ in
Examples \ref{ex2}, \ref{ex3}, \ref{ex5}, \ref{ex6}, \ref{ex7}.

We evaluate the quality of an estimator $\tilde{\theta}$ by computing
it on each of the $N$ samples. Let $\tilde{\theta}^{(i)}$ be the value
of this estimator corresponding to the $i$th sample
and let
\[
\widehat{R}_N (\tilde{\theta}) = \frac{1}{N} \sum
_{i=1}^N h^2 (s, f_{\tilde{\theta}^{(i)}}).
\]
The risk $\E [ h^2 (s, f_{\tilde{\theta}})  ]$ of the
estimator $\tilde{\theta}$ is estimated by $\widehat{R}_N(\tilde
{\theta})$.
We also introduce
\[
\widehat{\mathcal{R}}_{N,\mathrm{rel}} (\tilde{\theta}) = \frac
{\widehat
{R}_N (\hat{\theta})}{\widehat{R}_N (\tilde{\theta})}- 1
\]
in order to make the comparison of our estimator $\hat{\theta}$ and the
estimator $\tilde{\theta}$ easier.
When $\widehat{\mathcal{R}}_{N,\mathrm{rel}} (\tilde{\theta})$ is
negative, our estimator is better than $\tilde{\theta}$, whereas if
$\widehat{\mathcal{R}}_{N,\mathrm{rel}} (\tilde{\theta})$ is positive,
our estimator is worse than $\tilde{\theta}$. More precisely, if
$\widehat{\mathcal{R}}_{N,\mathrm{rel}} (\tilde{\theta}) = \alpha$, the
risk of our estimator corresponds to the one of $\tilde{\theta}$
reduced of $100 |\alpha| \%$ when $\alpha< 0$ and increased of $100
\alpha\%$ when $\alpha> 0$.\looseness=-1

\begin{table}
\caption{Risks of the estimators}\label{tab1}
\begin{tabular*}{\textwidth}{@{\extracolsep{4in minus 4in}}lld{2.4}d{2.4}d{2.4}d{2.4}d{2.4}@{}}
\hline
& & \multicolumn{1}{l}{$n = 10$} & \multicolumn{1}{l}{$n = 25$} &
\multicolumn{1}{l}{$n = 50$} & \multicolumn{1}{l}{$n = 75$} &
\multicolumn{1}{l@{}}{$n = 100$}
\\
\hline
Example~\ref{ex1} & $\widehat{R}_{10^6}(\hat{\theta})$ & 0.0130 & 0.0051 &
0.0025 & 0.0017 & 0.0013 \\
& $\widehat{R}_{10^6}(\tilde{\theta}_{\mathrm{mle}})$ & 0.0129 & 0.0051
& 0.0025 & 0.0017 & 0.0013 \\[3pt]
& $\widehat{\mathcal{R}}_{10^6,\mathrm{rel}} (\tilde{\theta}_{\mathrm
{mle}})$ & \multicolumn{1}{c}{$6 \cdot10^{-4}$} & \multicolumn{1}{c}{$10^{-5}$} & \multicolumn{1}{c}{$7 \cdot10^{-7}$} & \multicolumn{1}{c}{$-8
\cdot10^{-9}$} & \multicolumn{1}{c@{}}{$2 \cdot10^{-9}$} \\[6pt]
Example~\ref{ex2} & $\widehat{R}_{10^6}(\hat{\theta})$ & 0.0123 & 0.0050 &
0.0025 & 0.0017 & 0.0012 \\
& $\widehat{R}_{10^6}(\tilde{\theta}_{\mathrm{mle}})$ & 0.0123 & 0.0050
& 0.0025 & 0.0017 & 0.0012 \\[3pt]
& $\widehat{\mathcal{R}}_{10^6,\mathrm{rel}} (\tilde{\theta}_{\mathrm
{mle}})$ & \multicolumn{1}{c}{$5 \cdot10^{-10}$} & \multicolumn{1}{c}{$9 \cdot10^{-10}$} & \multicolumn{1}{c}{$- 2 \cdot
10^{-9}$} & \multicolumn{1}{c}{$- 2 \cdot10^{-9}$} & \multicolumn{1}{c@{}}{$- 3 \cdot10^{-9}$} \\[6pt]
Example~\ref{ex3} & $\widehat{R}_{10^6}(\hat{\theta})$ & 0.0152 & 0.0054 &
0.0026 & 0.0017 & 0.0013 \\
& $\widehat{R}_{10^4}(\tilde{\theta}_{\mathrm{mle}})$ & 0.0149 & 0.0054
& 0.0026 & 0.0017 & 0.0012 \\[3pt]
& $\widehat{\mathcal{R}}_{10^4,\mathrm{rel}} (\tilde{\theta}_{\mathrm
{mle}})$ & -0.001 & \multicolumn{1}{c}{$-2 \cdot10^{-4}$} & \multicolumn{1}{c}{$- 10^{-8}$} & \multicolumn{1}{c}{$-3 \cdot
10^{-8}$} & \multicolumn{1}{c@{}}{$9 \cdot10^{-8}$} \\[6pt]
Example~\ref{ex4} & $\widehat{R}_{10^6}(\hat{\theta})$ & 0.0468 &0.0192 &
0.0096 & 0.0064 & 0.0048 \\
& $\widehat{R}_{10^6}(\tilde{\theta}_{\mathrm{mle}})$ & 0.0476 & 0.0196
& 0.0099 & 0.0066 & 0.0050 \\
& $\widehat{R}_{10^6}(\tilde{\theta}_{\mathrm{best}})$ & 0.0333 & 0.0136
& 0.0069 & 0.0046 & 0.0035 \\[3pt]
& $\widehat{\mathcal{R}}_{10^6,\mathrm{rel}} (\tilde{\theta}_{\mathrm
{mle}})$ & -0.0160 & -0.0202 & -0.0287 & -0.0271 & -0.0336 \\
& $\widehat{\mathcal{R}}_{10^6,\mathrm{rel}} (\tilde{\theta}_{\mathrm
{best}})$ & 0.4059 & 0.4086 & 0.3992 & 0.4025 & 0.3933 \\[6pt]
Example~\ref{ex5} & $\widehat{R}_{10^6}(\hat{\theta})$ & 0.0504 & 0.0197 &
0.0098 & 0.0065 & 0.0049 \\
& $\widehat{R}_{10^6}(\tilde{\theta}_{\mathrm{mle}})$ &0.0483 &
0.0197 &
0.0099 & 0.0066 & 0.0050 \\[3pt]
& $\widehat{\mathcal{R}}_{10^6,\mathrm{rel}} (\tilde{\theta}_{\mathrm
{mle}})$ & 0.0436 & -0.0019 & -0.0180 & -0.0242 & -0.0263 \\[6pt]
Example~\ref{ex6} & $\widehat{R}_{10^6}(\hat{\theta})$ & 0.0455 & 0.0193 &
0.0098 & 0.0066 & 0.0050 \\
& $\widehat{R}_{10^6}(\tilde{\theta}')$ & 0.0454 & 0.0192 & 0.0098 &
0.0066 & 0.0050 \\[3pt]
& $\widehat{\mathcal{R}}_{10^6,\mathrm{rel}} (\tilde{\theta}')$ & 0.0029
& 0.0029 & 0.0031 & 0.0028 & 0.0030 \\[6pt]
Example~\ref{ex7} & $\widehat{R}_{10^4}(\hat{\theta})$ & 0.050 & 0.022 & 0.012
& 0.008 &0.006 \\
& $\widehat{R}_{10^4}(\tilde{\theta}_{\mathrm{mean}})$ & 0.084 &
0.061 &
0.049 & 0.043 & 0.039 \\
& $\widehat{R}_{10^4}(\tilde{\theta}_{\mathrm{median}})$ & 0.066 & 0.036
& 0.025 & 0.019 & 0.017 \\[1pt]
& $\widehat{R}_{10^4}(\tilde{\theta}_{\mathrm{mspe}})$ & 0.050 &
0.022 &
0.012 & 0.008 & 0.006 \\[3pt]
& $\widehat{\mathcal{R}}_{10^4,\mathrm{rel}} (\tilde{\theta}_{\mathrm
{mean}})$ &-0.40 & -0.64 & -0.76 & -0.82 & -0.85 \\
& $\widehat{\mathcal{R}}_{10^4, \mathrm{rel}} (\tilde{\theta}_{\mathrm
{median}})$ & -0.25 & -0.39 & -0.54 & -0.59 & -0.65 \\
\hline
\end{tabular*}
\end{table}

The numerical results are given in Table \ref{tab1}.
In the first three examples, the risk of our estimator is almost equal
to the one of the m.l.e., whatever the value of $n$. In Example~\ref{ex4}, our
estimator slightly improves the maximum likelihood estimator but has a
risk $40 \%$ larger than the one of $\tilde{\theta}_{\mathrm{best}}$. In
Example~\ref{ex5}, the risk of our estimator is larger than the one of the
m.l.e. when $n = 10$ but is slightly smaller as soon as $n$ becomes
larger than $25$. In Example~\ref{ex6}, the risk of our estimator is {$0.3 \%$}
larger than the one of $\tilde{\theta}'$. In Example~\ref{ex7}, our estimator
significantly improves the empirical mean and the median. Its risk is
comparable to the one of the m.s.p.e.

When the model is regular enough, these simulations show that our
estimation strategy provides an estimator whose risk is very close to
the one of the maximum likelihood estimator. Moreover, our estimator
seems to work rather well in a model where the m.l.e. does not exist
(case of Example~\ref{ex7}).

\subsection{Link with the m.l.e.}
We now study numerically the connection between our estimator and the
m.l.e. when the model is regular enough (that is, in the first three
examples). Let for $c \in\{0.99,0.999, 1\}$, $q_{c}$ be the
$c$-quantile of the random variable $  |\hat{\theta} - \tilde
{\theta
}_{\mathrm{mle}}  |$, and $\hat{q}_c$ be the empirical version based
on $N$ samples ($N = 10^{6}$ in Examples \ref{ex1}, \ref{ex2} and $N = 10^{4}$ in
Example~\ref{ex3}).

\begin{table}
\caption{Connection with the m.l.e.}\label{tab2}
\begin{tabular*}{\textwidth}{@{\extracolsep{\fill}}lllllll@{}}
\hline
& & \multicolumn{1}{l}{$n = 10$} & \multicolumn{1}{l}{$n = 25$} & \multicolumn{1}{l}{$n = 50$} & \multicolumn{1}{l}{$n = 75$} & \multicolumn{1}{l@{}}{$n = 100$} \\
\hline
Example~\ref{ex1} & $\hat{q}_{0.99}$ & $10^{-7}$ & $10^{-7}$ & $10^{-7}$ &
$10^{-7}$ & $10^{-7}$\\
& $\hat{q}_{0.999} $ & $0.07$ & $10^{-7}$ & $10^{-7}$ & $10^{-7}$ &
$10^{-7}$ \\
& $\hat{q}_{1} $ & $1.9$ & $0.3$ & $0.06$ & $0.005$ & $10^{-7}$ \\[3pt]
Example~\ref{ex2} & $\hat{q}_{0.99}$ & $2 \cdot10^{-7}$ & $3 \cdot10^{-7}$ &
$3 \cdot10^{-7}$ & $3 \cdot10^{-7}$ & $3 \cdot10^{-7}$ \\
& $\hat{q}_{0.999} $ & $3 \cdot10^{-7}$ & $3 \cdot10^{-7}$ & $3
\cdot10^{-7}$ & $3 \cdot10^{-7}$ & $3 \cdot10^{-7}$ \\
& $\hat{q}_{1}$ & $3 \cdot10^{-7}$ & $3 \cdot10^{-7}$ & $3 \cdot
10^{-7}$ & $3 \cdot10^{-7}$ & $3 \cdot10^{-7}$ \\ [3pt]
Example~\ref{ex3} & $\hat{q}_{0.99}$ & $10^{-6}$ & $10^{-6}$ & $10^{-6}$ &
$10^{-6}$ & $10^{-6}$ \\
& $\hat{q}_{0.999} $ & $3\cdot10^{-6}$ & $10^{-6}$ & $10^{-6}$ &
$10^{-6}$ & $10^{-6}$ \\
& $\hat{q}_{1}$ & $1.5$ & $0.1$ & $10^{-6}$ & $10^{-6}$ & $10^{-6}$ \\
\hline
\end{tabular*}
\end{table}

Table \ref{tab2} shows that with large probability, our estimator is almost
equal to the m.l.e. This probability is quite high for small values of
$n$ and even more for larger values of $n$. This explains why the risks
of these two estimators are very close in the first three examples.
Note that the value of $\eta$ prevents the empirical quantiles from
being smaller\vadjust{\goodbreak} than something of the order $10^{-7}$ according to the
examples (in Example~\ref{ex3}, the value of $10^{-6}$ is due to the procedure
used to build the m.l.e.).

\subsection{Speed of the procedure} \label{SectionSpeedProcedure} For
the sake of completeness, we specify in Table \ref{figureriskcomparaison} the number of tests that
have been calculated in the preceding examples.

\begin{table}[b]
\caption{Number of tests computed averaged over $10^6$ samples for
Examples \ref{ex1} to \ref{ex6} and over $10^4$ samples for Example \ref{ex7}. The
corresponding standard deviations are in brackets}
 \label{figureriskcomparaison}
\begin{tabular*}{\textwidth}{@{\extracolsep{\fill}}llllll@{}}
\hline
& \multicolumn{1}{l}{$n = 10$} & \multicolumn{1}{l}{$n = 25$} & \multicolumn{1}{l}{$n = 50$} &
\multicolumn{1}{l}{$n = 75$} & \multicolumn{1}{l@{}}{$n = 100$} \\
\hline
Example~\ref{ex1} & 77 (1.4) & 77 (0.9) & 77 (0.7) & 77 (0.6) & 77 (0.5) \\
Example~\ref{ex2} & 293 (1) & 294 (1) & 294 (0.9) & 295 (0.9) & 295 (0.9) \\
Example~\ref{ex3} & 100 (3.5) & 100 (0.5) & 100 (0.001) & 100 (0) & 100 (0) \\
Example~\ref{ex4} & 460 (3) & 461 (1) & 462 (0.6) & 462 (0.4) & 462 (0.3) \\
Example~\ref{ex5} & 687 (0) & 687 (0) & 687 (0) & 687 (0) & 687 (0)\\
Example~\ref{ex6} & 412 (8) & 419 (8) & 425 (8) & 429 (8) & 432 (8) \\
Example~\ref{ex7} & 173,209 (10) & 173,212 (0) & 173,212 (0.9) & 173,206 (12) &
173,212 (0.3) \\
\hline
\end{tabular*}
\end{table}

We observe in Figure \ref{fig1} that the number of tests computed is quite small, except for
Example~\ref{ex7}. The number of tests computed in this example is quite large
because $\underline{r} (\theta,\theta')$ and $\overline{r} (\theta
,\theta')$ are defined by relation (\ref{EqDefintionRDim1}) and
$\alpha
= 1/2$. The smaller $\alpha$, the longer it takes to compute the
estimator. Notice however that is possible to use less tests by
choosing $\kappa$ closer to $\bar{\kappa}$ or by using a more accurate
control of the Hellinger distance $h({\theta},{\theta'})$.

\subsection{Simulations when \texorpdfstring{$s\notin\F$}{$s notin\mathscr{F}$}} \label
{SectionRobustessDim1} In Section~\ref{SectionDim1VraiS}, we were in
the favourable situation where the true density $s$ belonged to the
model $\F$, which may not hold true in practice.
We now work with random variables $X_1,\ldots,X_n$ simulated according
to a density $s \notin\F$ to illustrate the robustness properties of
our estimator.

We propose an example based on the mixture of two uniform laws. We use
the parametric model $\F=  \{ f_{\theta},   \theta\in[0.01,
10] \}$ with $f_{\theta} = \theta^{-1} \1 _{[0, \theta]}$,
take $p
\in(0,1)$ and simulate the data according to the density
\[
s_{p} (x) = (1-p) f_1(x) + p f_2(x)\qquad
\mbox{for all $x \in\R$.}
\]
Set $p_0 = 1-1/\sqrt{2}$. One can check that
\begin{eqnarray*}
h^2(s_p,\F) &=& \cases{ h^2(s_p,f_1),
& \quad $\mbox{if $p \leq p_0$},$ \vspace*{2pt}
\cr
h^2(s_p,f_2),
&\quad $ \mbox{if $p \geq p_0$}$}
\\
&=& \cases{ 1 - \sqrt{2-p}/\sqrt{2}, &\quad  $\mbox{if $p \leq p_0$},$
\vspace*{2pt}
\cr
1 - (\sqrt{2-p}+\sqrt{p})/{2}, & \quad $\mbox{if $p \geq
p_0$,}$}
\end{eqnarray*}
which means that the best approximation of $s_p$ in $\F$ is $f_1$ when
$p < p_0$ and $f_2$ when $p > p_0$.

We now compare our estimator $\hat{\theta}$ to the m.l.e. $\tilde
{\theta
}_{\mathrm{mle}} = \max_{1 \leq i \leq n} X_i$. For a lot of values
of $p$, we simulate $N$ samples of $n$ random variables with
density $s_p$ and investigate the behaviour of the estimator $\tilde
{\theta} \in\{\hat{\theta}, \tilde{\theta}_{\mathrm{mle}}\}$ by
computing the function
\[
\widehat{R}_{p,n,N} (\tilde{\theta}) = \frac{1}{N} \sum
_{i=1}^N h^2 (s_p,
f_{\tilde{\theta}^{(p,i)}}),
\]
where $\tilde{\theta}^{(p,i)}$ is the value of the estimator $\tilde
{\theta}$ corresponding to the $i$th sample whose
density is $s_p$. We draw below the functions $p \mapsto\widehat
{R}_{p,n,N} (\hat{\theta}) $, $p \mapsto\widehat{R}_{p,n,N} (\tilde
{\theta}) $ and $p \mapsto h^2(s_p, \F)$ for $n = 10^2$ and then for $n
= 10^4$.
%
\begin{figure}

\includegraphics{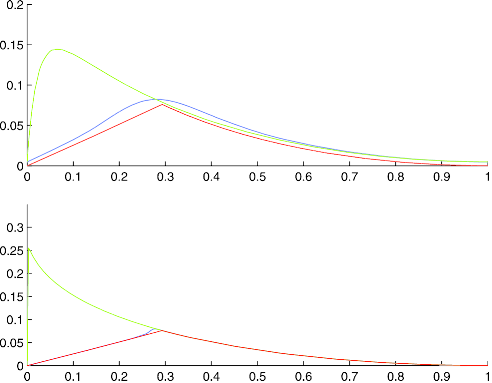}

\caption{Red: $p \mapsto h^2(s_p, \F)$. Blue: $p \mapsto\widehat
{R}_{p,n,5000} (\hat{\theta}) $. Green: $p \mapsto\widehat
{R}_{p,n,5000} (\tilde{\theta}_{\mathrm{mle}})$.}\label{fig1}
\end{figure}

We observe in Figure \ref{fig1} that the m.l.e. is rather good when $p \geq p_0$ and very
poor when $p < p_0$. This can be explained by the fact that
the m.l.e. $\tilde{\theta}_{\mathrm{mle}}$ is close to $2$ as soon as
the number $n$ of observations is large enough. The shape of the
function $p \mapsto\widehat{R}_{p,n,5000} (\hat{\theta}) $ looks more
like the function $p \mapsto h^2(s_p, \F)$. The lower figure suggests
that $ \widehat{R}_{p,n,N} (\hat{\theta}) $ converges to $h^2(s_p,
\F)$
when $n,N$ go to infinity except on a small neighbourhood before~$p_0$.

\section{Models parametrized by a multi-dimensional parameter} \label
{SectionEstimationDimQuel}

\subsection{Assumption}
In the preceding sections, we have dealt with models indexed by a
finite interval of $\R$. We now turn to the multi-dimensional case and
consider models $\F=  \{f_{\ttheta},   \ttheta\in\Theta
 \}
$ indexed by a rectangle
$\Theta= \prod_{j=1}^d [m_j, M_j]$
of $\R^d$ and satisfying a multi-dimensional version of
Assumption~\ref
{HypSurLeModeleQuelquonqueDebutDim1}.

\begin{hyp} \label{HypSurLeModeleQuelquonqueDebutDimD}
There exist positive numbers $\alpha_1,\ldots,\alpha_d$, $\underline
{R}_1,\ldots,\underline{R}_d$, $\overline{R}_1,\ldots,\overline{R}_d$ such
that for all $\ttheta= (\theta_1,\ldots,\theta_d)$, $\ttheta' =
(\theta
_1',\ldots,\theta_d') \in\Theta= \prod_{j=1}^d [m_j, M_j]$,
\[
\sup_{j \in\{1,\ldots,d\}} \underline{R}_j \bigl|
\theta_j - \theta '_j\bigr|^{\alpha_j}
\leq h^2 \bigl({\ttheta}, {\ttheta'} \bigr) \leq \sup
_{j \in\{1,\ldots,d\}} \overline{R}_j \bigl|\theta_j -
\theta'_j\bigr|^{\alpha_j}.
\]
\end{hyp}

\subsection{Definition of the test} \label{SectionDefTestDimd}
As in the one-dimensional case, our estimation strategy is based on the
existence for all $\ttheta, \ttheta' \in\Theta$ of a measurable
function $T(\ttheta, \ttheta')$ of the observations possessing suitable
statistical properties. The definition of this functional is the
natural extension of the one we have proposed in Section~\ref{SectionDefTest}.

Let for $j \in\{1,\ldots,d\}$, $t_j \in(0,d^{1/\alpha_j}]$ and
$\varepsilon_j = t_j (\overline{R} n)^{-1/\alpha_j}$. We introduce the
finite sets
\begin{eqnarray*}
\Theta_{\mathrm{dis}} &=& \bigl\{ (m_1+ k_1
\varepsilon_1, \ldots, m_d+ k_d
\varepsilon_d ), \forall j \in\{1,\ldots, d\}, k_j
\leq(M_j-m_j) \varepsilon_j^{-1}
\bigr\},
\\
\Fd&=& \{f_{\ttheta}, \ttheta\in\Theta_{\mathrm{dis}} \}
\end{eqnarray*}
and the map $\pi$ on $\prod_{j=1}^d [m_j, M_j]$ by
%
\begin{eqnarray}
\pi({\mathbf{x}}) = \bigl(m_1 + \bigl\lfloor(x_1 -
m_1) / \varepsilon _1\bigr\rfloor\varepsilon_1,
\ldots, m_d + \bigl\lfloor(x_d - m_d) /
\varepsilon _d\bigr\rfloor\varepsilon_d \bigr)\nonumber \\
\eqntext{\mbox{for
all $\displaystyle\mathbf{x} = (x_1,\ldots,x_d) \in\prod
_{j=1}^d [m_j, M_j]$},}
\end{eqnarray}
where $\lfloor\cdot\rfloor$ is the integer part. We then define
$T(\ttheta,\ttheta')$ for all $\ttheta, \ttheta' \in\Theta$ by
%
\begin{equation}
\label{eqDefinitionTDimD} {T} \bigl({\ttheta},{\ttheta'}\bigr) =
\overline{T}(f_{\pi(\ttheta)},f_{\pi
(\ttheta')}),
\end{equation}
where $\overline{T}$ is given by (\ref{eqFonctionnalBaraud}).

\subsection{Basic ideas}\label{SectionHeuristiqueDim2}
For the sake of simplicity, we first restrict ourselves to the
dimension $d = 2$. The idea is to build a decreasing sequence $(\Theta
_i)_i$ of rectangles by induction (in the sense of set inclusion). When
there exists $\ttheta_0 \in\Theta$ such that $s = f_{\ttheta_0}$,
these rectangles $\Theta_i$ can be interpreted as confidence sets for
$\ttheta_0$.

We set $\Theta_1 = \Theta$. We suppose that the rectangle $\Theta_i$
has already been built and aim at building $\Theta_{i+1}$.

Let $a_1,b_1,a_2,b_2$ be such that $\Theta_i= [a_1,b_1] \times
[a_2,b_2]$. For all $\ttheta= (\theta_1,\theta_2) \in\Theta_i$,
$\ttheta' = (\theta'_1, \theta'_2) \in\Theta_i$, let $\mathcal{R}
(\ttheta, \ttheta')$ be a rectangle included in $\Theta_i$ and
containing a neighbourhood of $\ttheta$ (for the usual topology on
$\Theta_i$) such that
\[
\mathcal{R} \bigl(\ttheta, \ttheta'\bigr) \subset\B \bigl(\ttheta,
\kappa^{1/2} h \bigl({\ttheta}, {\ttheta'}\bigr) \bigr).
\]
We recall that for all $\ttheta\in\Theta$ and $r > 0$, $\B(\ttheta,
r) =  \{\ttheta' \in\Theta,   h ({\ttheta}, {\ttheta'}) \leq r
 \}$. Let $\mathcal{P}$ and $\mathcal{P}'$ be the two horizontal
sides of the rectangle $\Theta_i$:
\begin{eqnarray*}
\mathcal{P} &=& [a_1, b_1] \times\{a_2\},
\\
\mathcal{P}' &=& [a_1, b_1] \times
\{b_2\}.
\end{eqnarray*}
We begin by building $L+1$ elements $\ttheta^{(\ell)} \in\mathcal{P}$
and $L+1$ elements $\ttheta'^{(\ell)} \in\mathcal{P}'$ in such a way
that if $ \mathcal{R}^{(\ell)} $ designates the set
\begin{eqnarray*}
\mathcal{R}^{(\ell)} = \cases{ \mathcal{R} \bigl(\ttheta^{(\ell)},
\ttheta'^{(\ell)}\bigr), &\quad  $\mbox{if $T \bigl({
\ttheta^{(\ell)}},{\ttheta'^{(\ell)}}\bigr) > 0$,}$
\vspace*{2pt}
\cr
\mathcal{R} \bigl(\ttheta'^{(\ell)},
\ttheta^{(\ell)}\bigr), &\quad  $\mbox{if $T \bigl({\ttheta^{(\ell)}},{
\ttheta'^{(\ell)}}\bigr) < 0$,}$ \vspace*{2pt}
\cr
\mathcal{R}
\bigl(\ttheta^{(\ell)}, \ttheta'^{(\ell)}\bigr) \cup
\mathcal{R} \bigl(\ttheta'^{(\ell)}, \ttheta^{(\ell)}
\bigr), &\quad  $\mbox{if $T \bigl({\ttheta ^{(\ell
)}},{\ttheta'^{(\ell)}}
\bigr) = 0$,}$ }
\end{eqnarray*}
then, either

\begin{equation}
\label{definitiondeP} \mathcal{P} = \bigcup_{\ell= 1}^L
\bigl(\mathcal{R}^{(\ell)} \cap \mathcal{P} \bigr)\quad \mbox{or}\quad
\mathcal{P}' = \bigcup_{\ell
= 1}^L
\bigl(\mathcal{R}^{(\ell)} \cap\mathcal{P}' \bigr).
\end{equation}
The rectangle $\Theta_{i+1}$ is then defined in such a way that\vspace*{-1pt}
\[
\Theta_i {}\Big\backslash{}\bigcup_{\ell=1}^L
\mathcal{R}^{(\ell)} \subset \Theta_{i+1} \subset
\Theta_i\vspace*{-1pt}
\]
and that $\Theta_{i+1} \neq\Theta_i$. Its theoretical existence is
guaranteed by (\ref{definitiondeP}). Besides, it follows from the
heuristics of Section~\ref{SectionHeuristique} that $\Theta_{i+1}$ may
be interpreted as a confidence set for $\ttheta_0$ whenever it exists
(since it contains $\Theta_i \setminus\bigcup_{\ell=1}^L \mathcal
{R}^{(\ell)}$).

It remains to define $\ttheta^{(\ell)}$ and $\ttheta'^{(\ell)}$ for all
$\ell\in\{1,\ldots,L+1\}$. We define $\ttheta^{(1)} = (a_1,a_2)$ as
the bottom left corner of $\Theta_i$ and $\ttheta'^{(1)} = (a_1,b_2)$
as the top left corner of $\Theta_i$.
The definition of $\ttheta^{(2)}$ and $\ttheta'^{(2)}$ depends on the
sign of $T ({\ttheta^{(1)}},{\ttheta'^{(1)}})$:
\begin{itemize}
\item If $T ({\ttheta^{(1)}},{\ttheta'^{(1)}}) > 0$, we define
$\ttheta
^{(2)}$ as the bottom right corner of $\mathcal{R} (\ttheta^{(1)},
\ttheta'^{(1)}) $ and $\ttheta'^{(2)} = \ttheta'^{(1)}$.
\item If $T ({\ttheta^{(1)}},{\ttheta'^{(1)}}) < 0$, we define
$\ttheta
^{(2)} = \ttheta^{(1)}$ and $\ttheta'^{(2)}$ as the top right corner of
$\mathcal{R} (\ttheta'^{(1)}, \ttheta^{(1)}) $.
\item If $T ({\ttheta^{(1)}},{\ttheta'^{(1)}}) = 0$, we define
$\ttheta
^{(2)}$ as the bottom right corner of $\mathcal{R} (\ttheta^{(1)},
\ttheta'^{(1)}) $ and $\ttheta'^{(2)}$ as the top right corner of
$\mathcal{R} (\ttheta'^{(1)}, \ttheta^{(1)}) $.
\end{itemize}
If $\ttheta^{(2)} = (b_1,a_2)$ or if $\ttheta'^{(2)} = (b_1,b_2)$,
which means that one of these two points is a right corner of $\Theta
_i$, we set $L = 1$. 
In the contrary case, we define $\ttheta^{(3)}$ either as the bottom
right corner of $\mathcal{R} (\ttheta^{(2)}, \ttheta'^{(2)}) $ or as
$\ttheta^{(2)}$, according to the sign of $T ({\ttheta^{(2)}},
{\ttheta
'^{(2)}})$. Similarly, $\ttheta'^{(3)}$ is either $\ttheta'^{(2)}$ or
the top right corner of $\mathcal{R} (\ttheta'^{(2)}, \ttheta^{(2)}) $.
If one of the points $\ttheta^{(3)}$, $\ttheta'^{(3)}$ is a right
corner of $\Theta_i$, we set $L = 2$. 
Otherwise, we build $\ttheta^{(4)}$, $\ttheta'^{(4)}$ and so on. More
precisely, we build $\ttheta^{(\ell)}$ and $\ttheta'^{(\ell)}$ until
that one of these two elements becomes a right corner of $\Theta_i$. We
then stop the construction and set $L = \ell-1$. See Figure \ref{fig2} for an
illustration.
%
\begin{figure}[b]

\includegraphics{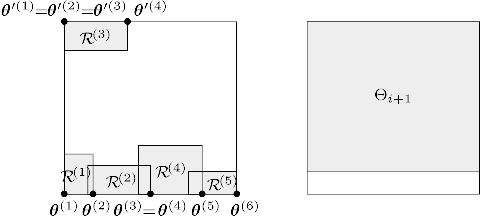}

\caption{Illustration when $L = 5$, $T (\ttheta^{(i)},\ttheta'^{(i)})
> 0$ for $i \in\{1,2,4,5\}$ and $T (\ttheta^{(3)},\ttheta'^{(3)}) < 0$.}
\label{fig2}
\end{figure}

\begin{remark*}
we define $\Theta_{i+1}$ as a rectangle to make the
procedure easier to implement in practice. Note that this rectangle
$\Theta_{i+1}$ is of the form $\Theta_{i+1} = [a_1,b_1] \times
[a_2',b_2']$ where $a_2',b_2'$ satisfy $b_2'-a_2' < b_2-a_2$.
We may also adapt the preceding ideas to build a confidence set $\Theta
_{i+1}$ of the form $\Theta_{i+1} = [a_1',b_1'] \times[a_2,b_2]$ where
$a_1',b_1'$ satisfy $b_1'-a_1' < b_1-a_1$.
\end{remark*}

We shall build the rectangles $\Theta_i$ until their diameters become
sufficiently small. The estimator we shall consider will then be the
center of the last rectangle built.

\subsection{Estimation procedure} \label{SectionEstimationGeneral}
\subsubsection{General scheme} \label{SectionGeneralScheme}
In this section, we aim at designing an estimator in dimension 2 or
higher. 
We build a finite sequence of nested rectangles $(\Theta_i)_{1 \leq i
\leq N}$ of $\R^d$ included in $\Theta$ by induction.
These rectangles can be interpreted as confidence sets for $\ttheta_0$
whenever it exists. We set $\Theta_1 = \Theta$. As long as the size of
$\Theta_i$ is large enough (in a suitable sense), we use an algorithm
(that we present below) to build $\Theta_{i+1}$ from $\Theta_i$. As
soon as the size of $\Theta_i$ becomes small enough, we stop the
construction of the rectangles. We then denote this last rectangle by
$\Theta_N$ and define our estimator ${\tttheta}$ as the center of
$\Theta_N$.

We now explain the general principle for constructing $\Theta_{i+1}$
from $\Theta_i$. Let $a_j, b_j$ be the numbers such that $\Theta_i=
\prod_{j=1}^d [a_j, b_j]$ and let $k$ be an integer in $\{1,\ldots,d\}$
to be specified later. The confidence set $\Theta_{i+1}$ will be of
the form
%
\begin{equation}
\label{eqDefPossible1Thetaiplus1} \Theta_{i+1}= \Biggl(\prod
_{j=1}^{k-1} [a_j, b_j]
\Biggr) \times \bigl[a_k',b'_k
\bigr] \times \Biggl(\prod_{j=k+1}^{d}
[a_j, b_j] \Biggr)
\end{equation}
with $a_k', b_k' \in[a_k, b_k]$ such that $b_k' - a_k' < b_k - a_k$.
In order to be a little more precise, let us consider $\kappa\in
(0,\bar{\kappa})$, and, for all $\ttheta, \ttheta' \in\Theta_i$,
let $
\mathcal{R} (\ttheta, \ttheta') $ be a rectangle included in $\Theta_i$
and containing a neighbourhood of $\ttheta$ (for the usual topology on
$\Theta_i$) such that
\[
\mathcal{R} \bigl(\ttheta, \ttheta'\bigr) \subset\B \bigl(\ttheta,
\kappa^{1/2} h \bigl({\ttheta}, {\ttheta'}\bigr) \bigr).
\]
Let $\mathcal{P}$ and $\mathcal{P}'$ be the two following opposite
faces of $\Theta_i$:
\begin{eqnarray*}
\mathcal{P} &=& \bigl\{(\theta_1,\ldots, \theta_{k-1},
a_k, \theta _{k+1},\ldots, \theta_d), \ttheta
\in\Theta_i \bigr\},
\\
\mathcal{P}' &=& \bigl\{(\theta_1,\ldots,
\theta_{k-1}, b_k, \theta _{k+1},\ldots,
\theta_d), \ttheta\in\Theta_i \bigr\}.
\end{eqnarray*}
As in Section~\ref{SectionHeuristiqueDim2}, the construction of
$\Theta
_{i+1}$ is based on the existence of $L+1$ elements $\ttheta^{(\ell)}
\in\mathcal{P}$ and $L+1$ elements $\ttheta'^{(\ell)} \in\mathcal
{P}'$ satisfying one of the two following relations:
%
\begin{equation}
\mathcal{P} = \bigcup_{\ell= 1}^L \bigl(
\mathcal{R}^{(\ell)} \cap \mathcal{P} \bigr)\quad \mbox{or}\quad \mathcal{P}'
= \bigcup_{\ell= 1}^L \bigl(
\mathcal{R}^{(\ell)} \cap \mathcal{P}' \bigr), \label{defDePoudePprime}
\end{equation}
where $\mathcal{R}^{(\ell)} $ stands for the set
\begin{eqnarray*}
\mathcal{R}^{(\ell)} = \cases{ \mathcal{R} \bigl(\ttheta^{(\ell)},
\ttheta'^{(\ell)}\bigr), & \quad $\mbox{if $T \bigl({
\ttheta^{(\ell)}},{\ttheta'^{(\ell)}}\bigr) > 0$},$
\vspace*{2pt}
\cr
\mathcal{R} \bigl(\ttheta'^{(\ell)},
\ttheta^{(\ell)}\bigr), &\quad  $\mbox{if $T \bigl({\ttheta^{(\ell)}},{
\ttheta'^{(\ell)}}\bigr) < 0$},$ \vspace*{2pt}
\cr
\mathcal{R}
\bigl(\ttheta^{(\ell)}, \ttheta'^{(\ell)}\bigr) \cup
\mathcal{R} \bigl(\ttheta'^{(\ell)}, \ttheta^{(\ell)}
\bigr), &\quad  $\mbox{if $T \bigl({\ttheta ^{(\ell
)}},{\ttheta'^{(\ell)}}
\bigr) = 0$.}$}
\end{eqnarray*}
Thanks to (\ref{defDePoudePprime}), there exist $a_k', b_k' \in[a_k,
b_k]$ such that $b_k' - a_k' < b_k - a_k$ and such that the rectangle
$\Theta_{i+1}$ defined by (\ref{eqDefPossible1Thetaiplus1}) satisfies
%
\begin{eqnarray}
\Theta_i {}\Big\backslash{}\bigcup_{\ell=1}^L
\mathcal{R}^{(\ell)} \subset \Theta_{i+1} \subset
\Theta_i. \label{RelatioNThetaiInclu}
\end{eqnarray}
The heuristics developed in Section~\ref{SectionHeuristique} show that
$\Theta_{i+1}$ may be interpreted as a confidence set for $\ttheta_0$
(whenever it exists). It remains to build $\Theta_{i+1}$ in a
constructive way.

\subsubsection{Construction of the confidence set \texorpdfstring{$\Theta_{i+1}$}{$Theta_{i+1}$} from \texorpdfstring{$\Theta_i$}{$Theta_i$}}
We present in this section an algorithm easy to code on a computer and
taking back the ideas of the preceding section to build $\Theta_{i+1}$
from $\Theta_i$. In what follows, it is convenient to introduce
positive numbers $\overline{r}_{\Theta_i,j} (\ttheta,\ttheta')$,
$\underline
{r}_{\Theta_i,j} (\ttheta,\ttheta')$ such that
\[
\mathcal{R} \bigl(\ttheta, \ttheta'\bigr) = \Theta_i
\cap\prod_{j=1}^d \bigl[
\theta_j - \underline{r}_{\Theta_i,j} \bigl(\ttheta,
\ttheta'\bigr), \theta _j+\overline{r}_{\Theta_i,j}
\bigl(\ttheta,\ttheta'\bigr) \bigr].
\]
We recall that this set must satisfy
%
\begin{equation}
\mathcal{R} \bigl(\ttheta, \ttheta'\bigr) \subset\B \bigl(\ttheta,
\kappa^{1/2} h \bigl({\ttheta}, {\ttheta'}\bigr) \bigr)
\label{eqInclusionRC1}.
\end{equation}
We also consider for all $j \in\{1,\ldots,d\}$, a number $\underline
{R}_{\Theta_i,j} \in[\underline{R}_j,+\infty)$ such that
%
\begin{equation}
\label{eqMinorationRCj} h^2 \bigl({\ttheta}, {\ttheta'} \bigr)
\geq\sup_{1 \leq j \leq d} \underline{R}_{\Theta_i,j} \bigl|
\theta_j - \theta'_j\bigr|^{\alpha_j}\qquad
\mbox {for all $\ttheta$, $\ttheta' \in\Theta_i$.}
\end{equation}
We finally consider for all $j \in\{1,\ldots,d\}$, a one-to-one map
$\psi_j$ from $\{1,\ldots,d-1\}$ into $\{1,\ldots,d\}\setminus\{j\}$.

We set $\Theta_1 =\Theta$. Given $\Theta_i$, we define $\Theta_{i+1}$
by using the algorithm below.
This algorithm ensues from the strategy described in the preceding
section. It defines $k$, builds the elements $\ttheta^{(\ell)}$,
$\ttheta'^{(\ell)}$ and, lastly returns $\Theta_{i+1}$.
\begin{algorithm}[H]
\caption{Definition of $\Theta_{i+1}$ from $\Theta_i$}
\label{algoConstructionDimQuelquonqueAvant}
\begin{algorithmic}[1]
\REQUIRE$\Theta_i = \prod_{j=1}^d [a_j, b_j]$
\STATE Choose ${k} \in\{1,\ldots,d\}$ such that
\[
\underline{R}_{\Theta_i, k} (b_{k} - a_{k})^{\alpha_k}
= \max_{1
\leq j
\leq d} \underline{R}_{\Theta_i,j} (b_{j}
- a_{j})^{\alpha_j}
\]
\STATE$\ttheta= (\theta_1,\ldots,\theta_d) \leftarrow(a_1,\ldots,a_d)$
\STATE$\ttheta' = (\theta_1',\ldots,\theta_d') \leftarrow
(a_1,\ldots
,a_{k-1},b_k,a_{k+1},\ldots,a_d)$
\STATE${\varrho_j} \leftarrow\overline{r}_{\Theta_i,j} (\ttheta
,\ttheta
')$ and $\varrho_j'\leftarrow\overline{r}_{\Theta_i,j} (\ttheta
',\ttheta)$
for all $j \in\{1,\ldots,d\} \setminus\{k\}$
\algstore{coupemonalgodimd}
\end{algorithmic}
\end{algorithm}
%
%
\begin{algorithm} 
%
\begin{algorithmic}[1]
\algrestore{coupemonalgodimd}
\STATE$\varrho_{k} \leftarrow(b_k - a_k)/2$ and $\varrho_k'
\leftarrow(b_k - a_k)/2$
\REPEAT
\STATE$\mathrm{Test} \leftarrow T(\ttheta,\ttheta') $
\STATE For all $j \in\{1,\ldots,d\}$, $\overline{r}_j \leftarrow\overline
{r}_{\Theta_i,j} (\ttheta,\ttheta') $, $\overline{r}_j' \leftarrow\overline
{r}_{\Theta_i,j} (\ttheta',\ttheta)$, $\underline{r}_j' \leftarrow
\underline{r}_{\Theta_i,j} (\ttheta',\ttheta)$
\IF{$\mathrm{Test} \geq0$}
\STATE$\varrho_{\psi_{k}(1)} \leftarrow\overline{r}_{\psi_k(1)} $
\STATE$\varrho_{\psi_k(j)} \leftarrow\min(\varrho_{\psi_k(j)},
\overline
{r}_{\psi_k(j)})$ for all $j \in\{2,\ldots, d-1\}$
\STATE$\varrho_{k} \leftarrow\min(\varrho_{k}, \overline{r}_{k})$
\STATE$J \leftarrow \{1 \leq j \leq d -1,  \theta_{\psi_{k}(j)}
+ \varrho_{\psi_{k}(j)} < b_{\psi_{k}(j)}  \}$
\IF{$J \neq\varnothing$}
\STATE$\mathfrak{j}_{\mathrm{min}} \leftarrow\min J$
\STATE$\theta_{\psi_{k}(j)} \leftarrow a_{\psi_{k}(j)}$ for all $j
\leq\mathfrak{j}_{\mathrm{min}} - 1$
\STATE$\theta_{\psi_{k}(\mathfrak{j}_{\mathrm{min}})} \leftarrow
\theta
_{\psi_{k}(\mathfrak{j}_{\mathrm{min}} )} + \varrho_{\psi
_{k}(\mathfrak
{j}_{\mathrm{min}})}$
\ELSE
\STATE$\mathfrak{j}_{\mathrm{min}} \leftarrow d$
\ENDIF
\ENDIF
\IF{$\mathrm{Test} \leq0$}
\STATE$\varrho_{\psi_{k}(1)}' \leftarrow\overline{r}'_{\psi_k(1)}$
\STATE$\varrho_{\psi_{k}(j)}' \leftarrow\min(\varrho_{\psi_{k}(j)}',
\overline{r}'_{\psi_{k}(j)})$ for all $j \in\{2,\ldots, d-1\}$
\STATE$\varrho_{k}' \leftarrow\min(\varrho_{k}', \underline{r}'_{k})$
\STATE$J' \leftarrow \{1 \leq j \leq d -1,  \theta_{\psi
_{k}(j)}' + \varrho_{\psi_{k}(j)}' < b_{\psi_{k}(j)} \}$
\IF{$J' \neq\varnothing$}
\STATE$\mathfrak{j}_{\mathrm{min}}' \leftarrow\min J'$
\STATE$\theta_{\psi_{k}(j)}' \leftarrow a_{\psi_{k}(j)}$ for all $j
\leq\mathfrak{j}_{\mathrm{min}}' - 1$
\STATE$\theta_{\psi_{k}(\mathfrak{j}_{\mathrm{min}} ')}' \leftarrow
\theta_{\psi_{k}(\mathfrak{j}_{\mathrm{min}} ')}' + \varrho_{\psi
_{k}(\mathfrak{j}_{\mathrm{min}} ')}'$
\ELSE
\STATE$\mathfrak{j}_{\mathrm{min}}' \leftarrow d$
\ENDIF
\ENDIF
\UNTIL{$\mathfrak{j}_{\mathrm{min}} = d$ or $\mathfrak{j}_{\mathrm
{min}} '
= d$}
\IF{$ \mathfrak{j}_{\mathrm{min}} = d$}
\STATE$a_{k} \leftarrow a_{k} + \varrho_{k}$
\ENDIF
\IF{ $\mathfrak{j}_{\mathrm{min}} ' = d$}
\STATE$b_{k} \leftarrow b_{k} - \varrho_{k}' $
\ENDIF
\STATE$\Theta_{i+1} \leftarrow\prod_{j=1}^{d} [a_j, b_j] $
\RETURN$\Theta_{i+1}$
\end{algorithmic}
\end{algorithm}
The parameters $\kappa$, $t_j$, $\overline{r}_{\Theta_i,j} (\ttheta
,\ttheta
')$, $\underline{{r}}_{\Theta_i,j} (\ttheta,\ttheta')$ can be
interpreted as in dimension $1$. We have introduced a new parameter
$\underline{R}_{\Theta_i,j}$ whose role is to control more accurately
the Hellinger distance in order to define $k$. Sometimes, the
computation of this parameter is difficult in practice. In this case,
we can overcome this issue by remarking that for all $\ttheta, \ttheta'
\in\Theta$,
\[
h^2 \bigl({\ttheta}, {\ttheta'} \bigr) \geq\sup
_{1 \leq j \leq d} \underline{R} \bigl|\theta_j -
\theta'_j\bigr|^{\alpha_j}\qquad \mbox{with $\underline{R } =
\min_{1 \leq j \leq d} \underline{R}_j$,}
\]
which means that we can always assume that $\underline{R}_j$ is
independent of $j$. Choosing $\underline{R}_{\Theta_i, j} =
\underline
{R}$ then simplifies the only line where this parameter is involved
(line 1). It becomes
$ (b_{k} - a_{k})^{\alpha_k} = \max_{1 \leq j \leq d} (b_{j} -
a_{j})^{\alpha_j}$
and $k$ can be calculated without computing $\underline{R}$.

\subsubsection{Construction of the estimator}
As explained in Section~\ref{SectionGeneralScheme}, we only build a
finite number of rectangles $\Theta_i$ in order to define our
estimator. We stop their construction when they become small enough.
More precisely, we consider $d$ positive numbers $\eta_1,\ldots,\eta_d$
and use the following algorithm to
design $\tttheta$.\vspace*{-3pt}
\begin{algorithm}[h]
\caption{Construction of the estimator}
\label{algoConstructionDimQuelquonque}
\begin{algorithmic}[1]
\STATE Set $a_j = m_j$ and $b_j = M_j$ for all $j \in\{1,\ldots,d\}$
\STATE$i \leftarrow0$
\WHILE{there exists $j \in\{1,\ldots,d\}$ such that $ b_j - a_j >
\eta_j$}
\STATE$i \leftarrow i + 1$
\STATE Build $\Theta_i$ and set $a_1,\ldots,a_d$, $b_1,\ldots,b_d$ such
that $\prod_{j=1}^d [a_j,b_j] = \Theta_i$
\ENDWHILE
\RETURN
\[
\tttheta= \biggl(\frac{a_1 + b_1}{2}, \ldots, \frac{a_d + b_d}{2} \biggr)
\]
\end{algorithmic}\vspace*{-3pt}
\end{algorithm}

The convergence of the two preceding algorithms is guaranteed under
mild conditions on $\overline{r}_{\Theta_i,j} (\ttheta,\ttheta')$ and
$\underline{r}_{\Theta_i,j} (\ttheta,\ttheta')$. We refer to
Section~\ref{SectionDefinitioNRCDimD} for more details on this point.

\subsection{A risk bound} \label{SectionPropEstimateurDimD}
The risk of $\tttheta$ can be bounded as soon as the preceding
parameters are suitably chosen.

\begin{thmm} \label{54515215415152151514}
Suppose that Assumption~\ref{HypSurLeModeleQuelquonqueDebutDimD} holds
with $d \geq2$. Let $\bar{\kappa}$ be defined by (\ref
{eqEsperanceTest}), and assume that $\kappa\in(0, \bar{\kappa})$.
Suppose that for all $j \in\{1,\ldots,d\}$, $t_j \in(0,d^{1/\alpha_j}]$,
$\varepsilon_j = t_j (\overline{R}_j n)^{-1/\alpha_j}$, and $\eta_j
\in
(0, d^{1/\alpha_j} (\overline{R}_j n)^{-1/\alpha_j}]$.
Suppose moreover that for all $i$, $\ttheta, \ttheta' \in\Theta_i$,
the numbers $\overline{r}_{\Theta_i,j} (\ttheta,\ttheta')$, $\underline
{{r}}_{\Theta_i,j} (\ttheta,\ttheta')$,
are such that (\ref{eqInclusionRC1}) holds and that the two preceding
algorithms converge.

Then, for all $\xi> 0$, the estimator $\tttheta$ derived from
Algorithm \ref{algoConstructionDimQuelquonque} satisfies
\[
\P \biggl[ C h^2(s,f_{\tttheta}) \geq h^2(s, \F) +
\frac{D_{\F}}{n} + \xi \biggr] \leq \mathrm{e}^{- n \xi},
\]
where $D_{\F} = d \vee\sum_{j=1}^d \log (1 + t_j^{-1}  ( (d/
\bar{\bolds{\alpha}}) (c \overline{R}_j / \underline{R}_j)
)^{1/\alpha_j} )$, with $c$ depending only on $\kappa$, $
\bar{\bolds{\alpha}}$ the harmonic mean of ${\bolds{\alpha}}$, and where
$C > 0$ depends only on $\kappa$ and $(\overline{R}_j/\underline
{R}_j)_{1 \leq j \leq d}$.
\end{thmm}

We can also prove that the estimator is asymptotically very close to
the m.l.e. when the model $\F$ is regular enough and contains $s$. We
refer to Theorem~\ref{thmGeneralLienAvecMLE} in Section~\ref{subsectionThmprincipalMLE}.

\subsection{Choice of \texorpdfstring{$\overline{r}_{\Theta_i,j} (\ttheta,\ttheta')$}{$\overline{r}_{Theta_i,j} (\bolds{theta},\bolds{theta}')$} and
\texorpdfstring{$\underline{{r}}_{\Theta_i,j} (\ttheta,\ttheta')$}{$\underline{{r}}_{Theta_i,j}(\bolds{theta},\bolds{theta}')$}} \label
{SectionDefinitioNRCDimD}
We now briefly discuss the choice of the parameters $\overline{{r}}_{\Theta
_i,j} (\ttheta,\ttheta')$, $\underline{{r}}_{\Theta_i,j} (\ttheta
,\ttheta')$. Note that they must be calculated in practice since they
are involved in Algorithm \ref{algoConstructionDimQuelquonqueAvant}. %
It turns out that the two preceding algorithms converge and that the
numerical complexity of the estimation procedure can be theoretically
upper bounded when $\overline{r}_{\Theta_i,j} (\ttheta,\ttheta')$ and
$\underline{{r}}_{\Theta_i,j} (\ttheta,\ttheta')$ are larger than
\[
\label{DefinitionRDimensionQuelquonque1} \Bigl( {\kappa} \sup_{1 \leq k \leq d} \bigl\{ (\underline
{R}_{k}/\overline{R}_{j} )\bigl |\theta'_k
- \theta_k\bigr|^{\alpha
_k} \bigr\} \Bigr)^{1/\alpha_j},
\]
which is in particular true when they are larger than $  ( (\kappa
/ \overline{R}_{j}) h^2(f_{\ttheta}, f_{\ttheta'})  )^{1/\alpha_j}$.
This bound may be found in Proposition~6 of Chapter~6 of Sart \cite
{SartThese} (it is omitted here to reduce the size of the paper).
Besides, the larger $\overline{{r}}_{\Theta_i,j} (\ttheta,\ttheta')$ and
$\underline{{r}}_{\Theta_i,j} (\ttheta,\ttheta')$, the faster the
convergence of the two algorithms.
They should therefore be as large as possible so that (\ref
{eqInclusionRC1}) holds. Note that changing the values of $\overline
{{r}}_{\Theta_i,j} (\ttheta,\ttheta')$ and $\underline{{r}}_{\Theta
_i,j} (\ttheta,\ttheta')$ may influence the value of the
estimator {$\tttheta$} but does not modify its theoretical properties.

We refer to Sections 6 and 8 of Chapter~6 of Sart \cite{SartThese} for
numerical simulations (the results are similar to dimension one) as
well as for more information on the practical implementation of the procedure.

\section{Proofs} \label{SectionProofs}

\subsection{Preliminary results on the estimation procedure} \label
{SectionGeneralTheoreticalResults}
The estimators we have built in the preceding sections were based on
particular sequences of subsets $(\Theta_i)_i$ of $\R^d$ that could be
interpreted as confidence sets for the true parameter $\ttheta_0$
whenever it exists. In this section, we make explicit the assumptions
we need to consider on the $(\Theta_i)$ in order to ensure that the
resulting estimator possesses good statistical properties.

The results of this section simultaneously cover the cases of models
indexed by a one-dimensional parameter (that is, $d = 1$) and those
indexed by a multi-dimensional parameter (that is, $d \geq2$). They
will allow us to prove the theoretical properties of the estimators
considered in the preceding sections.

\subsubsection{A risk bound} \label{subsectionThmprincipal}
%
\begin{thmm} \label{ThmPrincipal}
Suppose that Assumption~\ref{HypSurLeModeleQuelquonqueDebutDimD} holds.
Let $\kappa\in(0,\bar{\kappa})$, and let $\Theta_1 \cdots\Theta_N$
be $N$ non-empty subsets of $\Theta$ such that $\Theta_1 =\Theta$. For
all $j \in\{1,\ldots,d\}$, let $t_j$ be an arbitrary number in
$(0,d^{1/\alpha_j}]$ and $\varepsilon_j = t_j (\overline{R}_j
n)^{-1/\alpha_j}$.
Assume that for all $i \in\{1,\ldots, N-1\}$, there exists
$L_i \geq1$ such that for all $\ell\in\{1,\ldots,L_i\}$, there exist
two elements $\ttheta^{(i,\ell)} \neq\ttheta'^{(i,\ell)} $ of
$\Theta
_i$ such that
%
\begin{equation}
\label{DefintiionBiell} \Theta_i {}\Big\backslash{}\bigcup
_{\ell=1}^{L_i} B^{(i,\ell)} \subset \Theta
_{i+1} \subset\Theta_i,
\end{equation}
where $B^{(i,\ell)} $ is the set defined by
\begin{eqnarray*}
B^{(i,\ell)} = \cases{ \B \bigl(\ttheta^{(i,\ell)}, r_{i,\ell}
\bigr), & \quad $\mbox{if $T \bigl({\ttheta ^{(i,\ell)}},{\ttheta'^{(i,\ell)}}
\bigr) > 0$},$ \vspace*{2pt}
\cr
\B \bigl(\ttheta'^{(i,\ell)},
r_{i,\ell} \bigr), &\quad  $\mbox{if $T \bigl({\ttheta ^{(i,\ell)}},{
\ttheta'^{(i,\ell)}}\bigr) < 0$},$ \vspace*{2pt}
\cr
\B \bigl(
\ttheta^{(i,\ell)}, r_{i,\ell} \bigr) \cup\B \bigl(\ttheta
'^{(i,\ell)}, r_{i,\ell} \bigr),&\quad $\mbox{if $T \bigl({
\ttheta^{(i,\ell
)}},{\ttheta'^{(i,\ell)}}\bigr) = 0$},$}
\end{eqnarray*}
where $r_{i,\ell}^2 = \kappa h^2 ({\ttheta^{(i,\ell)}}, {\ttheta
'^{(i,\ell)}})$ and $T$ the functional defined by (\ref
{eqDefinitionTDimD}). 
Let $\ttheta_0$ be an arbitrary element of $\Theta$ such that
\[
h^2(s,f_{\ttheta_0}) \leq h^2(s,\F) + 1/n
\]
and $\delta$ be a non-negative map from $\Theta^2$ such that $\delta^2
(\ttheta,\ttheta) = 0$ for all $\ttheta\in\Theta$ and
%
\begin{equation}
\label{defdeDelta} \sup_{\ttheta,\ttheta' \in\Theta_i} \delta^2 \bigl(\ttheta,
\ttheta'\bigr) \leq \inf_{1 \leq\ell\leq L_i} h^2
\bigl({\ttheta^{(i,\ell)}},{\ttheta '^{(i,\ell
)}}\bigr)\qquad
\mbox{for all $ i \in\{1,\ldots, N\}$}.
\end{equation}
Then, for all $\xi> 0$,
\[
\P \biggl[ C \inf_{\ttheta\in\Theta_N} \delta^2 (
\ttheta_0, \ttheta) \geq h^2(s, \F) + \frac{D_{\F}^{(n)}}{n}
+ \xi \biggr] \leq \mathrm{e}^{- n
\xi},
\]
where $C > 0$ depends only on $\kappa$ and where
\[
D_{\F}^{(n)} = \max \Biggl\{d, \sum
_{j=1}^d \log \bigl(1 + t_j^{-1}
\bigl( (d/ \bar{\bolds{\alpha}}) \bigl(c \overline{R}_j /
\underline {R}_j^{(n)}\bigr) \bigr)^{1/\alpha_j} \bigr)
\Biggr\}.
\]
In the definition of $D_{\F}^{(n)}$, $\bar{\bolds{\alpha}}$ is the
harmonic mean of $\bolds{{\alpha}}$, $c$ depends only on $\kappa$,
and $\underline{R}^{(n)}_j$ is any positive number
such that $\underline{R}^{(n)}_j \geq\underline{R}_j$ and such that
\[
h^2\bigl({\ttheta}, {\ttheta'}\bigr) \geq\sup
_{1 \leq j \leq d} \underline {R}_j^{(n)} \bigl|
\theta_j - \theta'_j\bigr|^{\alpha_j}
\]
for all $\ttheta, \ttheta' \in\Theta$ satisfying $h^2({\ttheta
},{\ttheta'}) \leq\frac{c}{n}\sum_{j=1}^d \log (1 + t_j^{-1}
(M_j-m_j)  (\overline{R}_j n )^{1/\alpha_j}  )$.
\end{thmm}

\begin{remark*}
In this theorem, the sets $(\Theta_i)$, the numbers
$(L_i)$ and $N$ as well as the elements $\ttheta^{(i,\ell)}$,
$\ttheta
'^{(i,\ell)}$ may be random.
\end{remark*}

This theorem implies Theorem~\ref{54515215415152151514}. Indeed,
its assumptions are fulfilled when $ \underline{R}_j^{(n)} =
\underline
{R}_j$, when the $(\Theta_i)$ are those provided by Algorithm \ref
{algoConstructionDimQuelquonqueAvant}, when the elements $\ttheta
^{(i,\ell)}$ and $\ttheta'^{(i,\ell)}$ correspond to those defined in
Section~\ref{SectionEstimationGeneral} (the index $i$ has been omitted
in that section for ease of reading), and when $\delta^2$ is defined by
\[
\delta^2 \bigl(\ttheta, \ttheta'\bigr) = \sup
_{j \in\{1,\ldots,d\}} \underline {R}_j \bigl|\theta_j -
\theta'_j\bigr|^{\alpha_j}.
\]
%
The fact that (\ref{DefintiionBiell}) holds follows from the fact that
$\Theta_{i+1}$ has been built in such a way that (\ref
{RelatioNThetaiInclu}) holds. However, this point has only be claimed
and has not been proved. Its rigorous proof is quite long and is
therefore postponed to the \hyperref[SectionAnnexe1]{Appendix}. The fact that
(\ref{defdeDelta}) holds follows from the choice of $k$ in
Algorithm \ref{algoConstructionDimQuelquonqueAvant};
see the \hyperref[SectionAnnexe1]{Appendix}.

The above theorem then asserts that
\[
\P \biggl[ C \inf_{\ttheta\in\Theta_N} \sup_{j \in\{1,\ldots,d\}}
\underline{R}_j |\theta_{0,j} - {\theta}_j|^{\alpha_j}
\geq h^2(s, \F) + \frac{D_{\F}}{n} + \xi \biggr] \leq
\mathrm{e}^{- n \xi},
\]
where $C$ depends only on $\kappa$ and where $D_{\F}$ is defined in
Theorem~\ref{54515215415152151514}.
By using the triangular inequality, Assumption~\ref
{HypSurLeModeleQuelquonqueDebutDimD}, and the fact that the estimator
$\tttheta$ of Theorem~\ref{54515215415152151514} is very close to
any element $\ttheta$ of $\Theta_N$ (since its size is very small), we
finally get
\begin{eqnarray*}
\P \biggl[ C' h^2(s,f_{{\tttheta}})
\geq h^2(s, \F) + \frac{D_{\F
}}{n} + \xi \biggr] \leq
\mathrm{e}^{- n \xi}, 
\end{eqnarray*}
where $C'$ depends on $\kappa$ and $ \sup_{j \in\{1,\ldots,d\}}
\overline{R}_j /\underline{R}_j $.

\begin{remark*} In some models, $\underline{R}_j^{(n)}$ can be
chosen much larger than $\underline{R}_j$. This refinement is omitted
in Theorems \ref{ThmPrincipalDim1} and \ref{54515215415152151514}
for ease of presentation.
\end{remark*}

\subsubsection{Connection with the maximum likelihood estimator}
\label{subsectionThmprincipalMLE}
In this section, we carry out the general result that link our
estimator to the maximum likelihood one. In particular, it manages to
deal with multi-dimensional parametric models.

We need to introduce the following notation. We define $\mathring
{\Theta}$ as the interior of $\Theta$ and $l_{\ttheta} (x) = \log
f_{\ttheta} (x)$. The gradient of the map $\ttheta\mapsto\log
f_{\ttheta} (x)$ is denoted by $\dot{l}_{\ttheta} (x)$ and its Hessian
matrix by $\ddot{l}_{\ttheta} (x)$. The notation $(\cdot)^T$ represents
the transpose of a vector or a matrix. The Euclidean norm and its
induced matrix norm are both denoted by $\|\cdot\|$. We denote the log
likelihood by $L(\ttheta) = n^{-1} \sum_{i=1}^n l_{\theta} (X_i)$.

\begin{hyp} \label{AssumptionPourMLE}
The following conditions are satisfied:
\begin{longlist}[(iii)]
\item[(i)] Assumption~\ref{HypSurLeModeleQuelquonqueDebutDimD} holds with
$\alpha_1 = \cdots= \alpha_d = 2$ and there exists $\ttheta_0 \in
\mathring{\Theta}$ such that $s =   f_{\ttheta_0} \in\F$.
\item[(ii)]$\F$ and $\kappa$ do not depend on $n$. The $t_j$ depend on $n$
(one then write $t_j^{(n)}$ in place of $t_j$) and are chosen in such a
way that {$|\log t_j^{(n)}|/n$} tends to $0$ when $n$ goes to infinity.
\item[(iii)] For $\mu$-almost all $x \in\XX$, the mapping $\ttheta\mapsto
f_{\ttheta} (x) $ is positive and two times differentiable on
$\mathring
{\Theta}$.
\item[(iv)] The Fisher information matrix
\[
I(\ttheta) = \int_{\XX} \bigl(\dot{l}_{\ttheta} (x)
\bigr) \bigl(\dot {l}_{\ttheta} (x) \bigr)^T f_{\ttheta}
(x) \,\d\mu(x)
\]
exists for all $\ttheta\in\mathring{\Theta}$. Moreover, the map
$\ttheta\mapsto I(\ttheta)$ is continuous and non-singular at
$\ttheta_0$.
\item[(v)]
The integrals $\int_{\XX} \dot
{f}_{\ttheta
_0} (x) \,\d\mu(x)$, $\int_{\XX} \ddot{f}_{\ttheta_0} (x) \,\d\mu(x)$
exist and are zero.
\item[(vi)]
For all $\vartheta> 0$, there
exist a neighbourhood $\Theta_0 (\vartheta)$ of $\ttheta_0$
(independent of $n$) and an event $\mathcal{A}_n (\vartheta)$ on which
\begin{eqnarray*}
\sup_{\ttheta, \ttheta' \in\Theta_0 (\vartheta)} \frac{1}{n} \sum
_{i=1}^n \frac{\llvert  \log f_{\ttheta} (X_i) - \log f_{\ttheta'} (X_i)
\rrvert ^3}{\llVert \ttheta- \ttheta'\rrVert ^2} &\leq& \vartheta,
\\
\frac{1}{n} \sum_{i=1}^n \sup
_{\ttheta\in\Theta_0 (\vartheta)} \bigl\llVert \ddot{l}_{\ttheta}
(X_i) - \ddot{l}_{\ttheta_0} (X_i) \bigr\rrVert &
\leq & \vartheta.
\end{eqnarray*}
Moreover, the maps $\vartheta\mapsto\Theta_0 (\vartheta)$ and
$\vartheta\mapsto\mathcal{A}_n (\vartheta)$ are non-decreasing (in
the sense of set inclusion).
\end{longlist}
\end{hyp}

\begin{thmm}\label{thmGeneralLienAvecMLE}
Suppose that Assumption~\ref{AssumptionPourMLE} is fulfilled.
Let $\delta^2$ be a function satisfying the assumptions of
Theorem~\ref
{ThmPrincipal}.
Let, for each $n \in\N^{\star}$, $N_n$ be a (possibly random) positive
integer and $\Theta_1,\ldots,\Theta_{N_n}$ be (random) subsets
satisfying the assumptions of Theorem~\ref{ThmPrincipal}.
Let, for all $\vartheta> 0$, $\mathcal{A}_n' (\vartheta)$ be the
event on which
\begin{eqnarray*}
\Biggl\llVert \frac{1}{n} \sum_{i=1}^n
\dot{l}_{{\ttheta}_0} (X_i) \Biggr\rrVert \leq\vartheta\quad
\mbox{and} \quad\Biggl\llVert \frac{1}{n} \sum_{i=1}^n
\bigl( \ddot{l}_{\ttheta_0} (X_i) - \E \bigl[
\ddot{l}_{\ttheta_0} (X_i) \bigr] \bigr) \Biggr\rrVert \leq
\vartheta.
\end{eqnarray*}
Then there exist $\vartheta> 0$, $\xi> 0$, $n_0 \in\N^{\star}$ such
that for all $n \geq n_0$:
%
\begin{eqnarray}
\label{termeDroiteMLE} & & \P \Biggl[\exists\tilde{\ttheta} \in\mathring {\Theta }, \sum
_{i=1}^n \dot{l}_{\tilde{\ttheta}}
(X_i) = 0 \mbox{ and } \inf_{\ttheta\in\Theta_{N_n}}
\delta^2 (\ttheta, \tilde {\ttheta} ) \leq\frac{120 }{n} \sup
_{j \in\{1,\ldots,d\}} \bigl(t_j^{(n)}
\bigr)^{2} \Biggr]
\nonumber
\\[-8pt]
\\[-8pt]
\nonumber
&&\quad
\geq  1 - \bigl\{\P \bigl[ \bigl(\mathcal{A}_n (\vartheta)
\bigr)^c \bigr] + \P \bigl[ \bigl( \mathcal{A}_n'
(\vartheta) \bigr)^c \bigr] + \mathrm{e}^{-n
\xi} \bigr\}.
\end{eqnarray}
\end{thmm}

Remark that the law of large numbers implies that $\P [ \mathcal
{A}_n' (\vartheta)  ] $ converges to $1$ when $n$ goes to
infinity. The right-hand side of inequality (\ref{termeDroiteMLE})
tends therefore to $1$ as soon as $\P [ \mathcal{A}_n (\vartheta)
 ]$ converges to $1$. Moreover, under suitable assumptions, the
rate of convergence of $\P [ \mathcal{A}_n (\vartheta)  ]$
and $\P [ \mathcal{A}_n' (\vartheta)  ]$ to $1$ can be
specified as in Theorem~\ref{thmLienMLEDim1}.

\subsection{Proof of Theorem \texorpdfstring{\protect\ref{ThmPrincipal}}{5.1}}
Let $G\dvtx (1/\sqrt{2}, 1) \rightarrow(3 + 2\sqrt{2}, +\infty)$ be the
bijection defined by
\[
G (x) = \frac{ (1 + \min ( {(1 - x)}/{2}, x -
{1}/{\sqrt
{2}}  )  )^4 (1 + x) + \min ( {(1 - x)}/{2}, x -
{1}/{\sqrt{2}}  ) }{1 - x - \min ( {(1 - x)}/{2}, x
-{1}/{\sqrt{2}}  )}.
\]
Let $C_{\kappa}$ be such that $(1 + \sqrt{C_{\kappa}})^2 = \kappa
^{-1}$. Since $\kappa\in(0,\bar{\kappa})$, $C_{\kappa} \in(3 +
2\sqrt
{2}, +\infty)$ and there exists thus $\upsilon\in(1/\sqrt{2}, 1)$
such that $G (\upsilon) = C_{\kappa}$.
We then set
%
\begin{eqnarray} \label{defdecPreuveMLE}
c &=& 24 \bigl(2 + \sqrt{2}/6 (\upsilon- 1/\sqrt{2} ) \bigr) / (\upsilon- 1/
\sqrt{2} )^2 \cdot10^{3},\nonumber
\\
\beta_1 &=& \min \bigl\{ (1 - \upsilon)/{2}, \upsilon- 1/{\sqrt{2}}
\bigr\},
\nonumber
\\[-8pt]
\\[-8pt]
\nonumber
\beta_2 &=& (1+\beta_1) \bigl(1 + \beta_1^{-1}
\bigr) \bigl[1-\upsilon+ (1 +\beta _1) (1+\upsilon) \bigr],
\\
\beta_3 &=& \bigl(1+\beta_1^{-1}\bigr) \bigl[1
- \upsilon+ (1+\beta_1)^3 (1+\upsilon) \bigr] + c (1 +
\beta_1)^2.
\nonumber
\end{eqnarray}
We need the following claim, which will be proved immediately after the
present proof.

\begin{Claim} \label{ClaimOmegaXi}
For all $\xi> 0$, there exists an event $\Omega_{\xi}$ such that $\P
(\Omega_{\xi}) \geq1 - \mathrm{e}^{-n \xi}$ and on which, for all $f,f' \in
\Fd$,
\[
(1 - \upsilon ) h^2\bigl(s,f' \bigr) +
\frac{\overline{T} (f,f')
}{\sqrt{2}} \leq (1 + \upsilon ) h^2(s,f ) + c
\frac{
 (
D_{\F}^{(n)} + n \xi )}{n},
\]
where $D_{\F}^{(n)}$ is defined in Theorem~\ref{ThmPrincipal} for the
value of $c > 0$ given by (\ref{defdecPreuveMLE}).
\end{Claim}

We begin by proving the following lemma.

\begin{lemme} \label{LemmeControleH}
For all $\xi> 0$, the following assertion holds on $\Omega_{\xi}$:
if there exist $p \in\{1,\ldots,N-1\}$ and $\ell\in\{1,\ldots,L_p\}$
such that $\ttheta_0 \in\Theta_p$ and such that
%
\begin{eqnarray}
\label{eqControleH} \beta_2 h^2(s, f_{\ttheta_0}) +
\beta_3 \biggl(\frac{D_{\F
}^{(n)}}{n} + \xi \biggr) < \beta_1
\bigl(h^2(f_{\ttheta_0},f_{\ttheta^{(p,\ell
)}}) + h^2(f_{\ttheta_0},f_{\ttheta'^{(p,\ell)}})
\bigr),
\end{eqnarray}
then $\ttheta_0 \notin B^{(p,\ell)}$.
\end{lemme}

\begin{pf}
Without loss of generality, we may assume that $T ({\ttheta^{(p,\ell
)}},{\ttheta'^{(p,\ell)}}) = \overline{T}(f_{\pi(\ttheta^{(p,\ell
) })},
f_{\pi( \ttheta'^{(p,\ell)})} )$ is non-negative, and prove that
$\ttheta_0 \notin\B (\ttheta^{(p,\ell)}, r_{p,\ell}  )$.
On the event $\Omega_{\xi}$, we deduce from the claim that
\[
(1 - \upsilon ) h^2(s,f_{\pi(\ttheta'^{(p,\ell)})} ) \leq (1 + \upsilon )
h^2(s,f_{\pi(\ttheta^{(p,\ell)})} ) + c \frac
{  ( D_{\F}^{(n)} + n \xi )}{n}.
\]
Consequently, by using the triangular inequality and the above inequality
\begin{eqnarray*}
(1 - \upsilon ) h^2(f_{\ttheta_0},f_{\pi(\ttheta
'^{(p,\ell
)})} ) &\leq&
\bigl(1 + \beta_1^{-1} \bigr) (1 - \upsilon )
h^2(s,f_{\ttheta_{0}} )
\\
& &{} + (1 + \beta_1) (1 - \upsilon ) h^2(s,f_{\pi
(\ttheta'^{(p,\ell)})}
)
\\
&\leq& \bigl(1 + \beta_1^{-1} \bigr) (1 - \upsilon )
h^2(s,f_{\ttheta_{0}} )
\\
& &{} + (1 + \beta_1) \biggl[ (1 + \upsilon ) h^2(s,f_{\pi
(\ttheta^{(p,\ell)})}
) + c \frac{  ( D_{\F}^{(n)} + n \xi )}{n} \biggr].
\end{eqnarray*}
Since $h^2(s,f_{\pi(\ttheta^{(p,\ell)})} ) \leq(1+\beta_1^{-1})
h^2(s,f_{\ttheta_0}) + (1+\beta_1) h^2(f_{\ttheta_0},f_{\pi(\ttheta
^{(p,\ell)})})$,
%
\begin{eqnarray}
\label{eqPreuveThmPrincipalm} (1 - \upsilon ) h^2(f_{\ttheta_0},f_{\pi(\ttheta
'^{(p,\ell
)})}
)
 &\leq& \bigl(1 + \beta_1^{-1}\bigr) \bigl[1-\upsilon+ (1
+\beta_1) (1+\upsilon) \bigr] h^2(s, f_{\ttheta_0})\nonumber
\\
& &{} + (1+\beta_1)^2 (1 + \upsilon) h^2(f_{\ttheta_0},f_{\pi
(\ttheta^{(p,\ell)})})
\\
& &{} + \frac{c (1 + \beta_1)  ( D_{\F}^{(n)} + n \xi
)}{n}.
\nonumber
\end{eqnarray}
Remark now that for all $\ttheta\in\Theta$,
\[
h^2(f_{\ttheta},f_{\pi(\ttheta)} ) \leq\sup
_{1 \leq j \leq d} \overline {R}_j \varepsilon_j^{\alpha_j}
\leq d/n.
\]
By using the triangular inequality,
\begin{eqnarray*}
h^2(f_{\ttheta_0},f_{\pi(\ttheta^{(p,\ell)})} ) &\leq& (1+
\beta_1) h^2(f_{\ttheta_0},f_{\ttheta^{(p,\ell)}}) + d
\bigl(1 + \beta_1^{-1}\bigr)/n,
\\
h^2(f_{\ttheta_0},f_{\ttheta'^{(p,\ell)}} ) &\leq& (1+
\beta_1) h^2(f_{\ttheta_0},f_{\pi(\ttheta'^{(p,\ell)})}) + d
\bigl(1+\beta_1^{-1}\bigr)/n.
\end{eqnarray*}
We deduce from these two inequalities and from (\ref
{eqPreuveThmPrincipalm}) that
\begin{eqnarray*}
(1 - \upsilon ) h^2(f_{\ttheta_0},f_{\ttheta'^{(p,\ell
)}} ) &\leq&
\beta_2 h^2(s, f_{\ttheta_0}) + (1+\beta
_1)^4 (1 + \upsilon) h^2(f_{\ttheta_0},f_{\ttheta^{(p,\ell)}})
\\
& & {}+ \frac{d (1+\beta_1^{-1})  [1 -
\upsilon
+ (1+\beta_1)^3 (1+\upsilon)  ] + c (1 + \beta_1)^2  (
D_{\F
}^{(n)} + n \xi )}{n}.
\end{eqnarray*}
Since $D_{\F}^{(n)} \geq d$ and $\beta_3 \geq1$,
\begin{eqnarray*}
(1 - \upsilon ) h^2(f_{\ttheta_0},f_{\ttheta'^{(p,\ell
)}} ) &\leq&
\beta_2 h^2(s, f_{\ttheta_0}) + \frac{\beta_3  ( D_{\F}^{(n)}
+ n \xi )}{n}
\\
& &{} + (1+\beta_1)^4 (1 + \upsilon) h^2(f_{\ttheta
_0},f_{\ttheta
^{(p,\ell)}}).
\end{eqnarray*}
By using (\ref{eqControleH}),
\begin{eqnarray*}
(1 - \upsilon ) h^2(f_{\ttheta_0},f_{\ttheta'^{(p,\ell
)}} ) &<&
\beta_1 \bigl(h^2(f_{\ttheta_0},f_{\ttheta^{(p,\ell)}}) +
h^2(f_{\ttheta_0},f_{\ttheta'^{(p,\ell)}}) \bigr)
\\
& &{} + (1+\beta_1)^4 (1 + \upsilon) h^2(f_{\ttheta
_0},f_{\ttheta
^{(p,\ell)}})
\end{eqnarray*}
and thus
\begin{eqnarray*}
h^2(f_{\ttheta_0},f_{\ttheta'^{(p,\ell)}} ) &<& G (\upsilon)
h^2(f_{\ttheta_0},f_{\ttheta^{(p,\ell)}} )
\\
&<& C_{\kappa} h^2(f_{\ttheta_0},f_{\ttheta^{(p,\ell)}} ).
\end{eqnarray*}
Finally,
\begin{eqnarray*}
h^2(f_{\ttheta^{(p,\ell)}},f_{\ttheta'^{(p,\ell)}} ) &\leq&
\bigl(h(f_{\ttheta_0},f_{\ttheta^{(p,\ell)}} ) + h(f_{\ttheta
_0},f_{\ttheta
'^{(p,\ell)}}
) \bigr)^2
\\
&<& (1 + \sqrt{C_{\kappa}} )^2 h^2(f_{\ttheta
_0},f_{\ttheta
^{(p,\ell)}})
\\
&<& \kappa^{-1} h^2(f_{\ttheta_0},f_{\ttheta^{(p,\ell)}}),
\end{eqnarray*}
which leads to $\ttheta_0 \notin\B (\ttheta^{(p,\ell)},
r_{p,\ell
}  )$ as wished.
\end{pf}

Let us return to the proof of Theorem~\ref{ThmPrincipal}. Since the
result is straightforward when $\ttheta_0 \in\Theta_{N}$, we assume
that $\ttheta_0 \notin\Theta_N$. We then set
\[
p = \max \bigl\{i \in\{1,\ldots,N-1\}, \ttheta_0 \in\Theta
_i \bigr\}
\]
and consider any element $\ttheta_0'$ of $\Theta_N$. Then $\ttheta_0'$
belongs to $\Theta_p$ and
\begin{eqnarray*}
\delta^2 \bigl({\ttheta_0},{\ttheta_0'}
\bigr) &\leq& \sup_{\ttheta, \ttheta
' \in
\Theta_{p}} \delta^2 \bigl({\ttheta},{
\ttheta'}\bigr)
\\
&\leq& \inf_{\ell\in\{1,\ldots,L_p\}} h^2(f_{\ttheta^{(p,\ell
)}},f_{\ttheta'^{(p,\ell)}})
\\
&\leq& 2 \inf_{\ell\in\{1,\ldots,L_p\}} \bigl(h^2(f_{\ttheta
_0},f_{\ttheta^{(p,\ell)}})
+ h^2(f_{\ttheta_0},f_{\ttheta
'^{(p,\ell
)}}) \bigr).
\end{eqnarray*}
By the definition of $p$, $\ttheta_0 \in\Theta_p \setminus\Theta
_{p+1}$. We then derive from the above lemma that on $\Omega_{\xi}$,
\begin{eqnarray*}
\beta_1 \inf_{\ell\in\{1,\ldots,L_p\}} \bigl(h^2(f_{\ttheta
_0},f_{\ttheta^{(p,\ell)}})
+ h^2(f_{\ttheta_0},f_{\ttheta
'^{(p,\ell
)}}) \bigr) \leq
\beta_2 h^2(s, f_{\ttheta_0}) + \beta_3
\frac
{D_{\F
}^{(n)} + n \xi}{n}.
\end{eqnarray*}
Hence,
\[
\delta^2 \bigl({\ttheta_0},{\ttheta_0'}
\bigr) \leq\frac{2}{ \beta_1 } \biggl( \beta_2 h^2(s,
f_{\ttheta_0}) + \beta_3\frac{D_{\F}^{(n)} + n \xi}{n} \biggr).
\]
Since $ h^2(s, f_{\ttheta_0}) \leq h^2(s, \F) + 1/n$, there exists $C
> 0$ depending only on $\kappa$ such that
\[
C \delta^2 \bigl(\ttheta_0,\ttheta_0'
\bigr) \leq h^2(s, \F) + \frac{D_{\F}^{(n)}
}{n} + \xi \qquad\mbox{on $
\Omega_{\xi}$.}
\]
This concludes the proof of the theorem.

It remains to prove Claim~\ref{ClaimOmegaXi}. It actually derives from
the work of Baraud \cite{BaraudMesure}. More precisely, Proposition~2
of Baraud \cite{BaraudMesure} says that for all $f,f' \in\Fd$,
\[
\biggl(1 - \frac{1}{\sqrt{2}} \biggr) h^2\bigl(s,f'
\bigr) + \frac{\overline{T}
(f,f')}{\sqrt{2}} \leq \biggl(1 + \frac{1}{\sqrt{2}} \biggr)
h^2(s,f ) + \frac{\overline{T}(f,f') - \E [\overline{T}(f,f')
]}{\sqrt{2}}.
\]
Let $z = \upsilon- 1/\sqrt{2} \in(0, 1-1/\sqrt{2})$. We define
$\Omega
_{\xi}$ by
\[
\Omega_{\xi} = \bigcap_{f,f' \in\Fd} \biggl[
\frac{\overline{T}(f,f')
- \E [\overline{T}(f,f') ]}{ z  (h^2(s,f ) + h^2(s,f' )
 ) + c (D_{\F}^{(n)} + n \xi)/n} \leq\sqrt{2} \biggr].
\]
On this event,
\[
(1 - \upsilon ) h^2\bigl(s,f'\bigr) +
\frac{\overline{T}
(f,f')}{\sqrt
{2}} \leq (1 + \upsilon ) h^2(s,f ) + c
\frac{D_{\F
}^{(n)} +
n \xi}{n}
\]
and the inequality $\P(\Omega_{\xi}^c) \leq \mathrm{e}^{- n \xi}$ will follow
from Lemma~1 of Baraud \cite{BaraudMesure}. Before applying this lemma,
we need to check that his Assumption~3 is fulfilled. This is the
purpose of the claim below.

\begin{Claim} \label{ClaimDimMetriqueF2}
Let
\begin{eqnarray*}
\tau&=& 4 \frac{2 + ({n \sqrt{2}}/{6} )z}{({n^2}/{6}) z^2},
\\
\eta^2_{\F} &=& \max \Biggl\{3 d \mathrm{e}^{4}, \sum
_{j=1}^d \log \bigl(1 + 2
t_j^{-1} \bigl( (d/ \bar{\bolds{\alpha}}) \bigl(c
\overline{R}_j / \underline{R}_j^{(n)}\bigr)
\bigr)^{1/\alpha_j} \bigr) \Biggr\}.
\end{eqnarray*}
Then, for all $r \geq2 \eta_{\F}$,
%
\begin{equation}
\label{EqControlDimensioNFdis2} \bigl\llvert \Fd\cap\B_h (s, r \sqrt{\tau} )\bigr
\rrvert \leq \exp \bigl(r^2 / 2\bigr),
\end{equation}
where $\B_h(s, r \sqrt{\tau})$ is the Hellinger ball centered at $s$
with radius $r \sqrt{\tau}$ defined by
\[
\B_h (s, r \sqrt{\tau}) = \bigl\{f \in\L^1_+ (\XX,
\mu), h^2 (s,f) \leq r^2 \tau \bigr\}.
\]
\end{Claim}

\begin{pf}
If $\Fd\cap\B_h(s, r \sqrt{\tau}) = \varnothing$, (\ref
{EqControlDimensioNFdis2}) holds. In the contrary case, there exists
$\ttheta_0' = (\theta_{0,1}',\ldots, \theta_{0,d}') \in\Theta
_{\mathrm
{dis}}$ such that $h^2(s,f_{\ttheta_0'}) \leq r^2 \tau$, and thus
\[
\bigl|\Fd\cap\B_h (s, r \sqrt{\tau}) \bigr| \leq\bigl|\Fd\cap\B_h
(f_{\ttheta
_0'}, 2 r \sqrt{\tau})\bigr|.
\]
First of all, suppose that $r$ satisfies
%
\begin{equation}
\label{eqDansPreuveSurRTau} 4 r^2 \tau\leq\frac{c}{n} \sum
_{j=1}^d \log \bigl(1 + t_j^{-1}
(M_j-m_j) (\overline{R}_j n
)^{1/\alpha_j} \bigr).
\end{equation}
Then
\begin{eqnarray*}
\bigl|\Fd\cap\B_h (f_{\ttheta_0'}, 2 r \sqrt{\tau})\bigr| &=& \bigl
\llvert \bigl\{ f_{\ttheta}, \ttheta\in\Theta_{\mathrm{dis}},
h^2(f_{\ttheta}, f_{\ttheta_0'}) \leq4 r^2\tau
\bigr\} \bigr\rrvert
\\
&\leq& \bigl\llvert \bigl\{\ttheta\in\Theta_{\mathrm{dis}}, \forall j \in\{ 1,
\ldots,d\}, \underline{R}_j^{(n)} \bigl|\theta_j -
\theta_{0,j}' \bigr|^{\alpha_j} \leq4 r^2 \tau
\bigr\} \bigr\rrvert .
\end{eqnarray*}
Let $k_{0,j} \in\N$ be such that $\theta_{0,j}' = m_j + k_{0,j}
\varepsilon_j$. Then
\begin{eqnarray*}
\bigl|\Fd\cap\B_h (f_{\ttheta_0'}, 2 r \sqrt{\tau})\bigr| &\leq& \prod
_{j=1}^d \bigl\llvert \bigl
\{k_j \in\N, |k_j - k_{0,j} | \leq \bigl( {4
r^2 \tau}/{ \underline{R}_j ^{(n)}}
\bigr)^{1/\alpha_j} \varepsilon_j^{-1} \bigr\} \bigr
\rrvert
\\
&\leq& \prod_{j=1}^d \bigl(1 + 2
\varepsilon_j^{-1} \bigl( {4 r^2 \tau }/{
\underline{R}_j^{(n)} } \bigr)^{1/\alpha_j} \bigr).
\end{eqnarray*}
It is worthwhile to notice that $10^3\tau\leq c / n$. In particular,
by using the weaker inequality $4 \tau\leq c/n$ and $\varepsilon_j =
t_j (\overline{R}_j n)^{-1/\alpha_j}$,
\[
\bigl|\Fd\cap\B_h (f_{\ttheta_0'}, 2 r \sqrt{\tau})\bigr| \leq\prod
_{j=1}^d \bigl(1 + 2 t_j^{-1}
\bigl(r^2 c \overline{R}_j/{\underline{R}_j^{(n)}}
\bigr)^{1/\alpha_j} \bigr).
\]
If $\bar{\bolds{\alpha}} \leq \mathrm{e}^{-4}$, one can check that $\eta
_{\F
}^2 \geq4 d/ \bar{\bolds{\alpha}}$ (since $c \geq1$ and
$t_j^{-1} \geq d^{-1/\alpha_j}$).
If now $\bar{\bolds{\alpha}} \geq \mathrm{e}^{-4}$, then $\eta_{\F}^2
\geq
3 d \mathrm{e}^{4}
\geq3 d/ \bar{\bolds{\alpha}}$.
In particular, we always have $r^2 \geq10  ( d/ \bar{\bolds{\alpha}} )$.

We derive from the weaker inequality $r^2 \geq d/ \bar{\bolds{\alpha}}$ that
\begin{eqnarray*}
\bigl|\Fd\cap\B_h (f_{\ttheta_0'}, 2 r \sqrt{\tau})\bigr| &\leq& \biggl(
\frac
{r^2}{ d/ \bar{\bolds{\alpha}} } \biggr) ^{d/ \bar{\bolds{\alpha}}} \prod_{j=1}^d
\bigl(1 + 2 t_j^{-1} \bigl( ( d/ \bar{\bolds{\alpha}}) \bigl(c \overline{R}_j / \underline{R}_j^{(n)}
\bigr) \bigr)^{1/\alpha_j} \bigr)
\\
&\leq& \exp \biggl( \frac{\log (r^2/( d/ \bar{\bolds{\alpha}})
 )}{r^2/( d/ \bar{\bolds{\alpha}})} r^2 \biggr) \exp \bigl(\eta
_{\F}^2 \bigr).
\end{eqnarray*}
We then deduce from the inequalities $r^2/(d/ \bar{\bolds{\alpha}}) \geq10$ and $\eta_{\F}^2 \leq r^2/4$ that
\begin{eqnarray*}
\bigl|\Fd\cap\B_h (f_{\ttheta_0'}, 2 r \sqrt{\tau})\bigr| \leq\exp \bigl(
r^2/4 \bigr) \exp \bigl( r^2/4 \bigr) \leq\exp
\bigl(r^2/2\bigr)
\end{eqnarray*}
as wished. It remains to show that this inequality remains true when
(\ref{eqDansPreuveSurRTau}) does not hold. In this case,
\begin{eqnarray*}
\bigl|\Fd\cap\B_h (f_{\ttheta_0'}, 2 r \sqrt{\tau})\bigr| \leq| \Theta
_{\mathrm{dis}}| \leq\prod_{j=1}^d
\biggl(1 + \frac{M_j-m_j}{\varepsilon_j} \biggr) \leq \mathrm{e}^{4 r^2 \tau n / c}.
\end{eqnarray*}
The result follows from the inequality $4 \tau n / c \leq1/2$.
\end{pf}

We can now use Lemma~1 of Baraud \cite{BaraudMesure} to get for all
$\xi> 0$ and $y^2 \geq\tau ( 4 \eta^2_{\F} + n \xi )$,
\[
\P \biggl[ \sup_{f,f' \in\Fd} \frac{  (\overline{T}(f,f') -
\E
 [\overline{T}(f,f') ]  )/ \sqrt{2}}{ ( h^2(s,f
) +
h^2(s,f' )  ) \vee y^2} \geq z \biggr] \leq
\mathrm{e}^{- n \xi}.
\]
Since $4 \eta_{\F}^2 \leq10^{3} D_{\F}^{(n)}$ and $10^{3} \tau\leq c
/n$, we can choose
$y^2 = c  (D_{\F}^{(n)} + n \xi )/n$,
which concludes the proof of Claim~\ref{ClaimOmegaXi}. 

\subsection{Proof of Theorem \texorpdfstring{\protect\ref{thmGeneralLienAvecMLE}}{5.2}}
All along this proof, we set $\mathcal{C}_n (\vartheta) = \mathcal{A}_n
(\vartheta) \cap\mathcal{A}_n' (\vartheta)$ for all $\vartheta> 0$
and we denote by $\lambda_0$ the minimum between the smallest
eigenvalue of $I(\ttheta_0)$ and $1$. Since $I(\ttheta_0)$ is
invertible, $\lambda_0 \in(0,1]$. 

\begin{Claim} \label{claimExistenceMLE}
For all $r > 0$, there exists $\vartheta> 0$ such that, on the event
$\mathcal{C}_n (\vartheta)$, there exists a solution $\tilde{\ttheta}
\in\mathring{\Theta}$ of the likelihood equation
\[
\frac{1}{n} \sum_{i=1}^n
\dot{l}_{\tilde{\ttheta}} (X_i) = 0
\]
satisfying $\|\tilde{\ttheta} - \ttheta_0\| \leq r$.
\end{Claim}

\begin{pf}
The proof of this claim follows from classical arguments that can be
found in the literature. We make them explicit for the sake of completeness.
There exists a neighbourhood $\Theta_0 (\lambda_0/8)$ of $\ttheta_0$,
such that on $\mathcal{A}_n (\lambda_0/8) \cap\mathcal{A}_n'
(\lambda_0/8)$,
\begin{eqnarray*}
\frac{1}{n} \sum_{i=1}^n
\sup_{\ttheta\in\Theta_0 (\lambda_0/8)} \bigl\llVert \ddot{l}_{\ttheta}
(X_i) - \ddot{l}_{\ttheta_0} (X_i) \bigr\rrVert &
\leq & \frac{\lambda_0}{8},
\\
\Biggl\llVert \frac{1}{n} \sum_{i=1}^n
\bigl( \ddot{l}_{\ttheta_0} (X_i) - \E \bigl[
\ddot{l}_{\ttheta_0} (X_i) \bigr] \bigr) \Biggr\rrVert &\leq&
\frac
{\lambda_0}{8}.
\end{eqnarray*}
Without lost of generality, we may assume that $r$ is small enough so
that $r \leq1$ and that the ball $ \{\ttheta\in\Theta,   \|
\ttheta- \ttheta_0 \| \leq r  \}$ is a subset of $\Theta_0
(\lambda_0/8)$.
Let $\mathcal{S}_r$ be the $d$-sphere $\mathcal{S}_r =  \{
\ttheta
\in\Theta,   \|\ttheta- \ttheta_0 \| = r  \}$ and $\vartheta=
\lambda_0 r / 8$. Then, on $\mathcal{A}_n' (\vartheta)$,
\[
\Biggl\llVert \frac{1}{n} \sum
_{i=1}^n \dot{l}_{{\ttheta}_0}
(X_i) \Biggr\rrVert \leq\frac{r \lambda_0}{8}.
\]
We now use Taylor's theorem to show that for all $\ttheta\in\mathcal
{S}_r$, and $\mu$-almost all $x \in\XX$, there exists $\ttheta_x
\in
\Theta_0 (\lambda_0 / 8)$ such that
\[
l_{\ttheta} (x) = l_{{\ttheta_0}} (x) + \bigl(\dot{l}_{{\ttheta_0}}
(x) \bigr)^T (\ttheta- {\ttheta}_0) +
\tfrac{1}{2} (\ttheta- {\ttheta }_0)^T
\ddot{l}_{\ttheta_x} (x) (\ttheta- {\ttheta}_0).
\]
In particular, for all $\ttheta\in\mathcal{S}_r$,
\begin{eqnarray*}
&&\biggl\llvert L (\ttheta) - L (\ttheta_0) + \frac{1}{2} (
\ttheta- {\ttheta }_0)^T I (\ttheta_0) (
\ttheta- {\ttheta}_0) \biggr\rrvert \\
&&\quad\leq r \Biggl\llVert
\frac{1}{n} \sum_{i=1}^n
\dot{l}_{{\ttheta}_0} (X_i) \Biggr\rrVert + r^2 \Biggl
\llVert \frac{1}{2 n} \sum_{i=1}^n
\bigl(\ddot {l}_{\ttheta_{X_i}} (X_i) + I (\ttheta_0)
\bigr) \Biggr\rrVert
\\
&&\quad\leq r \Biggl\llVert \frac{1}{n} \sum_{i=1}^n
\dot {l}_{{\ttheta}_0} (X_i) \Biggr\rrVert + \frac{ r^2}{2 n}
\sum_{i=1}^n \sup_{\ttheta' \in\Theta_0 (\lambda_0/8)}
\bigl\llVert \ddot{l}_{\ttheta'} (X_i) - \ddot{l}_{\ttheta_{0}}
(X_i) \bigr\rrVert
\\
& &\quad\quad{} + r^2 \Biggl\llVert \frac{1}{2 n} \sum
_{i=1}^n \bigl(\ddot {l}_{\ttheta
_{0}}
(X_i) + I (\ttheta_0) \bigr) \Biggr\rrVert .
\end{eqnarray*}
Now, remark that
\[
f_{\ttheta_0} (x) \ddot{l}_{\ttheta_0} (x) = \ddot{f}_{\ttheta
_0}(x)
- \bigl(\dot{l}_{\ttheta_0}(x)\bigr) \bigl(\dot{l}_{\ttheta_0}(x)
\bigr)^T f_{\ttheta_0} (x),
\]
which, together with point (v) of Assumption~\ref
{AssumptionPourMLE}, yields $I(\ttheta_0) = - \E[ \ddot{l}_{\ttheta
_0} (X_1) ]$. Therefore, on the event $\mathcal{C}_n (\vartheta)
\subset\mathcal{A}_n (\lambda_0/8) \cap\mathcal{A}_n' (\vartheta)$,
\[
\biggl\llvert L (\ttheta) - L (\ttheta_0) + \frac{1}{2} (
\ttheta- {\ttheta }_0)^T I (\ttheta_0) (
\ttheta- {\ttheta}_0) \biggr\rrvert \leq\frac{
\lambda
_0 r^2}{4},
\]
which implies
\begin{eqnarray*}
L (\ttheta) - L (\ttheta_0) &\leq& \frac{\lambda_0 r^2}{4} -
\frac
{1}{2} (\ttheta- {\ttheta}_0)^T I (
\ttheta_0) (\ttheta- {\ttheta}_0)
\\
&\leq& \frac{\lambda_0 r^2}{4} - \frac{\lambda_0 r^2}{2} < 0.
\end{eqnarray*}
This means that for all $\ttheta\in\mathcal{S}_r$, $L(\ttheta) <
L(\ttheta_0)$ on the event $\mathcal{C}_n (\vartheta)$. In particular,
this proves that there exists $\tilde{\ttheta}$ in the ball $\{
\ttheta
\in\Theta,   \|\ttheta-\ttheta_0\| < r\}$ such that $\dot{L}
(\tilde
{\ttheta}) = 0$.
\end{pf}

\begin{Claim} \label{ClaimEncadrementHellinger}
For all $\tau\in(0,1)$, there exists a neighbourhood $\Theta_1 (\tau
)$ of $\ttheta_0$ such that for all $\ttheta, \ttheta' \in\Theta_1
(\tau)$,
\[
\frac{1-\tau}{8} \bigl(\ttheta- \ttheta'\bigr)^T I
(\ttheta_0) \bigl(\ttheta- \ttheta '\bigr) \leq
h^2 (f_{\ttheta}, f_{\ttheta'}) \leq\frac{1 + \tau}{8}
\bigl(\ttheta - \ttheta'\bigr)^T I (
\ttheta_0) \bigl(\ttheta- \ttheta'\bigr).
\]
\end{Claim}

The proof of this claim is omitted since it is very similar to the one
of Lemma~1.A of Section~31 of Borovkov \cite{borovkov1998}.

\begin{Claim} \label{ClaimComparaisonVraisemblancePolynome}
For all $\tau\in(0,1)$ and for all $r > 0$, there exist a
neighbourhood $\Theta_2 (\tau)$ of {$\ttheta_0$} (that does not depend
on $r$) and $\vartheta> 0$ such that on $\mathcal{C}_n (\vartheta)$:
\begin{itemize}
\item There exists a solution $\tilde{\ttheta} \in\Theta_2 (\tau)$ of
the likelihood equation satisfying $\|\tilde{\ttheta} - \ttheta_0 \|
\leq r$.
\item For all $ \ttheta\in\Theta_2 (\tau)$,
\[
\bigl\llvert L (\tilde{\ttheta}) - L (\ttheta) - \tfrac{1}{2} (\ttheta-
\tilde {\ttheta})^T I(\ttheta_0) (\ttheta- \tilde{
\ttheta}) \bigr\rrvert < \tau (\ttheta- \tilde{\ttheta})^T I (
\ttheta_0) (\ttheta- \tilde {\ttheta}).
\]
\end{itemize}
\end{Claim}

\begin{pf}
Let $\Theta_2 (\tau) $ be a convex neighbourhood of $\ttheta_0$
included in $\Theta_0 (\tau\lambda_0 )$. Without lost of generality,
we can assume that $r$ is small enough to ensure that the ball $\{
\ttheta\in\Theta,   \|\ttheta- \ttheta_0\| \leq r \}$ is included
in $\Theta_2 (\tau)$.

Thanks to Claim~\ref{claimExistenceMLE}, there exist a positive number
$\vartheta_0$ and a solution $\tilde{\ttheta} \in\Theta$ of the
likelihood equation satisfying $\|\tilde{\ttheta} - \ttheta_0 \|\leq r$
on $\mathcal{C}_n (\vartheta_0)$. In particular $\tilde{\ttheta}
\in
\Theta_2 (\tau)$. We then use Taylor's theorem to show that for all
$\ttheta\in\Theta_2 (\tau) $ and $\mu$-almost all $x \in\XX$, there
exists $\ttheta_x \in\Theta_2 (\tau) $ such that
\[
l_{\ttheta} (x) = l_{\tilde{\ttheta}} (x) + \bigl(\dot{l}_{\tilde
{\ttheta
}}
(x) \bigr)^T (\ttheta- \tilde{\ttheta}) + \tfrac{1}{2} (
\ttheta- \tilde{\ttheta})^T \ddot{l}_{\ttheta_x} (x) (\ttheta-
\tilde {\ttheta}).
\]
Therefore,
\begin{eqnarray*}
&&l_{\ttheta} (x) - l_{\tilde{\ttheta}} (x) - \bigl(\dot{l}_{\tilde
{\ttheta
}}
(x) \bigr)^T (\ttheta- \tilde{\ttheta}) + \tfrac{1}{2} (
\ttheta- \tilde{\ttheta})^T I (\ttheta_0) (\ttheta-
\tilde{\ttheta})\\
&&\quad = \tfrac
{1}{2} (\ttheta- \tilde{\ttheta})^T
\bigl(\ddot{l}_{\ttheta_x} (x) + I (\ttheta_0) \bigr) (\ttheta-
\tilde{\ttheta}).
\end{eqnarray*}
We derive that
\begin{eqnarray*}
\biggl\llvert L (\ttheta) - L (\tilde{\ttheta}) + \frac{1}{2} (\ttheta-
\tilde {\ttheta})^T I (\ttheta_0) (\ttheta- \tilde{
\ttheta}) \biggr\rrvert &=& \Biggl\llvert (\ttheta- \tilde{\ttheta})^T
\Biggl( \frac
{1}{2 n} \sum_{i=1}^n
\bigl( \ddot{l}_{\ttheta_{X_i}} (X_i) + I (\ttheta _0)
\bigr) \Biggr) (\ttheta- \tilde{\ttheta}) \Biggr\rrvert
\\
&\leq& \Biggl(\frac{1}{2 n} \sum_{i=1}^n
\sup_{\ttheta' \in\Theta_0 (\tau\lambda_0)} \bigl\llVert \ddot{l}_{\ttheta'}
(X_i) - \ddot{l}_{\ttheta_0} (X_i) \bigr\rrVert
\Biggr) \llVert \ttheta- \tilde{\ttheta}\rrVert ^2
\\
& &{} + \Biggl\llVert \frac{1}{2 n} \sum_{i=1}^n
\bigl(\ddot {l}_{\ttheta
_0} (X_i) + I (\ttheta_0)
\bigr) \Biggr\rrVert \llVert \ttheta- \tilde {\ttheta} \rrVert ^2.
\end{eqnarray*}
We now set $\vartheta= \min(\vartheta_0, \tau\lambda_0)$ so that
$\mathcal{C}_n (\vartheta) \subset\mathcal{C}_n (\vartheta_0) \cap
\mathcal{C}_n (\tau\lambda_0 )$. On the event $\mathcal{C}_n
(\vartheta
)$, we thus have for all $\ttheta\in\Theta_2 (\tau) $,
\begin{eqnarray*}
\bigl\llvert L (\ttheta) - L (\tilde{\ttheta}) + \tfrac{1}{2} (\ttheta-
\tilde {\ttheta})^T I (\ttheta_0) (\ttheta- \tilde{
\ttheta}) \bigr\rrvert &\leq& \tau\lambda_0 \llVert \ttheta-
\tilde{\ttheta}\rrVert ^2
\\
&\leq& \tau(\ttheta- \tilde{\ttheta})^T I (\ttheta_0)
(\ttheta- \tilde{\ttheta}).
\end{eqnarray*}
This completes the proof.
\end{pf}

\begin{Claim} \label{dernierClaimMLE}
For all $\tau\in(0,1)$ and $r > 0$, there exist a neighbourhood
$\Theta_3 (\tau)$ of {$\ttheta_0$} (that does not depend on $r$) and
$\vartheta> 0$ such that on $\mathcal{C}_n (\vartheta)$:
\begin{itemize}
\item There exists a solution of the likelihood equation $\tilde
{\ttheta
} \in\Theta$ satisfying $\|\tilde{\ttheta} - \ttheta_0 \|\leq r$.
\item For all $\ttheta, \ttheta' \in\Theta_3 (\tau)$,
%
\begin{equation}
\label{ComparaisonTestHellinger} \overline{T} (f_{\ttheta}, f_{\ttheta'}) \leq \biggl(
\frac{8 +
5\sqrt
{2}}{7} + \tau \biggr) h^2(f_{\tilde{\ttheta}},
f_{\ttheta}) - (8 - 5 \sqrt{2} - \tau ) h^2(f_{\tilde{\ttheta}},
f_{\ttheta'}).
\end{equation}
\end{itemize}
\end{Claim}

\begin{pf}
We introduce the function $F$ defined on $(0,+\infty)$ by $F(x) =
(\sqrt
{x}-1)/\sqrt{1+x}$ and define for all $\ttheta, \ttheta' \in\Theta$,
\begin{eqnarray*}
\overline{T}_1 (f_{\ttheta}, f_{\ttheta'}) &=&
\frac{1}{2 } \int_{\XX
} \sqrt{f_{\ttheta} (x) +
f_{\ttheta'}(x)} \bigl(\sqrt{f_{\ttheta'} (x)} - \sqrt{f_{\ttheta}
(x)} \bigr) \,\d\mu(x),
\\
\overline{T}_2 (f_{\ttheta}, f_{\ttheta'}) &=&
\frac{1}{n} \sum_{i=1}^n F \biggl({
\frac{f_{\ttheta'} (X_i) }{f_{\ttheta}
(X_i)}} \biggr).
\end{eqnarray*}
Remark that for all $\ttheta, \ttheta' \in\Theta$,
%
\begin{equation}
\label{decomTest} \overline{T} ( f_{\ttheta}, f_{\ttheta'}) =
\overline{T}_1 (f_{\ttheta
}, f_{\ttheta'}) +
\overline{T}_2 (f_{\ttheta}, f_{\ttheta'}).
\end{equation}
We begin by bounding $\overline{T}_1 (f_{\ttheta}, f_{\ttheta'})$ from
above. Since $f_{\ttheta}$ and $f_{\ttheta'}$ are two densities,
$\overline{T}_1 (f_{\ttheta}, f_{\ttheta'})$ is also equal to
\begin{eqnarray*}
\overline{T}_1 (f_{\ttheta}, f_{\ttheta'}) =
\frac{1}{2} \int_{\XX} \biggl(\sqrt{f_{\ttheta}
(x) + f_{\ttheta'}(x)} - \frac{\sqrt
{f_{\ttheta
}(x)} + \sqrt{f_{\ttheta'} (x)} }{\sqrt{2}} \biggr) \bigl(\sqrt
{f_{\ttheta'} (x)} - \sqrt{f_{\ttheta} (x)} \bigr) \,\d\mu(x).
\end{eqnarray*}
By using the inequality
\[
\biggl\llvert \sqrt{a+b} - \frac{\sqrt{a}+\sqrt{b}}{ \sqrt{2}} \biggr\rrvert \leq (1 - {1}/{
\sqrt{2}} ) |\sqrt{b}-\sqrt{a} |\qquad  \mbox{for all $a,b \geq0$,}
\]
we get
\[
\overline{T}_1 (f_{\ttheta}, f_{\ttheta'}) \leq (1- {1}/{
\sqrt {2}} ) h^2 (f_{\ttheta}, f_{\ttheta'}).
\]
We then use the triangular inequality to deduce
%
\begin{equation}
\label{eqMajorationT2} \overline{T}_1 (f_{\ttheta}, f_{\ttheta'})
\leq (1- {1}/{\sqrt {2}} ) \biggl[ \biggl(1+ \frac{5+4\sqrt{2}}{7} \biggr)
h^2 (f_{\ttheta
}, f_{\tilde{\ttheta}}) + \biggl(1+
\frac{7}{5+4\sqrt{2}} \biggr) h^2 (f_{\ttheta'}, f_{\tilde{\ttheta}})
\biggr].\quad
\end{equation}
We now aim at bounding $\overline{T}_2 (f_{\ttheta}, f_{\ttheta'}) $
from above. We consider $\tau_0 \in(0,1/2]$ such that
%
\begin{equation}
\label{defTheta0MLE} \frac{1+\tau_0}{1-\tau_0} \leq1+ \frac{\tau}{2},
\end{equation}
and define
\[
\Theta_3 (\tau) = \Theta_0 \biggl(\frac{{\tau\lambda_0 \sqrt
{2}}/{64}}{5 \sqrt{2}/{384}}
\biggr) \cap\Theta_1 (\tau_0) \cap \Theta_2
(\tau_0/2),
\]
where we recall that $\Theta_0 (\cdot)$ is given by point (vi) of Assumption~\ref{AssumptionPourMLE} and
that $\Theta_1 (\cdot)$, $\Theta_2 (\cdot)$ are defined in the two
preceding claims. Thanks to Claim~\ref
{ClaimComparaisonVraisemblancePolynome}, there exists $\vartheta_0 > 0$
such that on $\mathcal{C}_n (\vartheta_0)$, there exists a solution
$\tilde{\ttheta}$ of the likelihood equation satisfying $\|\tilde
{\ttheta} - \ttheta_0\| \leq r$, and such that for all $\ttheta\in
\Theta_2 (\tau_0/2)$,
%
\begin{eqnarray}
\label{eqUtilisePreuveMLE} \biggl\llvert L (\tilde{\ttheta}) - L (\ttheta) -
\frac{1}{2} (\ttheta- \tilde{\ttheta})^T I(
\ttheta_0) (\ttheta- \tilde{\ttheta}) \biggr\rrvert <
\frac{\tau_0}{2} (\ttheta- \tilde{\ttheta})^T I (
\ttheta_0) (\ttheta- \tilde{\ttheta}).
\end{eqnarray}
We then set
\[
\vartheta= \min \biggl\{\vartheta_0, \frac{\tau\lambda_0 \sqrt{2}/64}{
{5 \sqrt{2}}/{ 384} } \biggr\}.
\]
We shall bound $\overline{T}_2 (f_{\ttheta}, f_{\ttheta'})$ on the
event $\mathcal{C}_n (\vartheta)$. To this end, remark that
\[
\biggl\llvert F(x) - \frac{\log x}{2 \sqrt{2}} \biggr\rrvert \leq\frac{5 \sqrt
{2}}{384}
\llvert \log x\rrvert ^3\qquad \mbox{for all $x > 0$.}
\]
Consequently,
%
\begin{eqnarray}
\label{eqT2MLE} \overline{T}_2 (f_{\ttheta}, f_{\ttheta'})
- \frac{L(\ttheta')}{2
\sqrt{2}} + \frac{ L (\ttheta) }{2 \sqrt{2}} &=& \frac{1}{n} \sum
_{i=1}^n \biggl[ F \biggl( \frac{f_{\ttheta'} (X_i)}{f_{\ttheta}
(X_i)}
\biggr) - \frac{ 1 }{2 \sqrt{2}} \log \biggl( \frac
{f_{\ttheta'}
(X_i)}{ f_{\ttheta} (X_i)} \biggr) \biggr]
\nonumber
\\[-8pt]
\\[-8pt]
\nonumber
&\leq& \frac{5 \sqrt{2}}{ 384 n} \sum_{i=1}^n
\bigl\llvert \log f_{\ttheta'} (X_i) - \log f_{\ttheta}
(X_i) \bigr\rrvert ^3.
\end{eqnarray}
On the event $\mathcal{C}_n (\vartheta)$, for all $\ttheta, \ttheta'
\in\Theta_3 (\tau)$:
\begin{eqnarray*}
\frac{5 \sqrt{2}}{ 384 n} \sum_{i=1}^n \bigl
\llvert \log f_{\ttheta} (X_i) - \log f_{\ttheta'}
(X_i) \bigr\rrvert ^3 &\leq& \frac{\tau\sqrt{2} }{64}
\lambda_0 \bigl\llVert \ttheta- \ttheta'\bigr\rrVert
^2
\\
&\leq& \frac{\tau\sqrt{2}}{64} \bigl(\ttheta- \ttheta'
\bigr)^T I(\ttheta_0) \bigl(\ttheta-
\ttheta'\bigr).
\end{eqnarray*}
By using $\ttheta, \ttheta' \in\Theta_1(\tau_0)$ and that $\tau_0
\leq
1/2$, we deduce from Claim~\ref{ClaimEncadrementHellinger} that
\begin{eqnarray*}
\frac{5 \sqrt{2}}{ 384 n} \sum_{i=1}^n \bigl
\llvert \log f_{\ttheta} (X_i) - \log f_{\ttheta'}
(X_i) \bigr\rrvert ^3 \leq\frac{8}{1 - {1}/{2}} \times
\frac{\tau\sqrt{2}}{64} h^2(f_{\ttheta}, f_{\ttheta'}).
\end{eqnarray*}
By putting this inequality into (\ref{eqT2MLE}),
\begin{eqnarray*}
\overline{T}_2 (f_{\ttheta}, f_{\ttheta'}) &\leq&
\frac{L(\ttheta
')}{2 \sqrt{2}} - \frac{ L (\ttheta) }{2 \sqrt{2}} + \frac{\tau
\sqrt
{2}}{4} h^2(f_{\ttheta},
f_{\ttheta'})
\\
&\leq& \frac{L(\ttheta') - L (\tilde{\ttheta})}{2 \sqrt{2}} - \frac{ L
(\ttheta) - L (\tilde{\ttheta}) }{2 \sqrt{2}} + \frac{\tau\sqrt
{2}}{4}
h^2(f_{\ttheta}, f_{\ttheta'}).
\end{eqnarray*}
We deduce from (\ref{eqUtilisePreuveMLE}),
\begin{eqnarray*}
\overline{T}_2 (f_{\ttheta}, f_{\ttheta'}) &\leq&
\frac{1 + \tau_0}{4
\sqrt{2}} (\ttheta- \tilde{\ttheta})^T I (
\ttheta_0) (\ttheta- \tilde {\ttheta}) - \frac{1 - \tau_0}{4 \sqrt{2}} \bigl(
\ttheta' - \tilde {\ttheta }\bigr)^T I (
\ttheta_0) \bigl(\ttheta' - \tilde{\ttheta}\bigr)\\
&&{} +
\frac{\tau\sqrt
{2}}{4} h^2(f_{\ttheta}, f_{\ttheta'}).
\end{eqnarray*}
Since $\ttheta, \ttheta' $ belong together to $\Theta_1 (\tau_0)$,
\begin{eqnarray*}
\overline{T}_2 (f_{\ttheta}, f_{\ttheta'}) \leq\sqrt{2}
\frac{1 +
\tau_0}{1 - \tau_0} h^2(f_{\tilde{\ttheta}}, f_{\ttheta}) - \sqrt
{2}\frac{1 - \tau_0}{1 + \tau_0} h^2(f_{\tilde{\ttheta}}, f_{\ttheta'}) +
\frac{\tau\sqrt{2}}{4} h^2(f_{\ttheta}, f_{\ttheta'}).
\end{eqnarray*}
It follows from (\ref{defTheta0MLE}) that $(1 - \tau_0)/(1 + \tau_0)
\geq1/ (1+\tau/2) \geq1 - \tau/2$, and thus
\begin{eqnarray*}
\overline{T}_2 (f_{\ttheta}, f_{\ttheta'}) \leq\sqrt{2}
\biggl(1 + \frac{\tau}{2} \biggr) h^2(f_{\tilde{\ttheta}},
f_{\ttheta}) - \sqrt {2} \biggl(1 - \frac{\tau}{2} \biggr)
h^2(f_{\tilde{\ttheta}}, f_{\ttheta
'}) + \frac{\tau\sqrt{2}}{4}
h^2(f_{\ttheta}, f_{\ttheta'}).
\end{eqnarray*}
By using the triangular inequality,
\[
\frac{\tau\sqrt{2}}{4} h^2(f_{\ttheta}, f_{\ttheta'}) \leq
\sqrt{2} \biggl( \frac{\tau}{2} h^2(f_{\tilde{\ttheta}},
f_{\ttheta}) + \frac
{\tau}{2} h^2(f_{\tilde{\ttheta}},
f_{\ttheta'}) \biggr)
\]
and hence
\begin{eqnarray*}
\overline{T}_2 (f_{\ttheta}, f_{\ttheta'}) \leq(\sqrt{2}
+ \tau) h^2(f_{\tilde{\ttheta}}, f_{\ttheta}) - (\sqrt{2} - \tau)
h^2(f_{\tilde
{\ttheta}}, f_{\ttheta'}).
\end{eqnarray*}
We then use (\ref{decomTest}) and (\ref{eqMajorationT2}) to complete
the proof.
\end{pf}

We now return to the proof of Theorem~\ref{thmGeneralLienAvecMLE}. Let
$\beta_1, \beta_2, \beta_3$ be the numbers given at the beginning of
the proof of Theorem~\ref{ThmPrincipal} (they only depend on $\kappa$).
Let $\tau= 0.01$ and let $\Theta_3 (\tau)$ be the set given by
Claim~\ref{dernierClaimMLE}. There exists $r_0 > 0$ such that the ball
\[
\B(\ttheta_0,r_0) = \bigl\{\ttheta\in\Theta,
h(f_{\ttheta}, f_{\ttheta
_0}) \leq r_0\bigr\}
\]
is included in $\Theta_3 (\tau)$.
We define $\xi= r_0^2 \beta_1 / (9 \beta_3)$ and consider $\vartheta>
0$ so that there exists a solution $\tilde{\ttheta}$ of the likelihood
equation on $\mathcal{C}_n (\vartheta)$ satisfying
%
\begin{equation}
\label{eqPreuveMLEHellinger} h^2(f_{\ttheta_0}, f_{\tilde{\ttheta}}) \leq
\bigl[9 (1 + \beta _2/\beta_1 ) \bigr]^{-1}
r_0^2.
\end{equation}
We may assume (without lost of generality) that $\tilde{\ttheta} \notin\Theta_{N_n}$. We then set
\[
p = \max \bigl\{i \in\{1,\ldots,N_n-1\}, \tilde{\ttheta} \in \Theta
_i \bigr\}.
\]
By the definition of $p$, $\tilde{\ttheta} \in\Theta_{p} \setminus
\Theta_{p+1}$. There exists $\ell\in\{1,\ldots,L_p\}$ such that
$\tilde
{\ttheta} \in B^{(p,\ell)}$ and a look at the proof of Lemma~\ref
{LemmeControleH} shows that on the event $\Omega_{\xi} \cap\mathcal
{C}_n (\vartheta)$,
%
\begin{eqnarray}
\label{eqDansPreuveMLE} \beta_2 h^2(f_{\ttheta_0},
f_{\tilde{\ttheta}}) + \beta_3 \frac
{D_{\F
}^{(n)}}{n} + \frac{r_0^2 \beta_1}{9}
\geq\beta_1 \bigl(h^2(f_{\tilde
{\ttheta}},f_{\ttheta^{(p,\ell)}})
+ h^2(f_{\tilde{\ttheta
}},f_{\ttheta
'^{(p,\ell)}}) \bigr).
\end{eqnarray}
Without lost of generality, we may suppose that $T({\ttheta}^{(p,\ell
)}, {\ttheta}'^{(p,\ell)}) = \overline{T}(f_{\pi({\ttheta}^{(p,\ell)})},
f_{\pi({\ttheta}'^{(p,\ell)})})$ is non-negative and $\tilde
{\ttheta}
\in\B (\ttheta^{(p,\ell)}, r_{p,\ell}  )$.
Now, by using the triangular inequality and the fact that for all
$\ttheta\in\Theta$,
\[
h^2(f_{\ttheta}, f_{\pi(\ttheta)}) \leq\sup
_{1 \leq j \leq d} \overline {R}_j \varepsilon_j^{2}
= \frac{1}{n} \sup_{1 \leq j \leq d} \bigl(t_j^{(n)}
\bigr)^2 \leq d/n,
\]
we get
\[
h^2 (f_{\pi({\ttheta}^{(p,\ell)})}, f_{\ttheta_0} ) \leq 3 h^2
(f_{{\ttheta}^{(p,\ell)}}, f_{\tilde{\ttheta}} ) + 3 h^2 (f_{\tilde{\ttheta}},
f_{\ttheta_0} ) + 3 d /n.
\]
We use (\ref{eqDansPreuveMLE}) and then (\ref{eqPreuveMLEHellinger})
to deduce
\begin{eqnarray*}
h^2 (f_{\pi({\ttheta}^{(p,\ell)})}, f_{\ttheta_0} ) &\leq & 3 (1 +
\beta_2/\beta_1 ) h^2(f_{\ttheta_0},
f_{\tilde
{\ttheta
}}) + 3 (\beta_3/\beta_1)
\frac{D_{\F}^{(n)}}{n} + \frac{r_0^2
}{3} + \frac{3 d}{n}
\\
&\leq& \frac{2 r_0^2 }{3} + 3 (\beta_3/\beta_1)
\frac{D_{\F}^{(n)}}{n} + \frac{3 d}{n}.
\end{eqnarray*}
Since $3 (\beta_3/\beta_1) {D_{\F}^{(n)}}/{n} + {3 d}/{n}$ tends to $0$
when $n$ goes to infinity, there exists $n_0 \in\N^{\star}$ such that
for all $n \geq n_0$, $ h^2  (f_{\pi({\ttheta}^{(p,\ell)})},
f_{\ttheta_0}  ) \leq r_0^2$.
Similarly, the bound $ h^2  (f_{\pi({\ttheta'}^{(p,\ell)})},
f_{{\ttheta_0}}  ) \leq r_0^2$ also holds. In particular, $\pi
({\ttheta}^{(p,\ell)})$ and $\pi({\ttheta'}^{(p,\ell)})$ belong
together to $\Theta_3 (\tau)$. We can therefore use (\ref
{ComparaisonTestHellinger}) to get
\begin{eqnarray*}
\overline{T} (f_{\pi({\ttheta}^{(p,\ell)})}, f_{\pi({\ttheta
}'^{(p,\ell
)})}) &\leq& \biggl(
\frac{8 + 5\sqrt{2}}{7} + \tau \biggr) h^2(f_{\tilde{\ttheta}},
f_{\pi({\ttheta}^{(p,\ell)})})
\\
& &{} - (8 - 5 \sqrt{2} - \tau ) h^2(f_{\tilde
{\ttheta}},
f_{\pi({\ttheta'}^{(p,\ell)})}).
\end{eqnarray*}
Since $\overline{T} (f_{\pi({\ttheta}^{(p,\ell)})}, f_{\pi
({\ttheta
}'^{(p,\ell)})}) $ is non-negative, we may replace $\tau$ by its
numerical value $\tau= 0.01$ to get
\[
h (f_{\tilde{\ttheta} }, f_{\pi({\ttheta'}^{(p,\ell)})}) \leq1.6 h (f_{\tilde{\ttheta}},
f_{\pi({\ttheta}^{(p,\ell)})}).
\]
Therefore,
\begin{eqnarray*}
h(f_{{\ttheta}^{(p,\ell)}}, f_{{\ttheta}'^{(p,\ell)}}) &\leq& h(f_{{\ttheta}^{(p,\ell)}},
f_{\tilde{\ttheta}}) + h(f_{\tilde
{\ttheta
}}, f_{\pi({\ttheta}'^{(p,\ell)})}) +
h(f_{\pi({\ttheta}'^{(p,\ell)})}, f_{{\ttheta}'^{(p,\ell)}})
\\
&\leq& h(f_{{\ttheta}^{(p,\ell)}}, f_{\tilde{\ttheta}}) + 1.6 h (f_{\tilde{\ttheta}},
f_{\pi({\ttheta}^{(p,\ell)})}) + \sqrt {\frac
{1}{n} \sup_{1 \leq j \leq d}
\bigl(t_j^{(n)} \bigr)^{2}}
\\
&\leq& h(f_{{\ttheta}^{(p,\ell)}}, f_{\tilde{\ttheta}}) + 1.6 \biggl[ h
(f_{\tilde{\ttheta}}, f_{{\ttheta}^{(p,\ell)}}) + \sqrt{\frac
{1}{n} \sup
_{1 \leq j \leq d} \bigl(t_j^{(n)}
\bigr)^{2}} \biggr] + \sqrt{\frac{1}{n} \sup
_{1 \leq j \leq d} \bigl(t_j^{(n)}
\bigr)^{2}}
\\
&\leq& 2.6 h(f_{{\ttheta}^{(p,\ell)}}, f_{\tilde{\ttheta}}) + 2.6\sqrt {
\frac{1}{n} \sup_{1 \leq j \leq d} \bigl(t_j^{(n)}
\bigr)^{2}}.
\end{eqnarray*}
Since $\tilde{\ttheta} \in\B (\ttheta^{(p,\ell)}, r_{p,\ell}
 )$,
\begin{eqnarray*}
h(f_{{\ttheta}^{(p,\ell)}}, f_{{\ttheta}'^{(p,\ell)}}) \leq2.6 \bar {\kappa}^{1/2} h
(f_{{\ttheta}^{(p,\ell)}}, f_{{\ttheta}'^{(p,\ell)}}) + 2.6 \sqrt{\frac{1}{n}
\sup_{1 \leq j \leq d} \bigl(t_j^{(n)}
\bigr)^{2}}.
\end{eqnarray*}
By replacing $\bar{\kappa}$ by its numerical value $\bar{\kappa} =
3/2-\sqrt{2}$,
\begin{eqnarray*}
h(f_{{\ttheta}^{(p,\ell)}}, f_{{\ttheta}'^{(p,\ell)}}) \leq10.91 \sqrt {
\frac{1}{n} \sup_{1 \leq j \leq d} \bigl(t_j^{(n)}
\bigr)^{2}}.
\end{eqnarray*}
Let now $\ttheta$ be any element of $\Theta_{N_n}$. Since $\ttheta,
\tilde{\ttheta}$ belong together to $\Theta_p$,
\begin{eqnarray*}
\delta^2 ({\ttheta}, {\tilde{\ttheta}}) \leq\sup
_{\ttheta',
\ttheta''
\in\Theta_p} \delta^2 \bigl({\ttheta'}, {
\ttheta''}\bigr) \leq h^2(f_{{\ttheta
}^{(p,\ell)}},
f_{{\ttheta}'^{(p,\ell)}}) \leq\frac{120}{n} \sup_{1
\leq j \leq d}
\bigl(t_j^{(n)} \bigr)^{2}.
\end{eqnarray*}
Finally, we have shown that there exist $\xi> 0$, $\vartheta> 0$,
$n_0 \in\N^{\star}$ such that for all $n \geq n_0$,
\[
\P \biggl[\exists\tilde{\ttheta} \in\Theta, \dot{L}(\tilde {\ttheta}) = 0
\mbox{
and } \inf_{\ttheta\in\Theta_{N_n}} \delta^2 (\ttheta, \tilde{
\ttheta} ) \leq\frac{120 }{n} \sup_{1 \leq j \leq d}
\bigl(t_j^{(n)} \bigr)^{2} \biggr] \geq\P \bigl[
\Omega_{\xi} \cap \mathcal {C}_n (\vartheta) \bigr].
\]
The theorem follows from $\P(\Omega_{\xi}) \geq1 - \mathrm{e}^{ - n \xi}$ and
$\mathcal{C}_n (\vartheta) = \mathcal{A}_n (\vartheta) \cap
\mathcal
{A}'_n (\vartheta)$.

\subsection{Proof of Theorem \texorpdfstring{\protect\ref{ThmPrincipalDim1}}{2.1}}
It follows from Theorem~\ref{ThmPrincipal}, page \pageref
{ThmPrincipal}, where $d=1$, $\alpha_1 = \alpha$, $\Theta_i =
[\theta
^{(i)}, \theta'^{(i)}]$, $L_i = 1$, $\underline{R}_1^{(n)} =
\underline
{R}$, $\overline{R}_1 = \overline{R}$ and $\delta^2(\theta, \theta
') =
\underline{R} |\theta- \theta'|^{\alpha}$ in the first part of the
theorem and $\delta^2(\theta, \theta') = h^2 (f_{\theta}, f_{\theta'})$
in the second part. 

\subsection{Proof of Proposition \texorpdfstring{\protect\ref{PropCalculComplexiteDimen1}}{2.1}} \label{sectionProofPreuveComplexiteDim1}
For all $i \in\{1,\ldots,N-1\}$,
\begin{eqnarray*}
\theta^{(i+1)}&\in& \bigl\{ \theta^{(i)}, \theta^{(i)}
+ \min \bigl(\overline{r} \bigl(\theta^{(i)},\theta'^{(i)}
\bigr), \bigl(\theta'^{(i)} - \theta ^{(i)}
\bigr)/2 \bigr) \bigr\},
\\
\theta'^{(i+1)} &\in& \bigl\{\theta'^{(i)},
\theta'^{(i)} - \min \bigl(\underline{r} \bigl(
\theta^{(i)},\theta'^{(i)}\bigr), \bigl(
\theta'^{(i)} - \theta ^{(i)}\bigr)/2 \bigr) \bigr
\}.
\end{eqnarray*}
Since $\overline{r} (\theta^{(i)},\theta'^{(i)})$ and $\underline{r}
(\theta^{(i)},\theta'^{(i)})$ are larger than
$(\kappa\underline{R}/ \overline{R})^{1/\alpha} (\theta'^{(i)} -
\theta
^{(i)} )$,
\[
\theta'^{(i+1)} - \theta^{(i+1)} \leq\max \bigl\{1 -
{ (\kappa \underline{R}/\overline{R} )^{1/\alpha}}, {1}/{2} \bigr\} \bigl(
\theta '^{(i)} - \theta^{(i)} \bigr).
\]
By induction, we derive that for all $i \in\{1,\ldots,N-1\}$,
\[
\theta'^{(i+1)} - \theta^{(i+1)} \leq \bigl(\max
\bigl\{1 - { (\kappa\underline{R}/\overline{R} )^{1/\alpha}}, {1}/{2} \bigr\}
\bigr)^{i} (M -m).
\]
Consequently, the algorithm converges in less than $N$ iterations where
$N$ is the smallest integer such that
\[
\bigl( \max \bigl\{1 - { (\kappa\underline{R}/\overline{R} )^{1/\alpha}},
{1}/{2} \bigr\} \bigr)^{N} (M-m) \leq\eta,
\]
that is,
\[
N\geq\frac{\log ( (M- m )/\eta )}{ -\log [ \max
\{1 - { (\kappa\underline{R}/\overline{R}  )^{1/\alpha}},
1/2  \}  ]}.
\]
We conclude by using the inequality $-1/\log(1-x) \leq1/x$ for all $x
\in(0,1)$.

\subsection{Proof of Theorem \texorpdfstring{\protect\ref{thmLienMLEDim1}}{2.2}}
We shall apply Theorem~\ref{thmGeneralLienAvecMLE}, page \pageref
{thmGeneralLienAvecMLE}, with $\Theta_i = [\theta^{(i)}, \theta
'^{(i)}]$, $L_i = 1$, $\delta^2(\theta, \theta') = \underline{R}
|\theta- \theta'|^{2}$ and with $N_n$ corresponding to the last step
of the algorithm.

The proof that Assumption \ref{HypSurLeModeleQuelquonqueDebutDim1} holds can be derived from
Theorem~3 of Section~31 of Borovkov~\cite{borovkov1998}. Moreover,
point (vi) of Assumption~\ref
{AssumptionPourMLE} is satisfied with the event $\A_n (\vartheta)$ on which
\[
\frac{1}{n} \sum_{i=1}^n
\varphi_1^3 (X_i) \leq2 \E \bigl[
\varphi_1^3 (X_1) \bigr] \quad\mbox{and}\quad
\frac{1}{n} \sum_{i=1}^n
\varphi_2 (X_i) \leq2 \E \bigl[\varphi_2
(X_1) \bigr]
\]
and with
\[
\Theta_0 (\vartheta) = \biggl\{\theta\in(m,M), |\theta-
\theta_0| \leq\min \biggl\{\frac{1}{2} \biggl(
\frac{\vartheta}{ 2 \E
[\varphi
_1^3 (X_1)  ]} \biggr)^{1/(3\gamma_1-2)}, \biggl(\frac
{\vartheta}{2
\E [\varphi_2 (X_1)  ] }
\biggr)^{1/\gamma_2} \biggr\} \biggr\}.
\]
Theorem~\ref{thmGeneralLienAvecMLE} then asserts that there exist
$\vartheta> 0$, $\xi> 0$, $n_0 \in\N^{\star}$, such that for $n
\geq n_0$:
\begin{eqnarray*}
& & \P \Biggl[\exists\tilde {\theta } \in(m,M), \sum
_{i=1}^n \dot{l}_{\tilde{\theta}}
(X_i) = 0 \mbox { and } \inf_{\theta\in\Theta_{N_n}} \underline{R}
(\theta- \tilde {\theta})^2 \leq\frac{120 }{n}
\bigl(t^{(n)} \bigr)^2 \Biggr]
\\
&&\quad\geq  1 - \bigl\{\P \bigl[ \bigl(\mathcal{A}_n (\vartheta)
\bigr)^c \bigr] + \P \bigl[ \bigl( \mathcal{A}_n'
(\vartheta) \bigr)^c \bigr] + \mathrm{e}^{-n
\xi} \bigr\}.
\end{eqnarray*}
Recalling that $\theta'^{(N_n)} - \theta^{(N_n)} \leq t^{(n)}
(\overline
{R} n)^{-1/2}$ and that $\hat{\theta}$ is the middle of the
interval $\Theta_{N_n} = [\theta^{(N_n)}, \theta'^{(N_n)}]$,
\[
\underline{R} (\hat{\theta} - \tilde{\theta})^2 \leq2 \inf
_{\theta\in
\Theta_{N_n}} \underline{R} (\theta- \tilde{\theta})^2 + 2
( \underline {R}/{\overline{R}}) \frac{(t^{(n)})^2}{n}.
\]
This shows that there exists $C > 0$ such that
\begin{eqnarray*}
& & \P \Biggl[\exists\tilde {\theta } \in(m,M), \sum
_{i=1}^n \dot{l}_{\tilde{\theta}}
(X_i) = 0 \mbox { and } |\hat{\theta} - \tilde{\theta} | \leq C
\frac{ t^{(n)}}{
\sqrt{n}} \Biggr]
\\
&&\quad\geq  1 - \bigl\{\P \bigl[ \bigl(\mathcal{A}_n (\vartheta)
\bigr)^c \bigr] + \P \bigl[ \bigl( \mathcal{A}_n'
(\vartheta) \bigr)^c \bigr] + \mathrm{e}^{-n
\xi} \bigr\}.
\end{eqnarray*}
Therefore, for all $n$, this probability is always larger than $1-\zeta
_n$ where\vspace*{-1pt}
\begin{eqnarray*}
\zeta_n = \cases{ $1$, &\quad $\mbox{if $n < n_0$},$
\vspace*{2pt}
\cr
\min\bigl\{1, \P \bigl[ \bigl(\mathcal{A}_n (
\vartheta) \bigr)^c \bigr] + \P \bigl[ \bigl( \mathcal{A}_n'
(\vartheta) \bigr)^c \bigr] + \mathrm{e}^{- n \xi
}\bigr\} ,&\quad $\mbox{if $n
\geq n_0$.}$}\vspace*{-1pt}
\end{eqnarray*}
By the law of large numbers, the two probabilities $\P [
(\mathcal{A}_n (\vartheta) )^c  ]$ and $\P [ (
\mathcal{A}_n' (\vartheta) )^c ]$ converge to $1$ and
therefore also the sequence $(\zeta_n)_{n \geq1}$.

We now prove that $\hat{\theta}$ is asymptotically efficient.
Let $\tilde{\theta}$ be an estimator satisfying $\sum_{i=1}^n \dot
{l}_{\tilde{\theta}} (X_i) = 0$ and $|\hat{\theta} - \tilde{\theta
} |
\leq C { t^{(n)}}/{ \sqrt{n}}$ with probability tending to $1$ when
$n$ goes to infinity.
Let us consider for $\mu$-almost all $x \in\XX$,\vspace*{-1pt}
\[
R(x) = \int_0^1 \bigl(\ddot{l}_{\theta_0 + u (\tilde{\theta} -
\theta
_0)}
(x) - \ddot{l}_{\theta_0} (x) \bigr) \,\d u.\vspace*{-1pt}
\]
Then\vspace*{-1pt}
\[
\dot{l}_{\tilde{\theta}} (x) = \dot{l}_{\theta_0} (x) + \ddot
{l}_{\theta_0} (x) (\tilde{\theta} - \theta_0) + R(x) (\tilde
{\theta} - \theta_0).\vspace*{-1pt}
\]
Therefore,\vspace*{-1pt}
\begin{eqnarray*}
\frac{1}{n} \sum_{i=1}^n
\dot{l}_{\tilde{\theta}} (X_i) = \frac{1}{n} \sum
_{i=1}^n \dot{l}_{\theta_0}
(X_i) + \frac{1}{n} \sum_{i=1}^n
\ddot {l}_{\theta_0} (X_i) (\tilde{\theta} -
\theta_0) + \Biggl(\frac{1}{n} \sum
_{i=1}^n R(X_i) \Biggr) (\tilde{
\theta} - \theta_0),\vspace*{-1pt}
\end{eqnarray*}
and hence\vspace*{-1pt}
\begin{eqnarray*}
\sqrt{n} (\hat{\theta} - \theta_0) &=& \sqrt{n} (\hat{\theta} -
\tilde {\theta}) + \sqrt{n} (\tilde{\theta} - \theta_0)
\\[-1pt]
&=& \sqrt{n} (\hat{\theta} - \tilde{\theta})
+ \frac{({1}/{\sqrt{n}}) \sum_{i=1}^n \dot{l}_{\theta_0} (X_i)
-( {1}/{\sqrt{n}})
\sum_{i=1}^n \dot{l}_{\tilde{\theta}} (X_i) }{-({1}/{n}) \sum_{i=1}^n
\ddot{l}_{\theta_0} (X_i) -( {1}/{n}) \sum_{i=1}^n R(X_i) }.\vspace*{-1pt}
\end{eqnarray*}
Now, with probability tending to $1$ when $n$ goes to infinity,\vspace*{-1pt} 
\begin{eqnarray*}
\Biggl\llvert \frac{1}{n} \sum_{i=1}^n
R(X_i) \Biggr\rrvert &\leq& 2 \E \bigl[\varphi _2(X_1)
\bigr] |\tilde{\theta} - \theta_0|^{\gamma_2}
\\[-1pt]
&\leq& 2 \E \bigl[\varphi_2(X_1) \bigr] \biggl(C
\frac
{t^{(n)}}{\sqrt
{n}} + |\hat{\theta} - \theta_0|
\biggr)^{\gamma_2}.\vspace*{-1pt}
\end{eqnarray*}
Remark now that the term $D_{\F}/n$ involved in Theorem~\ref
{ThmPrincipalDim1} tends to $0$, which shows that $\hat{\theta}$
converges almost surely to $\theta_0$. Therefore, $n^{-1} \sum_{i=1}^n
R(X_i)$ converges to $0$ in probability.
Slutsky's theorem then shows that\vspace*{-1pt}
\[
\frac{({1}/{\sqrt{n}}) \sum_{i=1}^n \dot{l}_{\theta_0} (X_i) -
({1}/{\sqrt{n}}) \sum_{i=1}^n \dot{l}_{\tilde{\theta}} (X_i) }
{-({1}/{n}) \sum_{i=1}^n \ddot{l}_{\theta_0} (X_i) - ({1}/{n}) \sum_{i=1}^n R(X_i) }
\]
converges in distribution to $\mathcal{N}  (0,1/I(\theta_0) )$.
We then reuse Slutsky's theorem to prove the asymptotic efficiency of
$\hat{\theta}$.

Suppose now that there exists $\lambda> 0$ such that $\E[\exp
(\lambda\varphi_2 (X_1)) ]$, $\E[ \exp(\lambda|\dot{l}_{\theta_0}
(X_1)|) ]$ and $\E[\exp(\lambda|\ddot{l}_{\theta_0}(X_1)|) ]$ are
finite. Then the fact that $\P [  (\mathcal{A}_n (\vartheta
) )^c  ] $ and $\P [  (\mathcal{A}_n'
(\vartheta
) )^c  ] $ go to $0$ exponentially fast ensues from the
following result which goes back to Cram\'{e}r.

\begin{lemme}
Let $Y_1,\ldots,Y_n$ be $n$ independent and identically distributed $\R
$-valued random variables satisfying $\E [\exp(\lambda
|Y_1|)
] < \infty$ for some $\lambda> 0$. Then, for all $\vartheta> 0$,
there exists $\sigma> 0$ such that
\[
\P \Biggl[ \Biggl\llvert \frac{1}{n} \sum_{i=1}^n
\bigl(Y_i - \E [Y_i] \bigr) \Biggr\rrvert \geq
\vartheta \Biggr] \leq2 \mathrm{e}^{- \sigma n}.
\]
\end{lemme}

Notice now that one can always replace $\varphi_1$ by $\varphi_1 (x) =
\sup_{\theta\in(m,M)} |\dot{l}_{\theta} (x)|$ and $\gamma_1$ by
$\gamma_1 = 1$ since
\[
\bigl|\log f_{\theta'} (x) - \log f_{\theta} (x)\bigr| \leq \Bigl( \sup
_{\theta
'' \in(m,M)} \bigl|\dot{l}_{\theta''} (x)\bigr| \Bigr) \bigl|
\theta' - \theta\bigr|.
\]
We shall show that there exists $\lambda_1 > 0$ such that $\E[\exp
(\lambda_1 \varphi_1 (X_1))] < \infty$. Since $\varphi_2 (X_1)$ has
also finite exponential moments, the preceding lemma will show that $\P
 [  (\mathcal{A}_n (\vartheta) )^c  ] $ goes to $0$
exponentially fast. By setting $\lambda_0 = \lambda/\max\{2,
2(M-m)^{\gamma_2}\}$,
\begin{eqnarray*}
\E \Bigl[\exp \Bigl[ \lambda_0 \sup_{\theta\in(m,M)} \bigl
\{\bigl|\ddot {l}_{\theta
}(X_1)\bigr| \bigr\} \Bigr] \Bigr] &\leq& \E
\Bigl[\exp \bigl[\lambda_0 \bigl|\ddot {l}_{\theta_0}(X_1)\bigr|
\bigr] \exp \Bigl[\lambda_0 \sup_{\theta\in(m,M)} \bigl|
\ddot{l}_{\theta}(X_1) - \ddot{l}_{\theta_0}(X_1)\bigr|
\Bigr] \Bigr]
\\[-3pt]
&\leq& \E \bigl[\exp \bigl[ (\lambda/2 ) \bigl| \ddot {l}_{\theta
_0}(X_1)\bigr|
\bigr] \exp \bigl[ (\lambda/2 ) \varphi_2 (X_1) \bigr]
\bigr]
\\[-3pt]
&\leq& \E^{1/2} \bigl[\exp\bigl( \lambda\bigl|\ddot{l}_{\theta_0}(X_1)
\bigr)\bigr | \bigr] \E ^{1/2} \bigl[\exp\bigl( \lambda\varphi_2
(X_1)\bigr) \bigr]
\\[-3pt]
&<& \infty.
\end{eqnarray*}
By setting $\lambda_1 = (1/2)\min\{\lambda, \lambda_0/ (M-m)\}$,
\begin{eqnarray*}
\E \Bigl[\exp \Bigl[ \lambda_1 \sup_{\theta\in(m,M)} \bigl|\dot
{l}_{\theta
}(X_1) \bigr| \Bigr] \Bigr] &\leq& \E \Bigl[\exp \bigl[
\lambda_1\bigl |\dot {l}_{\theta_0}(X_1)\bigr | \bigr] \exp
\Bigl[ \lambda_1 \sup_{\theta\in
(m,M)} \bigl| \ddot{l}_{\theta}(X_1)\bigr|
(M-m) \Bigr] \Bigr]
\\[-3pt]
&\leq& \E^{1/2} \bigl[\exp\bigl( \lambda\bigl|\dot{l}_{\theta_0}(X_1)
\bigr|\bigr) \bigr] \E^{1/2} \Bigl[ \exp\Bigl[ \lambda_0 \sup
_{\theta\in(m,M)}\bigl | \ddot {l}_{\theta}(X_1)\bigr| \Bigr]
\Bigr]
\\[-3pt]
&<& \infty,
\end{eqnarray*}
which completes the proof. 

\begin{appendix}\label{app}

\section*{Appendix: Proof of Theorem \texorpdfstring{\protect\ref{54515215415152151514}}{4.1}}\label{SectionAnnexe1} 
This theorem follows from Theorem~\ref{ThmPrincipal} as explained in
Section~\ref{subsectionThmprincipal}.
It remains to prove that its assumptions are fulfilled, that is, that
(\ref{DefintiionBiell}) and (\ref{defdeDelta}) hold.

For this purpose, remark that the different parameters $\underline
{r}_j$, $\overline{r}_j$, $\varrho_j$, $\varrho_j'$, $\varrho_k, \ldots$ that
have been introduced in Algorithm \ref
{algoConstructionDimQuelquonqueAvant} depend on the set $\Theta_i$ and
may vary at each iteration of the until loop.
We need to make explicit this dependency in order to prove rigorously
(\ref{DefintiionBiell}) and (\ref{defdeDelta}). 
Unfortunately, this makes the algorithm more difficult to read.\vadjust{\goodbreak}
\begin{algorithm}
\caption{Rewriting of Algorithm \protect\ref{algoConstructionDimQuelquonqueAvant}}
\begin{algorithmic}[1]
\REQUIRE$\Theta_i = \prod_{j=1}^d [a_j^{(i)}, b_j^{(i)}]$
\STATE Choose $k^{(i)} \in\{1,\ldots,d\}$ such that
\[
\underline{R}_{\Theta_i, k^{(i)}} \bigl(b_{{k^{(i)}}}^{(i)} -
a_{{k^{(i)}}}^{(i)} \bigr)^{\alpha_{k^{(i)}}} = \max
_{1 \leq j \leq d} \underline{R}_{\Theta_i, j} \bigl(b_{j}^{(i)}
- a_{j}^{(i)} \bigr)^{\alpha_j}.
\]
\STATE$\ttheta^{(i,1)} = (a_1^{(i)},\ldots,a_d^{(i)})$
\STATE$\ttheta'^{(i,1)} = (a_1^{(i)},\ldots, a_{k^{(i)}-1}^{(i)},
b_{k^{(i)}}^{(i)}, a_{k^{(i)}+1}^{(i)}, a_d^{(i)})$
\STATE${\varrho_j}^{(i,0)} = \overline{r}_{\Theta_i,j} (\ttheta
^{(i,1)},\ttheta'^{(i,1)})$ and ${\varrho_j'}^{(i,0)} = \overline
{r}_{\Theta
_i,j} (\ttheta'^{(i,1)},\ttheta^{(i,1)})$ for all $j \neq k^{(i)} $
\STATE$\varrho_{k^{(i)}}^{(i,0)} = (b_{k^{(i)}}^{(i)} -
a_{k^{(i)}}^{(i)})/2$ and $\varrho_{k^{(i)}}'^{(i,0)} =
(b_{k^{(i)}}^{(i)} - a_{k^{(i)}}^{(i)})/2$
\FORALL{$\ell\geq1$}
\IF{$T(\ttheta^{(i,\ell)},\ttheta'^{(i,\ell)}) \geq0$}
\STATE$\varrho_{\psi_{k^{(i)}}(1)}^{(i,\ell)} = \overline{r}_{\Theta
_i,\psi
_{k^{(i)}}(1)} (\ttheta^{(i,\ell)},\ttheta'^{(i,\ell)})$
\STATE$\varrho_{\psi_{k^{(i)}} (j)}^{(i,\ell)} = \min(\varrho
_{\psi
_{k^{(i)}} (j)}^{(i,\ell-1)}, \overline{r}_{\Theta_i,\psi_{k^{(i)}}(j)}
(\ttheta^{(i,\ell)},\ttheta'^{(i,\ell)}))$, for all $j \in\{
2,\ldots,
d-1\}$
\STATE$\varrho_{k^{(i)}}^{(i,\ell)} = \min(\varrho
_{k^{(i)}}^{(i,\ell
-1)}, \overline{r}_{\Theta_i,k^{(i)}} (\ttheta^{(i,\ell)},\ttheta
'^{(i,\ell)}))$
\STATE$\mathfrak{J}^{(i,\ell)} =  \{1 \leq j \leq d -1,
\theta
_{\psi_{k^{(i)}}(j)}^{(i,\ell)} + \varrho_{\psi
_{k^{(i)}}(j)}^{(i,\ell
)} < b_{\psi_{k^{(i)}}(j)}^{(i)} \}$
\IF{$\mathfrak{J}^{(i,\ell)} \neq\varnothing$}
\STATE$\mathfrak{j}_{\mathrm{min}}^{(i,\ell)} = \min\mathfrak
{J}^{(i,\ell)}$
\STATE Define $\ttheta^{(i,\ell+1)}$ as
\begin{eqnarray*}
\cases{ \theta_{\psi_{k^{(i)}}(j)}^{(i,\ell+1)} = a_{\psi_{k^{(i)}}(j)}^{(i)},
& \quad $\mbox{for all $j < \mathfrak{j}_{\mathrm{min}}^{(i,\ell)} $},$
\vspace*{2pt}
\cr
\theta_{\psi_{k^{(i)}}(\mathfrak{j}_{\mathrm{min}}^{(i,\ell
)})}^{(i,\ell
+1)} = \theta_{\psi_{k^{(i)}}(\mathfrak{j}_{\mathrm{min}}^{(i,\ell
)})}^{(i,\ell)}
+ \varrho_{\psi_{k^{(i)}}(\mathfrak{j}_{\mathrm{min}}^{(i,\ell
)})}^{(i,\ell
)} ,\vspace*{2pt}
\cr
\theta_{\psi_{k^{(i)}}(j)}^{(i,\ell+1)}
= \theta_{\psi
_{k^{(i)}}(j)}^{(i,\ell)}, &\quad  $\mbox{for all $j > \mathfrak
{j}_{\mathrm
{min}}^{(i,\ell)} $} ,$ \vspace*{2pt}
\cr
\theta_{{k^{(i)}}}^{(i,\ell+1)}
= a_{{k^{(i)}}}^{(i)} }
\end{eqnarray*}
\ELSE
\STATE Define $\ttheta^{(i,\ell+1)} = \ttheta^{(i,\ell)}$
\STATE$\mathfrak{j}^{(i,\ell)}_{\mathrm{min}} = d$
\ENDIF
\ENDIF
\IF{$T(\ttheta^{(i,\ell)},\ttheta'^{(i,\ell)}) \leq0$}
\STATE$\varrho_{\psi_{k^{(i)}}(1)}'^{(i,\ell)} = \overline{r}_{\Theta
_i,\psi
_{k^{(i)}}(1)} (\ttheta'^{(i,\ell)},\ttheta^{(i,\ell)})$
\STATE$\varrho_{\psi_{k^{(i)}} (j)}'^{(i,\ell)} = \min(\varrho
_{\psi
_{k^{(i)}} (j)}'^{(i,\ell-1)}, \overline{r}_{\Theta_i,\psi_{k^{(i)}}(j)}
(\ttheta'^{(i,\ell)},\ttheta^{(i,\ell)}))$, for all $j \in\{
2,\ldots,
d-1\}$
\STATE$\varrho_{k^{(i)}}'^{(i,\ell)} = \min(\varrho
_{k^{(i)}}'^{(i,\ell-1)}, {\underline{r}}_{\Theta_i,k^{(i)}}
(\ttheta
'^{(i,\ell)},\ttheta^{(i,\ell)}))$
\STATE$\mathfrak{J}'^{(i,\ell)} =  \{1 \leq j \leq d -1,
\theta
_{\psi_{k^{(i)}}(j)}'^{(i,\ell)} + \varrho_{\psi
_{k^{(i)}}(j)}'^{(i,\ell
)} < b_{\psi_{k^{(i)}}(j)}^{(i)} \}$
\IF{$\mathfrak{J}'^{(i,\ell)} \neq\varnothing$}
\algstore{coupemonalgodimd}
\end{algorithmic}
\end{algorithm}
%
%
\begin{algorithm} 
%
\begin{algorithmic}[1]
\algrestore{coupemonalgodimd}
\STATE$\mathfrak{j}'^{(i,\ell)}_{\mathrm{min}} = \min\mathfrak
{J}'^{(i,\ell)} $
\STATE Define $\ttheta'^{(i,\ell+1)}$ as
\begin{eqnarray*}
\cases{ \theta_{\psi_{k^{(i)}}(j)}'^{(i,\ell+1)} =
a_{\psi_{k^{(i)}}(j)}^{(i)}, &\quad $\mbox{for all $j < \mathfrak{j}_{\mathrm{min}}'^{(i,\ell)}
$},$ \vspace*{2pt}
\cr
\theta_{\psi_{k^{(i)}}(\mathfrak{j}_{\mathrm{min}}'^{(i,\ell
)})}'^{(i,\ell
+1)} =
\theta_{\psi_{k^{(i)}}(\mathfrak{j}_{\mathrm{min}}'^{(i,\ell
)})}'^{(i,\ell)} + \varrho_{\psi_{k^{(i)}}(\mathfrak{j}_{\mathrm
{min}}'^{(i,\ell)})}'^{(i,\ell)},
\vspace*{2pt}
\cr
\theta_{\psi_{k^{(i)}}(j)}'^{(i,\ell+1)} =
\theta_{\psi
_{k^{(i)}}(j)}'^{(i,\ell)},& \quad$\mbox{for all $j >
\mathfrak {j}_{\mathrm{min}}'^{(i,\ell)} $},$ \vspace*{2pt}
\cr
\theta_{{k^{(i)}}}'^{(i,\ell+1)} = b_{{k^{(i)}}}^{(i)}}
\end{eqnarray*}
%
\ELSE
\STATE$\ttheta'^{(i,\ell+1)} = \ttheta'^{(i,\ell)}$
\STATE$\mathfrak{j}'^{(i,\ell)}_{\mathrm{min}} = d$
\ENDIF
\ENDIF
\IF{ $\mathfrak{j}^{(i,\ell)}_{\mathrm{min}} = d$ or $\mathfrak
{j}'^{(i,\ell)}_{\mathrm{min}} = d$}
\STATE$L_i = \ell$ and quit the loop
\ENDIF
\ENDFOR
\IF{$ \mathfrak{j}^{(i,\ell)}_{\mathrm{min}} = d$}
\STATE$a_{k^{(i)}}^{(i+1)} = a_{k^{(i)}}^{(i)} + \varrho_{k^{(i)}
}^{(i,L_i)} $
\ENDIF
\IF{ $\mathfrak{j}'^{(i,\ell)}_{\mathrm{min}} = d$}
\STATE$b_{k^{(i)}}^{(i+1)} = b_{k^{(i)}}^{(i)} - \varrho_{k^{(i)}
}'^{(i,L_i)} $
\ENDIF
\STATE$a^{(i+1)}_j = a^{(i)}_j$ and $b^{(i+1)}_j = b^{(i)}_j$ for all
$j \neq k^{(i)}$
\RETURN$\Theta_{i+1} = \prod_{j=1}^d [a_j^{(i+1)}, b_j^{(i+1)}]$
\algstore{coupemonalgo3}
\end{algorithmic}
\end{algorithm}\vspace*{-15pt}
%
%
\begin{algorithm}[H]
\caption{Rewriting of Algorithm \protect\ref{algoConstructionDimQuelquonque}}
\label{algoConstructionDimQuelquonque2}
\begin{algorithmic}[1]
\algrestore{coupemonalgo3}
\STATE$\Theta_1 = \prod_{j=1}^d [a_j^{(1)},b_j^{(1)}] = \prod_{j=1}^d
[m_j,M_j]$
\FORALL{$i \geq1$}
\IF{there exists $j \in\{1,\ldots,d\}$ such that $ b_j^{(i)} -
a_j^{(i)} > \eta_j$}
\STATE Compute $\Theta_{i+1}$
\ELSE
\STATE Leave the loop and set ${N} = i$
\ENDIF
\ENDFOR
\RETURN
\[
\tttheta= \biggl(\frac{a_1^{(N)} + b_1^{(N)}}{2}, \ldots, \frac
{a_d^{(N)} + b_d^{(N)}}{2} \biggr)
\]
\end{algorithmic}
\end{algorithm}
We begin by proving that (\ref{defdeDelta}) holds.

\begin{lemmee} \label{ClaimPreuveAlgoDim2Kappa0}
For all $i \in\{1,\ldots,N - 1\}$ and $\ell\in\{1,\ldots,L_i\}$,
\[
\sup_{\ttheta, \ttheta' \in\Theta_i} \delta^2 \bigl(\ttheta,
\ttheta'\bigr) \leq h^2(f_{\ttheta^{(i,\ell)}},f_{\ttheta'^{(i,\ell)}}).
\]
\end{lemmee}

\begin{pf}
Recalling that $\underline{R}_j \leq\underline{R}_{\Theta_i,j}$,
\begin{eqnarray*}
\sup_{\ttheta, \ttheta' \in\Theta_{i}} \delta^2 \bigl(\ttheta,
\ttheta'\bigr) &\leq& \sup_{1 \leq j \leq d}
\underline{R}_{\Theta_i,j} \bigl(b_{j}^{(i)} -
a_{j}^{(i)} \bigr)^{\alpha_j}
\\
&\leq& \underline{R}_{\Theta_i,k^{(i)}} \bigl(b_{{k^{(i)}}}^{(i)} -
a_{{k^{(i)}}}^{(i)} \bigr)^{\alpha_{k^{(i)}}}.
\end{eqnarray*}
Now, $\theta^{(i,\ell)}_{k^{(i)}} = a_{k^{(i)}}^{(i)}$ and $\theta
'^{(i,\ell)}_{k^{(i)}} = b_{k^{(i)}}^{(i)}$, and thus
\begin{eqnarray*}
\sup_{\ttheta, \ttheta' \in\Theta_i} \delta^2 \bigl(\ttheta,\ttheta
'\bigr) &\leq & \underline{R}_{\Theta_i,k^{(i)}} \bigl(
\theta'^{(i,\ell)}_{{k^{(i)}}} - \theta^{(i,\ell)}_{{k^{(i)}}}
\bigr)^{\alpha_{k^{(i)}}}
\\
&\leq& \sup_{1 \leq j \leq d} \underline{R}_{\Theta_i,j} \bigl(\theta
'^{(i,\ell)}_{j} - \theta^{(i,\ell)}_{j}
\bigr)^{\alpha_j}
\\
&\leq& h^2(f_{\ttheta^{(i,\ell)}},f_{\ttheta'^{(i,\ell)}}).
\end{eqnarray*}
\upqed\end{pf}

We now show that (\ref{DefintiionBiell}) holds:

\begin{lemmee} \label{PreuveLemmeAlgoDimQuelc}
For all $i \in\{1,\ldots, N-1\}$,
\begin{eqnarray*}
\Theta_i {}\Big\backslash{}\bigcup_{\ell=1}^{L_i}
B^{(i,\ell)} \subset \Theta _{i+1} \subset\Theta_i.
\end{eqnarray*}
\end{lemmee}

\begin{pf}
Since
\begin{eqnarray*}
\varrho_{k^{(i)} }^{(i,L_i)}&\leq&\frac{b^{(i)}_{k^{(i)} } -
a^{(i)}_{k^{(i)} }}{2} \quad\mbox{and}\\
\varrho_{k^{(i)}
}'^{(i,L_i)} &\leq&\frac{b^{(i)}_{k^{(i)} } - a^{(i)}_{k^{(i)} }}{2},
\end{eqnarray*}
we have $\Theta_{i+1} \subset\Theta_i$. We now aim at proving
$\Theta
_i \setminus\bigcup_{\ell=1}^{L_i} B^{(i,\ell)} \subset\Theta_{i+1}$.

We introduce the rectangles
\begin{eqnarray*}
\mathcal{R}_1'^{(i,\ell)} &=& \prod
_{q=1}^d \bigl[\theta _q^{(i,\ell)},
\theta_q^{(i,\ell)} + \varrho_{q}^{(i,\ell)}
\bigr],
\\
\mathcal{R}_2'^{(i,\ell)} &=& \prod
_{q=1}^{k^{(i)}-1} \bigl[\theta _q'^{(i,\ell)},
\theta_q'^{(i,\ell)} + \varrho_{q}'^{(i,\ell)}
\bigr] \times \bigl[\theta_{k^{(i)}}'^{(i,\ell)} -
\varrho _{k^{(i)}}'^{(i,\ell
)}, \theta_{k^{(i)}}'^{(i,\ell)}
\bigr] \\
&&{}\times\prod_{q=k^{(i)}+1}^{d} \bigl[
\theta_q'^{(i,\ell)}, \theta_q'^{(i,\ell
)}
+ \varrho_{q}'^{(i,\ell)} \bigr]
\end{eqnarray*}
and we set
\begin{eqnarray*}
\mathcal{R}_3'^{(i,\ell)}= \cases{
\mathcal{R}_1'^{(i,\ell)},& \quad$\mbox{if $T \bigl({
\ttheta^{(i,\ell
)}},{\ttheta '^{(i,\ell)}}\bigr) > 0$},$
\vspace*{2pt}
\cr
\mathcal{R}_2'^{(i,\ell)},&\quad $
\mbox{if $T \bigl({\ttheta^{(i,\ell
)}},{\ttheta '^{(i,\ell)}}
\bigr) < 0$},$ \vspace*{2pt}
\cr
\mathcal{R}_1'^{(i,\ell)}
\cup\mathcal{R}_2'^{(i,\ell)}, &\quad $\mbox {if $T
\bigl({\ttheta^{(i,\ell)}},{\ttheta'^{(i,\ell)}}\bigr) =
0$}$.}
\end{eqnarray*}
Using that $\Theta_i \cap\mathcal{R}_1'^{(i,\ell)} \subset\mathcal{R}
(\ttheta^{(i,\ell)}, \ttheta'^{(i,\ell)})$, $\Theta_i \cap
\mathcal
{R}_2'^{(i,\ell)} \subset\mathcal{R} (\ttheta'^{(i,\ell)}, \ttheta
^{(i,\ell)})$ together with (\ref{eqInclusionRC1}) yields $\Theta_i
\cap\mathcal{R}_3'^{(i,\ell)} \subset B^{(i,\ell)}$. It is then
sufficient to show
\[
\Theta_i {}\Big\backslash{}\bigcup_{\ell=1}^{L_i}
\mathcal{R}_3'^{(i,\ell)} \subset
\Theta_{i+1}.
\]
Note that either $T(\ttheta^{(i,L_i)},\ttheta'^{(i,L_i)}) \geq0$ or
$T(\ttheta^{(i,L_i)},\ttheta'^{(i,L_i)}) \leq0$. In what follows, we
assume that $T(\ttheta^{(i,L_i)},\ttheta'^{(i,L_i)}) \geq0$ but the
proof is similar if $T(\ttheta^{(i,L_i)},\ttheta'^{(i,L_i)})$ is non-positive.
Without lost of generality, and for the sake of simplicity, we suppose
that $k^{(i)} = d$ and $\psi_{d} (j) = j$ for all $j \in\{1,\ldots
,d-1\}$.
Let
\[
\mathcal{L} = \bigl\{1 \leq\ell\leq L_i, T\bigl(
\ttheta^{(i,\ell
)},\ttheta'^{(i,\ell)}\bigr) \geq0 \bigr\}
\]
and $\ell_1 < \cdots< \ell_r$ be the elements of $\mathcal{L} $.
It is sufficient to prove that
\renewcommand{\theequation}{\arabic{equation}}
\setcounter{equation}{35}
\begin{equation}
\label{eqInclusiondanspreuveannexe1} \Theta_i {}\Big\backslash{}\bigcup
_{\ell=1}^{L_i} \mathcal{R}_3'^{(i,\ell)}
\subset\prod_{q=1}^{d-1}
\bigl[a_q^{(i)}, b_q^{(i)} \bigr]
\times \bigl[a_d^{(i)} + \varrho_d^{(i,L_i)},
b_d^{(i)} \bigr].
\end{equation}
We shall actually prove
\[
\prod_{q=1}^{d-1} \bigl[a_q^{(i)},
b_q^{(i)} \bigr] \times \bigl[a_d^{(i)},
a_d^{(i)} + \varrho_d^{(i,L_i)} \bigr]
\subset\bigcup_{k=1}^{r}
\mathcal{R}_1'^{(i, \ell_k)},
\]
which, in particular, implies (\ref{eqInclusiondanspreuveannexe1}).
Remark now that for all $k \in\{1,\ldots,r\}$, $\theta_d^{(i,\ell_{k})}
= a_d^{(i)} $, and thus
\[
\mathcal{R}_1'^{(i, \ell_k)} = \prod
_{q=1}^{d-1} \bigl[ \theta _q^{(i,\ell_{k})},
\theta_q^{(i,\ell_{k})} + \varrho_{q}^{(i,\ell_k)}
\bigr] \times \bigl[ a_d^{(i)}, a_d^{(i)}
+ \varrho_{d}^{(i,\ell_k)} \bigr].
\]
By using the fact that the sequence $ (\varrho_{d}^{(i,\ell_k)})_k$ is
non-increasing,
\[
\bigl[a_d^{(i)}, a_d^{(i)} +
\varrho_d^{(i,L_i)} \bigr] \subset \bigcap
_{k=1}^{r} \bigl[ a_d^{(i)},
a_d^{(i)} + \varrho_{d}^{(i,\ell_k)} \bigr].
\]
This means that we only need to show
%
\begin{equation}
\label{eqInclusionPreuveDimenQ} \prod_{q=1}^{d-1}
\bigl[a_q^{(i)}, b_q^{(i)} \bigr]
\subset\bigcup_{k=1}^{r} \prod
_{q=1}^{d-1} \bigl[ \theta_q^{(i,\ell_{k})},
\theta _q^{(i,\ell_{k})} + \varrho_{q}^{(i,\ell_k)}
\bigr].
\end{equation}
Let us now define 
for all $p \in\{1,\ldots,d-1\}$, $k_{p,0} = 0$ and by induction for
all integer $\mathfrak{m}$,
\begin{eqnarray*}
k_{p,\mathfrak{m}+1} = \cases{ \inf \bigl\{k > k_{p,\mathfrak{m}},
\mathfrak{j}_{\mathrm
{min}}^{(i,\ell_{k})} > p \bigr\}, \vspace*{2pt}\cr
\hspace*{26pt}$\mbox{if there exists
$k \in\{k_{p,\mathfrak{m}} +1,\ldots,r \}$ such that $\mathfrak{j}_{\mathrm{min}}^{(i,\ell_{k})}
> p$},$ \vspace *{2pt}
\cr
r,  \qquad $\mbox{otherwise.}$}
\end{eqnarray*}
Let $\mathfrak{M}_p$ be the smallest integer $\mathfrak{m}$ such that
$k_{p,\mathfrak{m}} =
r$. Let then for all $\mathfrak{m}\in\{0,\ldots,\mathfrak{M}_p-1\}$,
\[
K_{p,\mathfrak{m}} = \{k_{p,\mathfrak{m}}+1, \ldots, k_{p,\mathfrak{m}+1} \}.
\]
We need the two following claims.

\begin{Claimm} \label{ClaimInclusionKpDansKp}
For all $\mathfrak{m}\in\{0,\ldots, \mathfrak{M}_{p+1}-1\}$, there
exists $\mathfrak{m}'
\in\{0,\ldots, \mathfrak{M}_{p}-1\}$ such that $k_{p,\mathfrak
{m}'+1} \in
K_{p+1,\mathfrak{m}}$.
\end{Claimm}

\begin{pf}
The set $\{\mathfrak{m}' \in\{0,\ldots, \mathfrak{M}_p-1\},
k_{p,\mathfrak{m}'+1} \leq
k_{p+1,\mathfrak{m}+1} \}$ is non-empty and we can thus define the largest
integer $\mathfrak{m}'$ of $\{0,\ldots, \mathfrak{M}_p-1\}$ such that
$ k_{p,\mathfrak{m}
'+1} \leq k_{p+1,\mathfrak{m}+1} $.
We then have
\[
k_{p,\mathfrak{m}'} = \sup \bigl\{k < k_{p,\mathfrak{m}'+1}, \mathfrak{j}_{\mathrm
{min}}^{(i,\ell_k)}
> p \bigr\}.
\]
Since $k_{p,\mathfrak{m}'} < k_{p+1,\mathfrak{m}+1}$,
\begin{eqnarray*}
k_{p,\mathfrak{m}'} &=& \sup \bigl\{k < k_{p+1,\mathfrak{m}+1}, \mathfrak{j}_{\mathrm
{min}}^{(i,\ell_k)}
> p \bigr\}
\\
&\geq& \sup \bigl\{k < k_{p+1,\mathfrak{m}+1}, \mathfrak {j}_{\mathrm
{min}}^{(i,\ell_k)}
> p + 1 \bigr\}
\\
&\geq& k_{p+1,\mathfrak{m}}.
\end{eqnarray*}
Hence, $k_{p,\mathfrak{m}'+1} \geq k_{p,\mathfrak{m}'} + 1 \geq
k_{p+1,\mathfrak{m}} + 1$. Finally,
$k_{p,\mathfrak{m}'+1} \in K_{p,\mathfrak{m}}$.
\end{pf}


\begin{Claimm} \label{ClaimPreuveAlgoDimen2}
Let $\mathfrak{m}' \in\{0,\ldots,\mathfrak{M}_{p+1}-1\}$, $p \in\{
1,\ldots,d-1\}
$. There exists a subset $\M$ of $\{0,\ldots,\mathfrak{M}_p-1\}$ such that
\[
K'_{p} = \{k_{p,\mathfrak{m}+1}, \mathfrak{m}\in\M \}
\subset K_{p+1,\mathfrak{m}'}
\]
and
\[
\bigl[a_{p+1}^{(i)}, b_{p+1}^{(i)} \bigr]
\subset\bigcup_{ k \in K_p'} \bigl[ \theta_{p+1}^{(i,\ell_{k})},
\theta_{p+1}^{(i,\ell_{k})} + \varrho_{p+1}^{(i,\ell_{k })}
\bigr].
\]
\end{Claimm}

\begin{pf}
Thanks to Claim~\ref{ClaimInclusionKpDansKp}, we can define the
smallest integer $\mathfrak{m}_0$ of $\{0,\ldots,\mathfrak{M}_{p}-1\}
$ such that
$k_{p,\mathfrak{m}_0 + 1} \in K_{p+1,\mathfrak{m}'}$, and the largest
integer $\mathfrak{m}_1$ of $\{
0,\ldots,\mathfrak{M}_{p}-1\}$ such that $k_{p, \mathfrak{m}_1 + 1}
\in K_{p+1,\mathfrak{m}
'}$. Define now
\[
\M= \{\mathfrak{m}_0,\mathfrak{m}_0+1,\ldots,
\mathfrak {m}_1 \}.
\]
Note that for all $\mathfrak{m}\in\{\mathfrak{m}_0,\ldots,\mathfrak
{m}_1\}$, $k_{p,\mathfrak{m}+1} \in
K_{p+1,\mathfrak{m}'}$ (this ensues from the fact that the sequence
$(k_{p,\mathfrak{m}
})_{\mathfrak{m}}$ is increasing).

Let $\mathfrak{m}\in\{0,\ldots,\mathfrak{M}_p-1\}$ be such that
$k_{p,\mathfrak{m}} \in
K_{p+1,\mathfrak{m}'}$ and $k_{p,\mathfrak{m}} \neq k_{p+1,\mathfrak
{m}'+1}$. Then $ \mathfrak
{j}_{\mathrm{min}}^{(i,\ell_{k_{p,\mathfrak{m}}})} \leq p + 1$ and
since $ \mathfrak
{j}_{\mathrm{min}}^{(i,\ell_{k_{p,\mathfrak{m}}})} > p $, we get
$\mathfrak
{j}_{\mathrm{min}}^{(i,\ell_{k_{p,\mathfrak{m}}})} = p + 1$.
Consequently,
\[
\theta_{p+1}^{  (i,\ell_{k_{p,\mathfrak{m}}+1} )} = \theta_{p+1}^{
(i,\ell_{k_{p,\mathfrak{m}}} )} +
\varrho_{p+1}^{
(i,\ell_{k_{p,\mathfrak{m}
}} )}.
\]
Now, $\theta_{p+1}^{  (i,\ell_{k_{p,\mathfrak{m}}+1} )} =
\theta_{p+1}^{
(i,\ell_{k_{p,\mathfrak{m}+1}})}$ since $k_{p,\mathfrak{m}}+1$ and
$k_{p,\mathfrak{m}+1}$ belong
together to $K_{p,\mathfrak{m}}$. The set
\[
\bigl[ \theta_{p+1}^{(i,\ell_{k_{p,\mathfrak{m}}})}, \theta _{p+1}^{(i,\ell
_{k_{p,\mathfrak{m}}})}
+ \varrho_{p+1}^{(i,\ell_{k_{p,\mathfrak
{m}} })} \bigr] \cup \bigl[
\theta_{p+1}^{(i,\ell_{k_{p,\mathfrak{m}+1}})}, \theta _{p+1}^{(i,\ell
_{k_{p,\mathfrak{m}+1}})} +
\varrho_{p+1}^{(i,\ell_{k_{p,\mathfrak
{m}+1} })} \bigr]
\]
is thus the interval
\[
\bigl[ \theta_{p+1}^{(i,\ell_{k_{p,\mathfrak{m}}})}, \theta _{p+1}^{(i,\ell
_{k_{p,\mathfrak{m}+1}})}
+ \varrho_{p+1}^{(i,\ell_{k_{p,\mathfrak
{m}+1} })} \bigr].
\]
We apply this argument to each $\mathfrak{m}\in\{\mathfrak
{m}_0+1,\ldots,\mathfrak{m}_1\}$ to derive
that the set
\[
I = \bigcup_{\mathfrak{m}=\mathfrak{m}_0}^{\mathfrak{m}_1} \bigl[
\theta_{p+1}^{(i,\ell_{k_{p,\mathfrak{m}
+1}})}, \theta_{p+1}^{(i,\ell_{k_{p,\mathfrak{m}+1}})} +
\varrho _{p+1}^{(i,\ell
_{k_{p,\mathfrak{m}+1} })} \bigr]
\]
is the interval
\[
I = \bigl[ \theta_{p+1}^{(i,\ell_{k_{p,\mathfrak{m}_0+1}})}, \theta _{p+1}^{(i,\ell_{k_{p,\mathfrak{m}_1+1}})}
+ \varrho_{p+1}^{(i,\ell
_{k_{p,\mathfrak{m}_1+1}
})} \bigr].
\]
The claim is proved if we show that
\[
\bigl[a_{p+1}^{(i)}, b_{p+1}^{(i)} \bigr]
\subset I.
\]
Since $I$ is an interval, it remains to prove that $a_{p+1}^{(i)} \in
I$ and $b_{p+1}^{(i)} \in I$.

We begin to show $a_{p+1}^{(i)} \in I$ by showing that $ a_{p+1}^{(i)}
= \theta_{p+1}^{(i,\ell_{k_{p,\mathfrak{m}_0+1}})}$.
If $k_{p+1,\mathfrak{m}'} = 0$, then $\mathfrak{m}' = 0$ and
$\mathfrak{m}_0 = 0$. Besides, since $1$
and $k_{p,1}$ belong to $K_{p,0}$, we have $\theta_{p+1}^{(i,\ell
_{k_{p,1}})} = \theta_{p+1}^{(i,\ell_{1})}$. Now, $\theta
_{p+1}^{(i,\ell
_{1})} = a_{p+1}^{(i)}$ and thus $a_{p+1}^{(i)} \in I$. We now assume
that $k_{p+1,\mathfrak{m}'} \neq0$.
Since $k_{p,\mathfrak{m}_0} \leq k_{p+1,\mathfrak{m}'}$, there are
two cases.
\begin{itemize}
\item First case: $k_{p,\mathfrak{m}_0} = k_{p+1,\mathfrak{m}'}$. We
then have $\mathfrak
{j}_{\mathrm{min}}^{(i,\ell_{k_{p,\mathfrak{m}_0}})} > p + 1$ and thus
$\theta
_{p+1}^{(i,\ell_{k_{p,\mathfrak{m}_0}+1})}= a_{p+1}^{(i)}$. Since
$k_{p,\mathfrak{m}_0+1}$
and $k_{p,\mathfrak{m}_0}+1$ belong to $K_{p,\mathfrak{m}_0}$,
$\theta_{p+1}^{(i,\ell_{k_{p,\mathfrak{m}_0+1}})} = \theta
_{p+1}^{(i,\ell_{k_{p,\mathfrak{m}
_0}+1})}$ and thus $\theta_{p+1}^{(i,\ell_{k_{p,\mathfrak{m}_0+1}})} =
a_{p+1}^{(i)}$ as wished.
\item Second case: $k_{p,\mathfrak{m}_0} + 1 \leq k_{p+1,\mathfrak
{m}'}$. Then $k_{p+1,\mathfrak{m}'}
\in K_{p,\mathfrak{m}_0}$, and thus
\[
\theta_{p+1}^{(i,\ell_{k_{p,\mathfrak{m}_0}+1})} = \theta _{p+1}^{(i,\ell
_{k_{p+1,\mathfrak{m}'}})}.
\]
Since $\mathfrak{j}_{\mathrm{min}}^{(i,\ell_{k_{p+1,\mathfrak{m}'}})}
> p + 1$, we
have $\theta_{p+1}^{(i,\ell_{k_{p+1,\mathfrak{m}'}} )} + \varrho
_{p+1}^{(i,\ell
_{k_{p+1,\mathfrak{m}'}} )} \geq b_{p+1}^{(i)}$. By using the fact
that the
sequence $( \varrho_{p+1}^{(i,\ell_{k})})_k$ is decreasing, we then deduce
\[
\theta_{p+1}^{(i,\ell_{k_{p,\mathfrak{m}_0} + 1})} + \varrho _{p+1}^{(i,\ell
_{k_{p,\mathfrak{m}_0} + 1} )}
\geq b_{p+1}^{(i)}
\]
and thus
$\mathfrak{j}_{\mathrm{min}}^{(i,\ell_{k_{p,\mathfrak{m}_0} + 1})} >
p + 1$. This
proves that
%
\begin{eqnarray}
\label{EqPreuveThetap} \theta_{p+1}^{(i,\ell_{k_{p,\mathfrak{m}_0} + 2})}= a_{p+1}^{(i)}.
\end{eqnarray}
Let us now show that $k_{p,\mathfrak{m}_0} + 2 \leq k_{p,\mathfrak
{m}_0+1}$. If this is not
true, $k_{p,\mathfrak{m}_0} + 2 \geq k_{p,\mathfrak{m}_0+1} + 1$,
and thus $k_{p,\mathfrak{m}_0} +
1 \geq k_{p,\mathfrak{m}_0+1}$ which means that $k_{p,\mathfrak{m}_0}
+ 1 = k_{p,\mathfrak{m}_0+1}$
(we recall that $(k_{p,\mathfrak{m}})_{\mathfrak{m}}$ is an
increasing sequence of
integers). Since we are in the case where $ k_{p,\mathfrak{m}_0}+1
\leq k_{p+1,\mathfrak{m}
'}$, we have $ k_{p,\mathfrak{m}_0+1} \leq k_{p+1,\mathfrak{m}'}$
which is impossible since
$ k_{p,\mathfrak{m}_0+1} \in K_{p+1,\mathfrak{m}'}$.

Therefore, we use that $k_{p,\mathfrak{m}_0} + 2 \leq k_{p,\mathfrak
{m}_0+1}$ to get
$k_{p,\mathfrak{m}_0} + 2 \in K_{p,\mathfrak{m}_0}$, and thus $\theta
_{p+1}^{(i,\ell
_{k_{p,\mathfrak{m}_0+1} })}= \theta_{p+1}^{(i,\ell_{k_{p,\mathfrak
{m}_0}+2 })}$. We then
deduce from (\ref{EqPreuveThetap}) that $\theta_{p+1}^{(i,\ell
_{k_{p,\mathfrak{m}
_0+1} })} = a_{p+1}^{(i)}$ as wished.
\end{itemize}
We now show that $b_{p+1}^{(i)} \in I$ by showing that $\theta
_{p+1}^{(i,\ell_{k_{p,\mathfrak{m}_1+1}})} + \varrho_{p+1}^{(i,\ell
_{k_{p,\mathfrak{m}_1+1}
})} \geq b_{p+1}^{(i)}$. If $\mathfrak{m}_1 = \mathfrak{M}_p-1$,
\[
\theta_{p+1}^{(i,\ell_{k_{p,\mathfrak{m}_1+1}})} + \varrho _{p+1}^{(i,\ell_{k_{p,\mathfrak{m}
_1+1} })}
= \theta_{p+1}^{(i,\ell_{r})} + \varrho_{p+1}^{(i,\ell_{r })}
= \theta_{p+1}^{(i,L_i)} + \varrho_{p+1}^{(i,L_{i })}.
\]
Since $\mathfrak{J}^{(i, L_i)} = \varnothing$, we have $\theta
_{p+1}^{(i,L_i)} + \varrho_{p+1}^{(i,L_{i })} \geq b_{p+1}^{(i)}$,
which proves the result.

We now assume that $\mathfrak{m}_1 < \mathfrak{M}_p-1$.
We begin to prove that $k_{p,\mathfrak{m}_1+1} = k_{p+1,\mathfrak
{m}'+1}$. If this equality
does not hold, we derive from the inequalities $k_{p,\mathfrak{m}_1+1}
\leq
k_{p+1,\mathfrak{m}'+1} < k_{p,\mathfrak{m}_1+2}$, that
$k_{p,\mathfrak{m}_1+1} + 1 \leq k_{p+1,\mathfrak{m}
'+1}$ and thus $k_{p+1,\mathfrak{m}'+1} \in K_{p,\mathfrak{m}_1+1}$.
Since $\mathfrak
{j}_{\mathrm{min}}^{(i, \ell_{k_{p+1,\mathfrak{m}'+1}})} > p +1$,
\[
\theta_{p+1}^{(i,\ell_{k_{p+1,\mathfrak{m}'+1}} )} + \varrho _{p+1}^{(i,\ell
_{k_{p+1,\mathfrak{m}'+1}} )}
\geq b_{p+1}^{(i)}.
\]
Hence,
\[
\theta_{p+1}^{(i,\ell_{(k_{p,\mathfrak{m}_1+1})+1 })} + \varrho _{p+1}^{(i,\ell
_{(k_{p,\mathfrak{m}_1+1})+1 } )}
\geq b_{p+1}^{(i)}\qquad \mbox {which implies } \mathfrak{j}_{\mathrm{min}}^{(i,\ell_{(k_{p,\mathfrak
{m}_1+1})+1 })}
> p + 1.
\]
Since
\[
k_{p+1,\mathfrak{m}'+1} = \inf \bigl\{k > k_{p+1,\mathfrak{m}'}, \mathfrak{j}_{\mathrm
{min}}^{(i,\ell_{k})}
> p + 1 \bigr\}
\]
and $k_{p,\mathfrak{m}_1+1} + 1 > k_{p+1,\mathfrak{m}'}$, we have $
k_{p+1,\mathfrak{m}'+1} \leq
k_{p,\mathfrak{m}_1+1} + 1 $.
Moreover, since $ k_{p+1,\mathfrak{m}'+1} \geq k_{p,\mathfrak{m}_1+1}
+ 1 $, we have
$k_{p,\mathfrak{m}_1+1} + 1 = k_{p+1,\mathfrak{m}'+1}$. Consequently,
\[
k_{p,\mathfrak{m}_1+2} = \inf \bigl\{k > k_{p,\mathfrak{m}_1+1}, \mathfrak{j}_{\mathrm
{min}}^{(i,\ell_{k})}
> p \bigr\} = k_{p+1,\mathfrak{m}'+1}.
\]
This is impossible because $ k_{p+1,\mathfrak{m}'+1} < k_{p,\mathfrak
{m}_1+2}$, which
finally implies that $k_{p,\mathfrak{m}_1+1} = k_{p+1,\mathfrak{m}'+1}$.

We then deduce from this equality,
\[
\mathfrak{j}_{\mathrm{min}}^{(i,\ell_{k_{p,\mathfrak{m}_1+1}})} = \mathfrak {j}_{\mathrm{min}}^{(i,\ell_{k_{p+1,\mathfrak{m}'+1}})}
> p + 1.
\]
Hence, $ \theta_{p+1}^{(i, \ell_{{k_{p,\mathfrak{m}_1+1} }} )} +
\varrho
_{p+1}^{(i, \ell_{{k_{p,\mathfrak{m}_1+1} }} )} \geq b_{p+1}^{(i)} $
and thus $b_{p+1}^{(i)} \in I$. This completes the proof.
\end{pf}

We now return to the proof of Lemma~\ref{PreuveLemmeAlgoDimQuelc} and
prove by induction on $p$ the following result. For all $p \in\{
1,\ldots
,d-1\}$ and all $\mathfrak{m}\in\{0,\ldots,\mathfrak{M}_p-1\}$,
%
\begin{equation}
\label{EqInclusionDansPreuve} \prod_{q=1}^{p}
\bigl[a_q^{(i)}, b_q^{(i)} \bigr]
\subset\bigcup_{k\in
K_{p,\mathfrak{m}}} \prod
_{q=1}^p \bigl[ \theta_q^{(i,\ell_{k})},
\theta _q^{(i,\ell_{k})} + \varrho_{q}^{(i,\ell_k)}
\bigr].
\end{equation}
Note that {(\ref{eqInclusionPreuveDimenQ})} follows from this inclusion
when $p = d -1$ and $\mathfrak{m}= 0$.

We begin to prove (\ref{EqInclusionDansPreuve}) for $p = 1$ and all
$\mathfrak{m}
\in\{0,\ldots,\mathfrak{M}_1-1\}$. For all $k \in\{k_{1,\mathfrak
{m}} +1, \ldots,
k_{1,\mathfrak{m}+1} - 1\}$, $\mathfrak{j}_{\mathrm{min}}^{(i,\ell
_{k})} \leq1$,
and thus
\[
\theta_1^{(i,\ell_{k+1})} \in \bigl\{ \theta_1^{(i,\ell_{k})},
\theta _1^{(i,\ell_{k})} + \varrho_{1}^{(i,\ell_{k})}
\bigr\}.
\]
This implies that the set
\[
\bigcup_{k=k_{1,\mathfrak{m}}+1}^{k_{1,\mathfrak{m}+1}} \bigl[
\theta_1^{(i,\ell_{k})}, \theta_1^{(i,\ell_{k})} +
\varrho_{1}^{(i,\ell_{k})} \bigr]
\]
is an interval. Now,
$\theta_1^{(i,\ell_{k_{1,\mathfrak{m}}+1})} = a_{1}^{(i)}$, $
\theta_1^{(i,\ell
_{k_1,\mathfrak{m}+1})} + \varrho_{1}^{(i,\ell_{k_1,\mathfrak
{m}+1})}\geq b_1^{(i)}$ since
$\mathfrak{j}_{\mathrm{min}}^{(i,\ell_{k_1,\mathfrak{m}+1})} > 1$.
Therefore,
\[
\bigl[a_{1}^{(i)}, b_1^{(i)} \bigr]
\subset\bigcup_{k=k_{1,\mathfrak{m}
}+1}^{k_{1,\mathfrak{m}+1}} \bigl[
\theta_1^{(i,\ell_{k})}, \theta _1^{(i,\ell
_{k})} +
\varrho_{1}^{(i,\ell_{k})} \bigr],
\]
which establishes (\ref{EqInclusionDansPreuve}) when $p = 1$.

Let now $p \in\{1,\ldots,d-2\}$ and assume that for all $\mathfrak
{m}\in\{
0,\ldots, \mathfrak{M}_p-1\}$,
\begin{eqnarray*}
\prod_{q=1}^{p} \bigl[a_q^{(i)},
b_q^{(i)} \bigr] \subset\bigcup
_{k\in
K_{p,\mathfrak{m}}} \prod_{q=1}^p
\bigl[ \theta_q^{(i,\ell
_{k})},\theta_q^{(i,\ell
_{k})}
+ \varrho_{q}^{(i,\ell_k)} \bigr].
\end{eqnarray*}
Let $\mathfrak{m}' \in\{0,\ldots, \mathfrak{M}_{p+1}-1\}$. We shall
show that
\begin{eqnarray*}
\prod_{q=1}^{p+1} \bigl[a_q^{(i)},
b_q^{(i)} \bigr] \subset\bigcup
_{k
\in K_{p+1,\mathfrak{m}'}} \prod_{q=1}^{p+1}
\bigl[ \theta _q^{(i,\ell_{k})}, \theta_q^{(i,\ell_{k})}
+ \varrho_{q}^{(i,\ell_k)} \bigr].
\end{eqnarray*}
Let $\mathbf{x} \in\prod_{q=1}^{p+1}  [a_q^{(i)},
b_q^{(i)} ]$.
By using Claim~\ref{ClaimPreuveAlgoDimen2}, there exists $\mathfrak
{m}\in\{
0,\ldots,\mathfrak{M}_{p}-1\}$ such that
\[
x_{p+1} \in \bigl[ \theta_{p+1}^{(i,\ell_{k_{p,\mathfrak{m}+1}})}, \theta
_{p+1}^{(i,\ell_{k_{p,\mathfrak{m}+1}})} + \varrho_{p+1}^{(i,\ell
_{k_{p,\mathfrak{m}+1}})} \bigr]
\]
and such that $k_{p,\mathfrak{m}+1} \in K_{p+1,\mathfrak{m}'}$.
By using the induction assumption, there exists $k \in K_{p,\mathfrak
{m}}$ such that
\[
\mathbf{x} = (x_1,\ldots,x_p) \in\prod
_{q=1}^p \bigl[ \theta _q^{(i,\ell_{k})},
\theta_q^{(i,\ell_{k})} + \varrho_{q}^{(i,\ell_k)}
\bigr].
\]
Since $k \in K_{p,\mathfrak{m}}$, $\theta_{p+1}^{(i,\ell_{k})} =
\theta
_{p+1}^{(i,\ell_{k_{p,\mathfrak{m}+1}})} $ and
$ \varrho_{p+1}^{(i,\ell_{k_{p,\mathfrak{m}+1}})} \leq\varrho
_{p+1}^{(i,\ell
_{k})}$. Hence,
\[
x_{p+1} \in \bigl[ \theta_{p+1}^{(i,\ell_{k})}, \theta
_{p+1}^{(i,\ell
_{k})} + \varrho_{p+1}^{(i,\ell_{k})} \bigr].
\]
We finally use the claim below to show that $k \in K_{p+1,\mathfrak
{m}'}$ which
concludes the proof.
\end{pf}

\begin{Claimm} \label{ClaimPreuveAlgoDimen2deuxieme}
Let $\mathfrak{m}\in\{0,\ldots,\mathfrak{M}_p-1\}$ and $\mathfrak
{m}' \in\{0,\ldots,
\mathfrak{M}_{p+1}-1\}$. If $k_{p,\mathfrak{m}+1} \in
K_{p+1,\mathfrak{m}'}$, then $K_{p,\mathfrak{m}
} \subset K_{p+1,\mathfrak{m}'}$.
\end{Claimm}

\begin{pf}
We have
\[
k_{p+1,\mathfrak{m}'} = \sup \bigl\{ k < k_{p+1,\mathfrak{m}'+1}, \mathfrak{j}_{\mathrm
{min}}^{(i,\ell_{k})}
> p + 1 \bigr\}.
\]
Since $k_{p,\mathfrak{m}+1} > k_{p+1,\mathfrak{m}'}$,
\begin{eqnarray*}
k_{p+1,\mathfrak{m}'} &=& \sup \bigl\{ k < k_{p,\mathfrak{m}+1}, \mathfrak{j}_{\mathrm
{min}}^{(i,\ell_{k})}
> p + 1 \bigr\}
\\
&\leq& \sup \bigl\{ k < k_{p,\mathfrak{m}+1}, \mathfrak {j}_{\mathrm
{min}}^{(i,\ell_{k})}
> p \bigr\}
\\
&\leq& k_{p,\mathfrak{m}}.
\end{eqnarray*}
We then derive from the inequalities $k_{p+1,\mathfrak{m}'} \leq
k_{p,\mathfrak{m}}$ and
$k_{p,\mathfrak{m}+1} \leq k_{p+1
, \mathfrak{m}'+1}$ that $K_{p,\mathfrak{m}} \subset K_{p+1,\mathfrak{m}'}$.
\end{pf}
\end{appendix}

\section*{Acknowledgements} The author acknowledges the support of the
French Agence Nationale de la Recherche (ANR), under grant Calibration
(ANR 2011 BS01 010 01). We are thankful to Yannick Baraud for his
numerous comments and his valuable suggestions.



\printhistory
\end{document}